\documentclass[a4paper, 11pt]{amsart}
\pdfoutput=1

\usepackage[a4paper, margin=1in]{geometry}

\usepackage[utf8]{inputenc}

\usepackage[english]{babel} 

\usepackage{import}
\usepackage{xifthen}
\usepackage{pdfpages}
\usepackage{transparent}
\usepackage{wrapfig}

\usepackage{graphicx}
\usepackage{subcaption}
\usepackage{amsmath}
\usepackage{amsfonts}

\usepackage[
style=apa,
backend=biber,
alldates=year,
sortcites=false,
]{biblatex}
\usepackage[mode=build]{standalone}
\usepackage{tikz}
\usetikzlibrary{cd}
\usetikzlibrary{calc}
\usetikzlibrary{decorations.markings}
\tikzstyle{state}=[
rectangle,
minimum size =1.25cm,
thick,
align=center
]
\usepackage{scalefnt}
\usepackage{hyperref}
\usepackage{cleveref}

\usepackage{xcolor, soulutf8}
\definecolor{bubbles}{rgb}{0.91, 1.0, 1.0}
\definecolor{aquamarine}{rgb}{0.5, 1.0, 0.83}
\definecolor{bubblegum}{rgb}{0.99, 0.76, 0.8}
\definecolor{bluebell}{rgb}{0.64, 0.64, 0.82}
\definecolor{dollarbill}{rgb}{0.72, 0.93, 0.6}

\usepackage{makecell}

\usepackage{amssymb}  

\usepackage{amsthm}
\newtheorem{theorem}{Theorem}
\newtheorem{conjecture}{Conjecture}
\newtheorem{proposition}{Proposition}
\newtheorem{definition}{Definition}
\newtheorem{remark}[theorem]{Remark}
\newtheorem{example}[theorem]{Example}
\newtheorem{lemma}[theorem]{Lemma}

\newcommand{\figref}[1]{Fig.~\ref{#1}}
\newcommand{\tabref}[1]{Table~\ref{#1}}

\newcommand{\conjref}[1]{Conjecture~\ref{#1}}
\newcommand{\propref}[1]{Proposition~\ref{#1}}
\newcommand{\lemref}[1]{Lemma~\ref{#1}}
\newcommand{\defref}[1]{Definition~\ref{#1}}
\newcommand{\secref}[1]{Section~\ref{#1}}

\newcommand{\Z}{\mathbb Z}
\newcommand{\N}{\mathbb N}
\newcommand{\Q}{\mathbb Q}

\newcommand{\C}{\mathbb C}
\newcommand{\bbP}{\mathbb P}

\newcommand{\im}{ \operatorname{im}}
\newcommand{\res}{\operatorname{Res}}

\newcommand{\Prim}{\operatorname{Prim}}

\newcommand{\Lop}{\mathcal{L}}
\newcommand{\Tr}{\operatorname{Tr}}
\newcommand{\rk}{\operatorname{rk}}
\newcommand{\sindex}[1]{{#1}}

\newcommand{\Hpara}{H^\mathrm{para}}

\newcommand{\ud}{\mathrm{d}}
\def\st{\ \middle|\ }

\usepackage{algorithm}
\usepackage[noend]{algpseudocode}

\usepackage{booktabs}
\usepackage{longtable}
\usepackage{multirow}

\usepackage{supertabular}
\usepackage{cuted}
\usepackage{changepage}

\makeatletter 
\def\ps@pprintTitle{%
  \let\@oddhead\@empty
  \let\@evenhead\@empty
  \def\@oddfoot{\reset@font\hfil\thepage\hfil}
  \let\@evenfoot\@oddfoot
}
\makeatother

\usepackage{pseudo}
\pseudoset{
  compact, kw,
  label=\scriptsize\sffamily\color{gray}\arabic*,
  font=\sffamily,
  ct-left=\quad\ctfont{[},
  ct-right=\ctfont{]},
}

\title[Periods of fibre products of elliptic surfaces and the Gamma conjecture]{Periods of fibre products of elliptic surfaces\\ and the Gamma conjecture}
     
\author{Eric Pichon-Pharabod }
\address{Max Planck Institute for Mathematics in the Sciences, Inselstr.\ 22, 04103 Leipzig, Germany}
\email{eric.pichon@mis.mpg.de}
\date{\today}

\thanks{This work has been supported by the Agence nationale de la recherche
  (ANR), grant agreement ANR-19-CE40-0018 (De Rerum Natura), grant agreement
ANR-20-CE40-0026-01 (Symmetries and moduli spaces in algebraic geometry and physics); and the European
  Research Council (ERC) under the European Union’s Horizon Europe research and
  innovation programme, grant agreement 101040794 (10000~DIGITS)}


\begin{document}

\begin{abstract}
We provide an algorithm for computing a basis of homology of fibre products of elliptic surfaces over $\bbP^1$, along with the corresponding intersection product and period matrices.
We use this data to investigate the Gamma conjecture for Calabi--Yau threefolds obtained in this manner.
We find a formula that works for all operators of a list of 105 fibre products, as well as fourth order operators of the Calabi--Yau database.
  This algorithm comes with a SageMath implementation.
\end{abstract}

\maketitle


\setcounter{tocdepth}{1}
\tableofcontents

\section{Introduction}

Let $S_1\to \bbP^1$ and $S_2\to \bbP^1$ be two relatively minimal elliptic surfaces with section.
We are interested in the fibre product $S_1\times_{\bbP^1}S_2$.
When the critical locus of $S_1$ and $S_2$ are disjoint, this defines a smooth threefold. 
When the two surfaces are additionally relatively minimal rational elliptic surfaces, \textcite{Schoen_1988} showed that the fibre product defines a smooth \emph{Calabi--Yau} threefold. 
Furthermore, authorising certain types of singular fibres to coincide, it is shown in \textcite{KapustkaKapustka_2009} that the possibly singular threefold admits a small or crepant resolution into a Calabi--Yau threefold.

This construction offers a class of Calabi--Yau threefolds that are tractable enough to investigate in ample details, yet varied enough to show a wide panel of phenomena --- as an example, \textcite{Schoen_1988} showed that this construction yields Calabi--Yau threefolds with Euler characteristic $k$ for any $0\le k\le 100$ apart from a handful.
As such they are of particular interest in string theory and mirror symmetry. 

One main focus of this paper will be one-parameter families of Calabi--Yau varieties carrying motives of type $(1,1,1,1)$.
An example of a fibre product of elliptic surfaces carrying such a motive was studied in \textcite{GolyshevVanStraten_2023} where the authors use the fibre product structure to recover arithmetic and geometric properties the motive.
The relative holomorphic periods of these families are subject to a linear differential equation, the Picard--Fuchs equation.
This Picard--Fuchs equation turns out to be a \emph{Calabi--Yau operator} of degree 4, i.e., part of a long list of operators with a maximal unipotent monodromy point which exhibit interesting geometric and arithmetic properties stemming from mirror symmetry.
In \textcite{AlmkvistEtAl_2005}, the authors gave a list of such Calabi--Yau operators obtained by a computer-aided search, which was then named the AESZ list.
The AESZ list has since been gradually extended 
to include now more than 600 fourth order operators, gathered in the Calabi-Yau database (CYDB)\footnote{\url{https://cydb.mathematik.uni-mainz.de/}}.
An updated version of the same database is under development.

When a Calabi--Yau operator is the Picard--Fuchs equation of a one-parameter family of Calabi--Yau varieties, we call this family a \emph{geometric realisation} of the operator.
It is conjectured that all Calabi--Yau operators have geometric realisations.
In particular, their monodromy representation is expected to describe the integral variation of the homology of the fibre of this family, and yields a discrete invariant that allows to distinguish and classify these operators.
When the periods of the geometric model are known with certified bounds of precision, one may compute generators of the monodromy group certifiably.
For many Calabi--Yau operators, however, a geometric realisation is not known.
To circumvent this, periods are conjectured to be given by a \emph{Gamma-class} formula \parencite{CandelasEtAl_1991a, Libgober_1999, KatzarkovEtAl_2008, Iritani_2009, HalversonEtAl_2015, CandelasEtAl_2020}, which relates the Frobenius basis at the MUM point to the integral basis of periods via topological invariants of the mirror Calabi--Yau of the (hypothetical) geometric realisation.
This can in practice be used to heuristically recover period matrices (see e.g. \textcite{KnappMcGovern_2025, Elmi_2024}).
However the formula in its current state is known to fail for certain Calabi--Yau operators.
One of the results of this paper is way to compute the relation between the integral periods and the Frobenius basis explicitly for fibre products, which provides insight into a more general Gamma-class formula.

Our goal for this paper is to provide tools to study the motives appearing in fibre products of elliptic surfaces.
We will consider threefolds of this type that are singular. 
We will not attempt to resolve these singularities --- it is not known whether every motive can be realised by a family of generically smooth varieties, as seen for example for the 14th hypergeometric case \parencite{ClingherEtAl_2016}.
Instead we will consider \emph{smoothings} of these elliptic surface, by formally splitting off the colliding singular fibres of the threefold.

\subsection*{Contributions}
We provide an algorithm for computing the full homology lattice with its intersection product of a smoothing $\tilde T$ of a fibre product $T = S_1\times_{\bbP^1} S_2$ of relatively minimal elliptic surfaces $S_1$, $S_2$ with section.
We also provide
\begin{itemize}
\item The embedding of the parabolic homology lattice $\Hpara_3(T)$ in $H_3(\tilde T)$.
\item The embedding of lattice of vanishing cycles $\Lambda_{\rm vc}$ in $H_3(\tilde T)$.
\item The period vectors $\pi(\omega)\in \C^{r}$ on $\Lambda_{\rm vc}^\perp$ (with $r=\rk\Lambda_{\rm vc}^\perp$) of cohomology forms $\omega\in H^3(\tilde T)$ of the form $\omega = \omega_t\wedge \ud t$. The vectors are given numerically with certified bounds of precision, in quasilinear time with respect to precision.
\end{itemize}
We then use this algorithm to investigate the Gamma-class formula for \emph{Hadamard products} of elliptic surfaces, i.e., one parameter families of such threefolds.
We consider 105 such families and find a general shape for the Gamma-class formula that fits all of them numerically to very high precision:
\begin{equation*}
\begin{pmatrix} 
    \chi\lambda - \frac{\alpha}{2} \frac{c_2\cdot H}{24} -\frac{\delta}{2} \, &\, M\frac{c_2\cdot H}{24} \,& \,\frac{\alpha}{2} \frac{H^3}{2!} \,& \,M\frac{H^3}{3!}\\[6pt]
    \frac{c_2\cdot H}{24} & N\frac{\sigma}{2} & -\frac{H^3}{2!} & 0\\[6pt]
    1 & 0 & 0 & 0\\[6pt]
    \frac{\alpha}{2} \frac{N}{M} & N & 0 & 0
    \end{pmatrix}\,.
\end{equation*}
Here $\lambda = \zeta(3)/(2\pi i)^3$ and all the other variables are integers.
We provide the matching values for the Hadamard products in \tabref{tab:gamma_class_invariants}, as well as for the examples of Calabi--Yau database with integral monodromy in \tabref{tab:cydb_gamma_class}.
The algorithm presented in this paper is implemented in the \texttt{lefschetz-family}\footnote{\url{https://github.com/ericpipha/lefschetz-family}} package in SageMath \parencite{sagemath}.

\subsection*{Previous works}
In the case of curves, an algorithm for computing periods was first given by \textcite{DeconinckVanHoeij_2001}, and later extended by numerous authors, e.g. \textcite{Swierczewski_2017, Neurohr_2018, MolinNeurohr_2019, BruinEtAl_2019} to name a few.
In higher dimensions, the works of \textcite{CynkVanStraten_2019} and \textcite{ElsenhansJahnel_2022} give methods for computing the periods of Calabi--Yau manifolds obtained from double covers of respectively $\bbP^3$ and $\bbP^2$ ramified along a hyperplane arrangement.
An algorithm for the computation of periods of smooth projective hypersurfaces was given by \textcite{Sertoz_2019}.
An alternative approach for hypersurfaces was then developed in \textcite{LairezEtAl_2024}, and later extended beyond hypersurfaces to elliptic surfaces in \textcite{Pichon-Pharabod_2025}.
The methods presented here build on these last two papers.

Finally methods for numerical computations of periods of fibre products of elliptic surfaces were independently studied in \textcite{Donlagic_2025}.
Our approach is different: the author works directly with a resolution of the singular model, whereas we consider a smoothing.
In particular our results apply in more generality, as we do not restrict to semi-stable singular fibres.
Furthermore our approach yields additional relevant data, namely a certified description of the third homology group of a smoothing along with its intersection product, as well as certified precision bounds for the numerical approximations of the periods.


\subsection*{Outline}
We begin in \secref{sec:fibrations} by recalling generalities about fibrations.
\secref{sec:elliptic_surfaces} recalls the necessary ingredients for computing the homology of elliptic surfaces that are relevant to our discussion, mostly following \textcite{Pichon-Pharabod_2025}.
\secref{sec:fibre_products} extends the methods to compute the third homology group of fibre products of elliptic surfaces.
We start from smooth fibre products obtained from elliptic surfaces with disjoint sets of critical values, and then access the singular case by means of smoothings.
\secref{sec:period_evaluation} explains how this description of the homology can be used to compute the periods of the threefolds, and gives a means to recover the holomorphic forms of the resulting threefolds.
Finally \secref{sec:gamma_class} applies these methods to 1-parameter families of Calabi--Yau threefolds obtained from fibre products associated to Calabi--Yau operators.
We compute the associated monodromy representation certifiably, and recover insight into the Gamma-class formula for these families.
More precisely, we find a general formula that fits all the examples we have considered, stated in \conjref{conj:gamma_class}.
This formula seems to hold for all examples of the Calabi--Yau database with integral monodromy, and we give the corresponding invariants.

\subsection*{Acknowledgements}
The topic of this paper was suggested to me by Duco van Straten.
I am indebted to Charles Doran, Mohamed Elmi, Nutsa Gegelia, Abhiram Kidambi, Pierre Lairez, Duco van Straten, and Pierre Vanhove for insightful discussions and valuable comments on an earlier version of the text.

\section{Homology of fibrations}\label{sec:fibrations}
Let $X$ be a smooth complex manifold equipped with a proper surjective map $f\colon X\to V$ for some connected complex curve $V$.
For $t\in V$, we denote $F_t = f^{-1}(t)$.
We denote by $\Sigma$ the finite set of critical values of $\Sigma$.

\subsection{Monodromy, extensions and parabolic homology}
We briefly recall the notions of monodromy and extensions, following \textcite{Lamotke_1981} and Section 2.1.2 of \textcite{LairezEtAl_2024}.
The restriction of $f$ to $f^{-1}(V\setminus \Sigma)$ is a locally trivial fibration:
if $U\subset V\setminus\Sigma$ is open and simply connected, there is a trivialisation $f^{-1}(U) \simeq F_b\times U$ of the fibration, for all $b\in U$.
In particular a non-self-intersecting path $\ell : [0,1]\to \bbP^1\setminus \Sigma$ induces a diffeomorphism $F_{\ell(0)}\simeq F_{\ell(1)}$ which is unique up to some automorphism of $F_{\ell(1)}$ that is isotopic to the identity. 
Thus $\ell$ induces an automorphism $\ell_*:H_k(F_{\ell(0)})\to H_k(F_{\ell(1)})$ for all $k$.
For $\ell$, $\ell'$ two non-intersecting such paths compatible for concatenation that do not intersect, one may show that 
\begin{equation}\label{eq:comp_monodromy}
(\ell'\ell)_* = \ell'_*\circ \ell_*\,,
\end{equation}
(where $\ell'\ell$ is the path that goes through $\ell$ first, then through $\ell'$). 
Using this formula, we can extend the notion of monodromy to self intersecting paths, and to loops. 
Furthermore, one may show that the map $\ell_*$ depends only on the homotopy class of $\ell$.

Let $b\in\bbP^1\setminus \Sigma$.
The above construction yields the \emph{monodromy representation} 
\begin{equation}
\begin{cases}\pi_1(V\setminus\Sigma, b)\to \operatorname{Aut}(H_k(F_b))\\
 [\ell]\mapsto \ell_*
 \end{cases}\,,
\end{equation}
where $[\ell]$ denotes the homotopy class of $\ell$.
The map $\ell_*$ is called the {\em action of monodromy along $\ell$ on $H_k(F_b)$}.
As readily seen from the trivialisation of the fibration, monodromy preserves the intersection product.
Methods to compute this matrix for the middle homology group of the fibre with semi-numerical computations involving the Picard-Fuchs equation of $F_t$ and a period matrix of $F_t$ were developed in \textcite[\S3.5.2]{LairezEtAl_2024}. 
These methods apply to all the monodromy matrices considered in this text.
\begin{remark}
We note here that, alternatively, the monodromy representation of the fibre products considered in \secref{sec:fibre_products} can be computed as the tensor product of the monodromy representations of the corresponding elliptic surfaces, which themselves can be automatically computed using the algorithm of \textcite{Pichon-Pharabod_2025}.
\end{remark}

A related notion to monodromy is that of {\em extensions}. 
Given a non-intersecting path $\ell$, a simply connected neighbourhood $V$ of $\im\ell$ and a $k$-chain $\Delta$ of $F_{\ell(0)}$, 
the identification of $\Delta\times \im\ell$ in $f^{-1}(V)\subset X$ produces a $k+1$-chain with boundary in $F_{\ell(0)}\cup F_{\ell(1)}$.
Once again, the relative homology class of this $k+1$-chain is unique in the relative homology group $H_{k+1}(X, F_{\ell(0)}\cup F_{\ell(1)})$, and only depends on the homotopy class of $\ell$ and the homology class of $\Delta$.
Therefore we define the {\em extension map} $\tau_{\ell} : H_{k}(F_{\ell(0)}) \to H_{k+1}(X, F_{\ell(0)}\cup F_{\ell(1)})$.
Similarly to monodromy, extensions satisfy a composition rule which allows to extend their definition to self-intersecting paths:
\begin{equation}\label{eq:extensions}
\tau_{\ell'\ell}(\gamma) = \tau_\ell(\gamma) + \tau_{\ell'}(\ell_*(\gamma)) \,,
\end{equation}
in $H_{k+1}(X, F_{\ell(0)}\cup F_{\ell(1)}\cup F_{\ell'(1)})$.
In particular, when $\ell$ is a loop pointed at $b$, we obtain a map
\begin{equation}
\tau : \begin{cases}
\pi_1(V\setminus\Sigma, b)\times H_k(F_b)\to H_{k+1}(X, F_b)\\
[\ell], \Delta \mapsto \tau_\ell(\Delta)
\end{cases}\,.
\end{equation}

Extensions and monodromy are closely related by the formula
\begin{equation}\label{eq:boundary_extension}
\partial(\tau_\ell(\gamma)) = \ell_*\gamma - \gamma\,,
\end{equation}
where $\partial : H_{k+1}(X, F_b)\to H_k(F_b)$ is the boundary map.
This construction is illustrated in \figref{fig:thimbles}.

When the boundary $\partial(u)$ of such a relative cycle $u\in H_k(X, F_b)$ is zero, then $u$ can be lifted to a closed cycle of $X$ module cycles that are contained solely in the fibre.
This is the content of the following proposition.
\begin{proposition}\label{prop:gluing_thimbles}
There is a canonical identification $H_k(X)/\iota_*H_k(F_b)\simeq \ker \partial$ where $\iota_*$ is the pushforward of the inclusion $\iota\colon F_b\to X$.
\end{proposition}
\begin{proof}
This is a direct consequence of the long exact sequence of the pair $(X, F_b)$:
\begin{equation}
H_k(F_b) \xrightarrow{\iota_*} H_k(X)\to H_{k}(X, F_b) \xrightarrow{\partial} H_{k}(F_b)\,.
\end{equation}
\end{proof}

This allows us to define the \emph{parabolic homology} of $X$ equipped with its fibration.
\begin{definition}
The \emph{$k$-th parabolic homology group} $\Hpara_k(X)$ of $X$ equipped with $f\colon X\to V$ is the submodule of $H_k(X)/H_k(F_b)$ generated by closed extensions:
\begin{equation}
	\Hpara_k(X) := \ker\partial\cap \left(\im \tau_{\ell_1} + \dots + \im \tau_{\ell_r}\right)
\end{equation}
for $\ell_1, \dots, \ell_r$ a generating set of $\pi_1(V\setminus\Sigma, b)$.
\end{definition}
The main point of this definition, as we will see in \secref{sec:period_evaluation}, is that the periods of such cycles may be computed using the methods of \textcite{LairezEtAl_2024}, and that this information is sufficient to compute periods on all cycles.

\begin{figure}[]
  \centering
  \begin{subfigure}[b]{0.45\linewidth}
    \begin{center}
  \def\svgwidth{\columnwidth}
\begingroup%
  \makeatletter%
  \providecommand\color[2][]{%
    \errmessage{(Inkscape) Color is used for the text in Inkscape, but the package 'color.sty' is not loaded}%
    \renewcommand\color[2][]{}%
  }%
  \providecommand\transparent[1]{%
    \errmessage{(Inkscape) Transparency is used (non-zero) for the text in Inkscape, but the package 'transparent.sty' is not loaded}%
    \renewcommand\transparent[1]{}%
  }%
  \providecommand\rotatebox[2]{#2}%
  \newcommand*\fsize{\dimexpr\f@size pt\relax}%
  \newcommand*\lineheight[1]{\fontsize{\fsize}{#1\fsize}\selectfont}%
  \ifx\svgwidth\undefined%
    \setlength{\unitlength}{417.98354742bp}%
    \ifx\svgscale\undefined%
      \relax%
    \else%
      \setlength{\unitlength}{\unitlength * \real{\svgscale}}%
    \fi%
  \else%
    \setlength{\unitlength}{\svgwidth}%
  \fi%
  \global\let\svgwidth\undefined%
  \global\let\svgscale\undefined%
  \makeatother%
  \begin{picture}(1,0.85519589)%
    \lineheight{1}%
    \setlength\tabcolsep{0pt}%
    \put(0,0){\includegraphics[width=\unitlength,page=1]{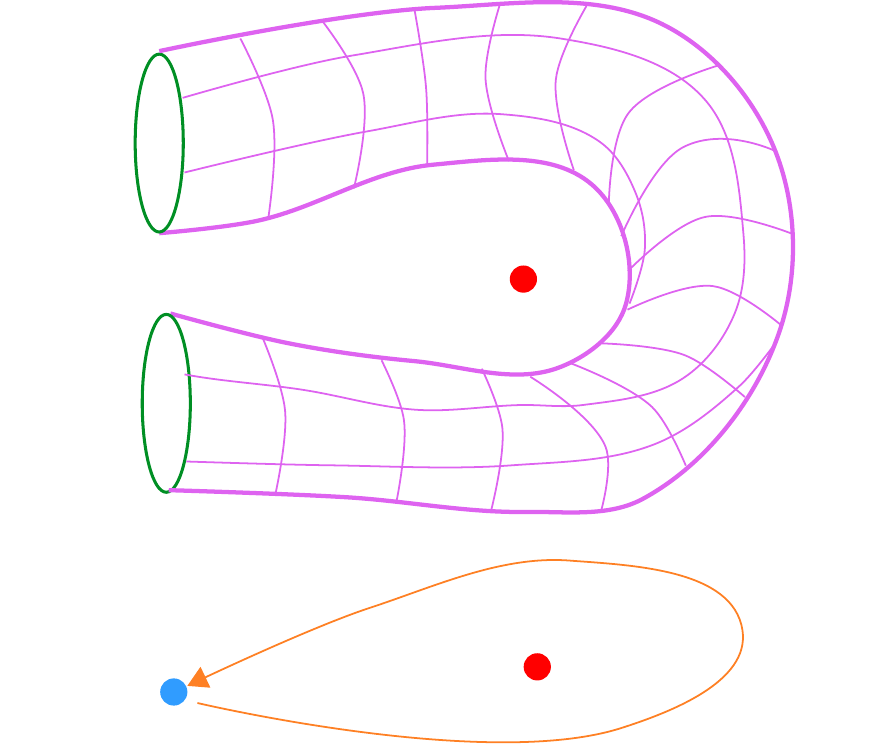}}%
    \put(0.12861315,0.37471888){\color[rgb]{0,0,0}\makebox(0,0)[t]{\lineheight{1.25}\smash{\begin{tabular}[t]{c}$\gamma$\end{tabular}}}}%
    \put(0.87899475,0.18680271){\color[rgb]{0,0,0}\makebox(0,0)[t]{\lineheight{1.25}\smash{\begin{tabular}[t]{c}$\ell$\end{tabular}}}}%
    \put(0.15327059,0.11097117){\color[rgb]{0,0,0}\makebox(0,0)[t]{\lineheight{1.25}\smash{\begin{tabular}[t]{c}$b$\end{tabular}}}}%
    \put(0.65752955,0.08541752){\color[rgb]{0,0,0}\makebox(0,0)[t]{\lineheight{1.25}\smash{\begin{tabular}[t]{c}$t_j$\end{tabular}}}}%
    \put(0.09194184,0.68166978){\color[rgb]{0,0,0}\makebox(0,0)[t]{\lineheight{1.25}\smash{\begin{tabular}[t]{c}$\ell_*(\gamma)$\end{tabular}}}}%
    \put(0.63197601,0.57434441){\color[rgb]{0,0,0}\makebox(0,0)[t]{\lineheight{1.25}\smash{\begin{tabular}[t]{c}$x_j$\end{tabular}}}}%
    \put(0.89212078,0.78705278){\color[rgb]{0,0,0}\makebox(0,0)[t]{\lineheight{1.25}\smash{\begin{tabular}[t]{c}$\tau_\ell(\gamma)$\end{tabular}}}}%
  \end{picture}%
\endgroup%

    \end{center}
  \end{subfigure}
  \begin{subfigure}[b]{0.45\linewidth}
    \begin{center}
  \def\svgwidth{\columnwidth}
\begingroup%
  \makeatletter%
  \providecommand\color[2][]{%
    \errmessage{(Inkscape) Color is used for the text in Inkscape, but the package 'color.sty' is not loaded}%
    \renewcommand\color[2][]{}%
  }%
  \providecommand\transparent[1]{%
    \errmessage{(Inkscape) Transparency is used (non-zero) for the text in Inkscape, but the package 'transparent.sty' is not loaded}%
    \renewcommand\transparent[1]{}%
  }%
  \providecommand\rotatebox[2]{#2}%
  \newcommand*\fsize{\dimexpr\f@size pt\relax}%
  \newcommand*\lineheight[1]{\fontsize{\fsize}{#1\fsize}\selectfont}%
  \ifx\svgwidth\undefined%
    \setlength{\unitlength}{430.90026855bp}%
    \ifx\svgscale\undefined%
      \relax%
    \else%
      \setlength{\unitlength}{\unitlength * \real{\svgscale}}%
    \fi%
  \else%
    \setlength{\unitlength}{\svgwidth}%
  \fi%
  \global\let\svgwidth\undefined%
  \global\let\svgscale\undefined%
  \makeatother%
  \begin{picture}(1,0.82956048)%
    \lineheight{1}%
    \setlength\tabcolsep{0pt}%
    \put(0,0){\includegraphics[width=\unitlength,page=1]{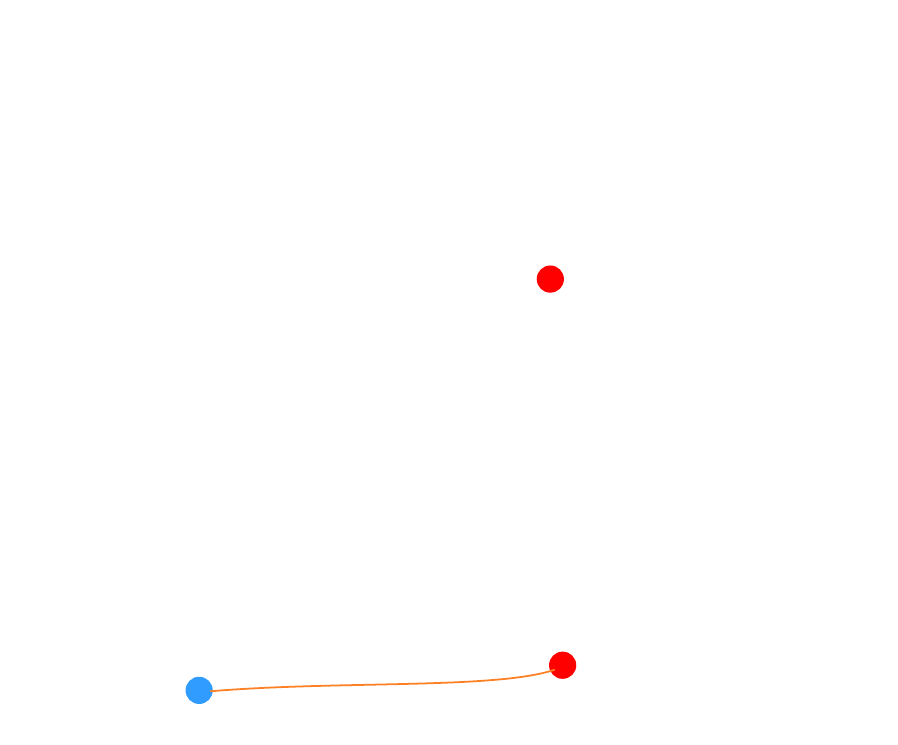}}%
    \put(0.08007343,0.57149155){\color[rgb]{0,0,0}\makebox(0,0)[t]{\lineheight{1.25}\smash{\begin{tabular}[t]{c}$\partial\Delta$\end{tabular}}}}%
    \put(0.39352431,0.38729099){\color[rgb]{0,0,0}\makebox(0,0)[t]{\lineheight{1.25}\smash{\begin{tabular}[t]{c}$\Delta$\end{tabular}}}}%
    \put(0.45026464,0.08481528){\color[rgb]{0,0,0}\makebox(0,0)[t]{\lineheight{1.25}\smash{\begin{tabular}[t]{c}$\ell$\end{tabular}}}}%
    \put(0.17693377,0.10944814){\color[rgb]{0,0,0}\makebox(0,0)[t]{\lineheight{1.25}\smash{\begin{tabular}[t]{c}$b$\end{tabular}}}}%
    \put(0.66607695,0.0846606){\color[rgb]{0,0,0}\makebox(0,0)[t]{\lineheight{1.25}\smash{\begin{tabular}[t]{c}$t_j$\end{tabular}}}}%
    \put(0,0){\includegraphics[width=\unitlength,page=2]{thimble2.pdf}}%
    \put(0.67310427,0.51563997){\color[rgb]{0,0,0}\makebox(0,0)[t]{\lineheight{1.25}\smash{\begin{tabular}[t]{c}$x_j$\end{tabular}}}}%
  \end{picture}%
\endgroup%

    \end{center}
  \end{subfigure}
    \caption{\textit{Right:} Extending the $k$-cycle $\gamma\in H_{k}(F_b)$ (in green) along the loop $\ell\in \pi_1(V\setminus\Sigma)$ (in orange) yields a $k+1$-cycle $\tau_\ell(\gamma)$ (in pink). The monodromy along $\ell$ sends $\gamma$ to $\ell_*(\gamma)$, and the difference $\ell_*\gamma-\gamma$ is the boundary of $\tau_\ell(\gamma)$. \textit{Left:} When the loop is a simple loop around a single critical value $t_j\in\Sigma$ with a simple node $x_j$ in $F_{t_j}$, such a $k+1$-cycle is a \emph{thimble}. The thimble can equivalently be obtained by extending a vanishing $k-1$ cycle along a path to the critical value. The vanishing cycle $\partial \Delta$ collapses and vanishes into the singular point $x_j$.}\label{fig:thimbles}
  \label{fig:thimble}
\end{figure}

\subsection{Local decomposition, thimbles and Lefschetz fibrations}\label{sec:local_decomposition}

We start by describing $H_k(X, F_b)$ in terms of local contributions from the singular fibres.
We assume here that $V$ is homeomorphic to a disk.
Let $b\in V\setminus \Sigma$ be a generic base point.
Index the critical values $\{t_1, \dots, t_r\}  = \Sigma$, and pick pairwise non-intersecting topological open disks ${D_1, \dots, D_r, D_b\subset \bbP^1}$ around $t_1, \dots, t_r, b$ respectively.
 Pick $b_j\in D_j$ and pairwise non-intersecting paths $p_j:[0,1]\to \bbP^1$ such that $p_j(0)=b$ and $p_j(1)=b_j$.
Up to reordering the labels of the critical values, we can assume without loss of generality that the anticlockwise generator of $\pi_1(D_b\setminus \{b\})$ intersects $p_1, p_2, \dots, p_r$ in (cyclic) order.
Finally let $\ell_j$ be the loop obtained by conjugating the anticlockwise generator of $\pi_1(D_j\setminus\{t_i\}, b_j)$ with $p_i$, so as to obtain a class in $\pi_1(V\setminus\Sigma, b)$.
The above reordering ensures that the composition $\ell_r\cdots \ell_1$ is a simple loop around all the critical values, i.e., isotopic to the boundary of $V$ in $\pi_1(V\setminus\Sigma)$.

\begin{figure}[h]
  \centering
  \begin{subfigure}[b]{0.45\linewidth}
    \begin{center}
  \def\svgwidth{\columnwidth}
  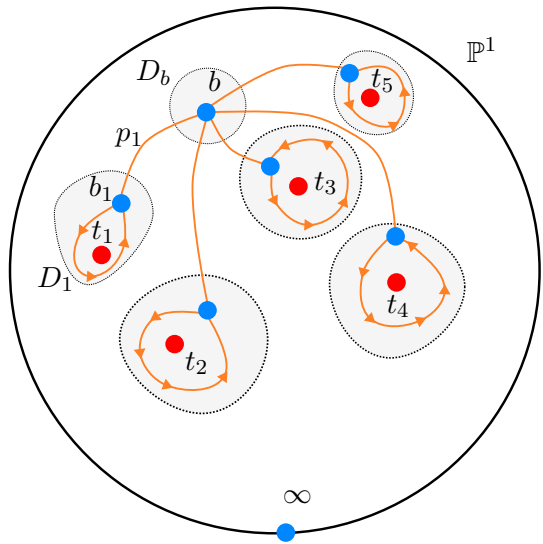

    \end{center}
  \end{subfigure}
    \caption{An illustration of the choices made at the begining of \secref{sec:local_decomposition} in the case $V = \bbP^1\setminus\{\infty\}$. We pick a generic point $b\in V\setminus\Sigma$. Around each $t_j\in\Sigma$ lies an open disk $D_j$ in which we pick a point $b_j\ne t_j$. We connect $b$ to $b_j$ via a path $p_j$. The loop $\ell_j$ is the conjugation of the anticlockwise generator of $\pi_1(D_j\setminus t_j, b_j)$ by $p_j$. We reorder the points to the composition $\ell_5\cdots \ell_1$ is homotopic to a simple anticlockwise loop around all the points of $\Sigma$ (in red), which is equivalent to a simple clockwise loop around $\infty$.}
  \label{fig:thimble}
\end{figure}

\begin{lemma}\label{lem:local_decomp}
The inclusion yields an isomorphism
\begin{equation}
\bigoplus_{j=1}^r H_k(f^{-1}(D_j), F_{b_j})\xrightarrow{\sim} H_k(X, F_b)\,,
\end{equation}
where the identification $F_{b_j}\simeq F_{b}$ is induced by transport along $p_j$.
\end{lemma}
\begin{proof}
See \textcite[\S5.3]{Lamotke_1981}.
\end{proof}

We now define a notion of genericity for fibrations.
\begin{definition}
We say a singular fibre $F_{t_j}$ is \emph{of Lefschetz type} when its singular locus consists of a simple node.
We say $f\colon X\to V$ is a \emph{Lefschetz fibration} when all singular fibres are of Lefschetz type.
\end{definition}
If $t_j$ is of Lefschetz type, then the relative homology spaces $H_k(f^{-1}(D_i), F_{b_j})$ discussed in \lemref{lem:local_decomp} is particularly simple.
\begin{lemma}[\textcite{Lamotke_1981}, Main lemma]\label{lem:thimbles}
If $t_j$ is of Lefschetz type,
\begin{enumerate}
\item $H_k(f^{-1}(D_j), F_{b_j})=0$ if $k\ne n:= \dim X$.
\item $H_n(f^{-1}(D_j), F_{b_j})$ is free of rank $1$.
\end{enumerate}
\end{lemma}
This allows to define the thimble of this singular fibre.
\begin{definition}\label{def:thimble}
The \emph{thimble} $\Delta_j$ at $t_j$ is the image of one of the two generators of $H_n(f^{-1}(D_j), F_{b_j})$ in $H_n(X, F_b)$.
It depends on the choice of the isotopy class of $p_j\in \pi_1(V\setminus\Sigma, b, b_j)$ (isotopy group of paths starting at $b$ and ending at $b_j$).
\end{definition}

The following proposition realises thimbles as extensions.
\begin{proposition}
The image of $H_n(f^{-1}(D_j), F_{b_j})$ in $H_n(X, F_b)$ coincides with the image of $\tau_{\ell_j}$.
In particular $\Delta_j = \tau_{\ell_j}(\gamma)$ for some cycle $\gamma\in H_{n-1}(F_b)$.
\end{proposition}

From \lemref{lem:local_decomp} and \lemref{lem:thimbles} we see that when $V$ is a disk and $X\to V$ is a Lefschetz fibration, thimbles generate $H_n(X, F_b)$ freely.\\

One upside of describing cycles in this manner is that we may evaluate the intersection product away from the singular fibres.
Indeed, given two closed extensions along two paths $\ell_1, \ell_2\in \pi(V\setminus\Sigma)$, one may always deform the loop so that they intersect at finitely many points.
The intersection product is then only a finite sum of signed intersection products of the $n-1$-cycles that are being extended.
In particular, the knowledge of the monodromy representation and the intersection product on the homology of the fibre is sufficient to compute the intersection of the two cycles.
In fact, this intersection product can be derived from a pseudo-intersection product on the thimbles, or rather, a pseudo-lattice structure on $H_k(X^*, F_b)$.
For more details, see \textcite[\S5.2]{LairezEtAl_2024}.

\section{Homology of elliptic surfaces}\label{sec:elliptic_surfaces}

We start by recalling some definitions and results of \textcite{Pichon-Pharabod_2025} relating to elliptic surfaces that are relevant to our discussion.
For a more detailed account of these result, we point to this paper, and for more background on elliptic surfaces we point to \textcite{Miranda_1989, SchuttShioda_2010, Esole_2017}.
\begin{definition}
Let $V$ be a complex curve.
An \emph{elliptic surface} over $V$ is a complex surface $S$ equipped with a proper surjective map $f\colon S\to V$ such that
\begin{itemize}
\item for all but finitely many $t\in V$, the fibre $E_t = f^{-1}(t)$ is an elliptic curve (a smooth genus 1 complex curve).
\item no fibre contains a smooth rational curve of self-intersection $-1$.
\end{itemize}
\end{definition}
The second condition ensures that $S$ is relatively minimal, as such curves can always be blown down.
We will consider elliptic surfaces over $\bbP^1$.
We will use the shorthand $S/\bbP^1$ to designate the surface $S$ equipped with its map.

\begin{definition}
A section of an elliptic surface $S/\bbP^1$ is a map $\pi\colon \bbP^1\to S$ such that $f\circ\pi = \operatorname{id}_{\bbP^1}$.
\end{definition}
In what follows we consider an elliptic surface with a section, and fix such a section which we call the \emph{zero section} and denote $O$.

The fibre types of an elliptic fibration have been classified by \textcite{Kodaira_1963} into
\begin{itemize}
\item two infinite families $\mathit{I}_n$ and $\mathit{I}_n^*$, where $n\ge 0$;
\item and six types $\mathit{II}$, $\mathit{III}$, $\mathit{IV}$, $\mathit{II}^*$, $\mathit{III}^*$, $\mathit{IV}^*$.
\end{itemize}
The type of a fibre is entirely determined by its monodromy, the conjugation class of which is given in \tabref{tab:kodaira}.

\begin{table}[]
\centering
\begin{tabular}{ccc}
\toprule
Type & Monodromy &  \thead{Euler \\characteristic} \\ \midrule
  $\mathit{I}_n, n\ge 0$  &  $\left(\begin{array}{cc} 1& n \\ 0 & 1 \end{array}\right)$                &    $n$   \\
  $\mathit{II}$  &  $\left(\begin{array}{cc} 1& 1 \\ -1 & 0 \end{array}\right)$               &    $2$   \\
  $\mathit{III}$  &  $\left(\begin{array}{cc} 0& 1 \\ -1 & 0 \end{array}\right)$                &    $3$   \\
  $\mathit{IV}$  &  $\left(\begin{array}{cc} 0& 1 \\ -1 & -1 \end{array}\right)$                &    $4$   \\\bottomrule
\end{tabular}
\quad 
\begin{tabular}{ccc}
\toprule
Type & Monodromy & \thead{Euler \\characteristic} \\ \midrule
  $\mathit{I}^*_n, n\ge 0$  &  $\left(\begin{array}{cc} -1& -n \\ 0 & -1 \end{array}\right)$                &    $n+6$   \\
  $\mathit{II}^*$  &  $\left(\begin{array}{cc} 0& -1 \\ 1 & 1 \end{array}\right)$               &    $10$   \\
  $\mathit{III}^*$  &  $\left(\begin{array}{cc} 0& -1 \\ 1 & 0 \end{array}\right)$               &    $9$   \\
  $\mathit{IV}^*$  &  $\left(\begin{array}{cc} -1&-1 \\ 1 & 0 \end{array}\right)$                &    $8$   \\\bottomrule
\end{tabular}
\caption{The Kodaira classification.}
\label{tab:kodaira}
\end{table}

Certain cycles of $H_2(S^*, E_b)$ are supported entirely on a single singular fibre and cannot be realised as extensions.
In particular, they define closed cycles, and are generated by components of singular fibres.
\begin{definition}
The \emph{lattice of components of singular fibres} $\operatorname{Sing}(S/\bbP^1)$ is the sublattice of $H_2(S^*, E_b)$ generated by components of singular fibres.
\end{definition}

We now define the primary lattice of $S$, slightly deviating from the definition of \textcite[Definition 3]{Pichon-Pharabod_2025} to better suit what follows.
\begin{definition}
The \emph{primary lattice} $\Prim(S)$ is the sublattice of $\ker\left(\partial\colon H_3(S^*, E_b)\to H_2(E_b)\right)$ generated by extensions and fibre components:
\begin{equation}
\Prim(S/\bbP^1) :=\Hpara_3(S/\bbP^1) \oplus \operatorname{Sing}(S/\bbP^1)\,.
\end{equation}
\end{definition}

\begin{proposition}\label{prop:fullrank_ES}
The primary lattice $\Prim(S/\bbP^1)$ has full rank in $\ker\partial$.
\end{proposition}
\begin{proof}
This is a corollary of the proof of Lemma 15 in \textcite{Pichon-Pharabod_2025}.
\end{proof}

Generic elliptic surfaces are Lefschetz fibrations, meaning they only have $I_1$ singularities.
In this case one may define thimbles as is done in \defref{def:thimble}, after choosing appropriate disks and paths.
In the case where $S$ has more complicated fibres, we fall back to the generic case by means of a \emph{morsification}.
\begin{definition}
A \emph{morsification} of $S/\bbP^1$ is the data of a disk $D$ and a threefold $\tilde S$ such that there is a commutative diagram
\begin{equation}
    \begin{tikzcd} [sep = .5 cm]
      \tilde S \arrow[r, "\tilde f"] \arrow[rr, dashed, bend right, "\eta"]& \bbP^1\times D \arrow[r, "p"]& D
    \end{tikzcd}\,,
\end{equation}
  where $p$ is the projection onto the second coordinate, satisfying
  \begin{itemize}
  \item $\eta: \tilde S\to D$ is a locally trivial smooth fibration;
  \item $\tilde f|_{S_0}: S_0 \to \bbP^1$ coincides with $S\to \bbP^1$;
  \item for $u\in D\setminus\{0\}$, $\tilde f|_{S_u}: S_u \to \bbP^1$ is a Lefschetz elliptic surface;
  \item $\eta$ has no critical values.
  \end{itemize}
\end{definition}
A theorem of Moishezon shows that every elliptic surface admits a morsification \parencite[Theorem 8]{Moishezon_1977}.
This allows to generalise the definition of thimble to elliptic surfaces that are not of Lefschetz type.
One consequence of this theorem is the following equivalent of \lemref{lem:thimbles}.
\begin{lemma}\label{lem:main_lemma}
For any $V\subset\bbP^1$ and $b\in V$
\begin{enumerate}
\item $H_k(f^{-1}(V), F_{b})=0$ if $k\ne n:= \dim X$.
\item $H_n(f^{-1}(V), F_{b})$ is free.
\end{enumerate}
\end{lemma}

Similarly, it allows to extend \defref{def:thimble} to any elliptic surface.

\begin{definition}
The thimbles of $S/\bbP^1$ are the thimbles of the generic fibre $\tilde S_t$, $t\in D\setminus\{0\}$ of a morsification of $S/\bbP^1$.
\end{definition}
\begin{remark}
Of course this definition depends on a choice of morsification, and of loops $\ell_i$ in the base of the morsification.
However this choice does not affect the rest of this paper, so we can make it arbitrarily.
\end{remark}

\section{Homology of fibre products}\label{sec:fibre_products}

We consider $S^\sindex{1}$ and $S^\sindex{2}$ two non-isotrivial elliptic surfaces over $\bbP^1$.
We keep the notations of \secref{sec:elliptic_surfaces} but annote them with the superscript $\sindex{i}$ for $i=1,2$ to refer to either the first or second elliptic surface, e.g., $\Sigma^\sindex{1}$, $f^\sindex{1}$, etc.. 
Furthermore, the index $i'$ will indicate the value that $i$ does not take, so that $(i,i')\in\{(1,2),(2,1)\}$.

Let $T = S^\sindex{1}\times_{\bbP^1} S^\sindex{2}$ be the fibre product of $S^\sindex{1}$ and $S^\sindex{2}$ naturally equipped with a map~$f\colon T\to \bbP^1$.
For $t\in \bbP^1$, we denote $F_t := f^{-1}(t) = E^\sindex{1}_t\times E^\sindex{2}_t$.

Our goal is to leverage the description of the homology of $S^\sindex{1}$ and $S^\sindex{2}$ recalled in \secref{sec:elliptic_surfaces} to obtain an explicit handle on cycles of $T$.
The general principle is the following: $n+1$-cycles may be obtained by extensions of $n$-cycles along loops.
Linear relations between extensions can be tracked by expressing them in terms of a basis of \emph{thimbles}.
In the case of elliptic surface, it was shown in \textcite{Pichon-Pharabod_2025} that for generic elliptic surfaces (i.e. with only $I_1$ singular fibres) such extensions, along with the section and fibre class, generate the full second homology group of the elliptic surface.
We will show here that the situation is even simpler, as there is no class stemming from the generic fibre, nor from a section.
However the direct approach employed in \textcite{Pichon-Pharabod_2025} is not applicable here, as the fibration is not of Lefschetz type (each singular fibre always has an even number of vanishing cycles, so in particular never only one).
Instead we rely on our understanding of the homology of $S^\sindex{1}$ and~$S^\sindex{2}$ to access the homology of $T$.

We first work under the simplifying hypothesis that the set of critical values of $S^{\sindex{1}}$ and $S^{\sindex{2}}$ are disjoint.
We will later alleviate this condition.
\begin{definition}
The threefold $T$ is said to be \emph{generic} if $\Sigma^{\sindex{1}}\cap \Sigma^{\sindex{2}} = \emptyset$, so all fibres are either of type $I_0\times F$ or $F\times I_0$, where $F$ is any Kodaira type. 
\end{definition}

\subsection{The generic case}
Assume $T$ is generic.
We will recover a description of $H_3(T)$ in terms of thimbles of $S^\sindex{1}$ and $S^\sindex{2}$.
We define the thimbles of $T$ to be the tensor products of thimbles of $S^\sindex{i}/\bbP^1$ with elements of an arbitrary basis of $E_b^\sindex{i'}$, for $i=1,2$.
We start by relating thimbles of the threefold to thimbles of the elliptic surfaces.

\begin{proposition}\label{prop:thimbles_FP_generic}
We have that
\begin{equation}
	H_k(T^*, F_b) \simeq 
	\bigoplus_{i\in\{1,2\}} H_2\left({S^\sindex{i}}^*, E^\sindex{i}_b\right) \otimes H_{k-2}\left(E^\sindex{i'}_b\right)\,.
\end{equation}
In words, when $k=3$, thimbles of $T/\bbP^1$ consist exactly of thimbles of $S^\sindex{i}/\bbP^1$ tensored with a 1-cycle of $E^\sindex{i'}$, for $i\in\{1,2\}$.
\end{proposition}
\begin{proof}
We set $\{t_1, \dots, t_r\} = \Sigma^\sindex{1}\cup \Sigma^\sindex{2}$, and choose $D_1, \dots, D_r$ and $b_1, \dots, b_r$ as in \lemref{lem:local_decomp}, and it follows from this lemma that we have
\begin{equation}
	H_k(T^*, F_b) = \bigoplus_{i\in\{1,2\}} \bigoplus_{t_j\in \Sigma^\sindex{i}} H_k\left(f^{-1}(D_j), F_{b_j}\right)\,.
\end{equation}
Let $t_j\in \Sigma^\sindex{1}$. The disk $D_i$ is a simply connected open subset not intersecting $\Sigma^\sindex{2}$.
Therefore there is a trivialisation of $(f^\sindex{2})^{-1}(D_j) \simeq E^\sindex{2}_{b_j}\times D_j$, and thus $f^{-1}(D_j)\simeq \left(f^\sindex{1}\right)^{-1}(D_j)\times E^\sindex{2}_{b_j}$.
In particular 
\begin{align}
	H_k\left(f^{-1}(D_j), F_{b_j}\right) &\simeq H_k\left((f^\sindex{1})^{-1}(D_j)\times E^\sindex{2}_{b_j}, E^\sindex{1}_{b_j}\times E^\sindex{2}_{b_j}\right)\\
		&\simeq \bigoplus_{l=0}^k H_{l}\left((f^\sindex{1})^{-1}(D_j), E^\sindex{1}_{b_j}\right)\otimes H_{k-l}\left( E^\sindex{2}_{b_j}\right)\\
		&\simeq H_{2}\left((f^\sindex{1})^{-1}(D_j), E^\sindex{1}_{b_j}\right)\otimes H_{k-2}\left( E^\sindex{2}_{b_j}\right)\,.
\end{align}
Here the Künneth formula allows to go from the first line to the second, and \lemref{lem:main_lemma} from the second to the third.
Therefore putting everything back together
\begin{equation}
	H_k(T^*, F_b) = \bigoplus_{i\in\{1,2\}} \bigoplus_{t_j\in \Sigma^\sindex{i}} H_{2}\left((f^\sindex{i})^{-1}(D_j), E^\sindex{i}_{b_j}\right)\otimes H_{k-2}\left(E^\sindex{i'}_{b_j}\right)\quad\,,
\end{equation}
and using \lemref{lem:local_decomp} again, we obtain the claim.
\end{proof}

We now use techniques similar to \textcite{LairezEtAl_2024, Pichon-Pharabod_2025} to relate $H_3(T, F_b)$ and then $H_3(T)$ to thimbles.
The first step is to add the fibre $F_\infty$ back.
As a consequence, this kills extensions around the simple loop $\ell_r\cdots \ell_1$ around $\infty$, as this loop becomes trivial in $\pi_1(\bbP^1\setminus\Sigma, b)$ (whereas it was not in $\pi_1(\bbP^1\setminus\Sigma, b)$).
We denote the extension map $\tau_{\ell_r\cdots \ell_1}$ by $\tau_\infty$ to ease the notation.
The second step glues thimbles together to obtain closed 3-cycles --- this amounts to taking the kernel of the boundary map.
We summarise this in the following theorem.

\begin{theorem}
We have that
\begin{equation}
	H_3(T, F_b)\simeq \frac{H_3(T^*, F_b)}{\im \left(\tau_\infty\colon H_2(F_b)\to H_3(T^*, F_b)\right)}\,,
\end{equation}
and furthermore
\begin{equation}
	H_3(T)\simeq \frac{\ker\left(\partial\colon H_3(T^*, F_b)\to H_2(F_b)\right)}{\im \left(\tau_\infty\colon H_2(F_b)\to H_3(T^*, F_b)\right)}\,.
\end{equation}
\end{theorem}
\begin{proof}
The long exact sequence of homology of the triple $(T, T^*, F_b)$ yields
\begin{equation}
H_2(F_b) \xrightarrow{\tau_\infty} H_3(T^*, F_b)\to H_{3}(T, F_b) \to  H_{1}(F_b)\xrightarrow{\tau_\infty} H_{2}(T^*, F_b)\,,
\end{equation}
where $H_k(T, T^*)$ has been identified with $H_{k-2}(F_b)$ by the Künneth formula, as $(T, T^*)\simeq F_b\times (D, S^1)$.

Let us show that $\tau_\infty\colon H_{1}(F_b)\to H_{2}(T^*, F_b)$ is injective.
By the Künneth formula, 
\begin{equation}
H_1(F_b) \simeq H_0(E^\sindex{1})\otimes H_1(E^\sindex{2}_b) \oplus H_1(E^\sindex{1}_b)\otimes H_0(E^\sindex{2}_b)\,.
\end{equation}
Let $u\otimes v\in H_0(E^\sindex{1}_b)\otimes H_1(E^\sindex{2}_b)$.
Then 
\begin{equation}
	\tau_\infty(u\otimes v) =  \tau^\sindex{1}_\infty(u)\otimes v + u\otimes\tau^\sindex{2}_\infty( v) = \tau^\sindex{1}_\infty(u)\otimes v\,.
\end{equation}
This shows that the injectivity of $\tau_\infty\colon H_{1}(F_b)\to H_{2}(T^*, F_b)$ boils down to the injectivity of
\begin{equation}
	\tau_\infty^\sindex{i}\colon H_{1}(E^\sindex{i}_b)\to H_{2}({S^\sindex{i}}^*, E^\sindex{i}_b)\,.
\end{equation}
which follows from the fact that the monodromy group of a non-isotrivial elliptic surface is $\operatorname{SL}_2(\Z)$.
This concludes the first point.\vspace{.5em}

To prove the second point, we rely on \propref{prop:gluing_thimbles}.
In our setting, the pushforward of the inclusion $\iota_*\colon H_3(F_b)\to H_3(T)$ is in fact the zero map.
Indeed, from the Künneth formula
\begin{equation}
	H_3(F_b)\simeq H_2(E^\sindex{1}_b)\otimes H_1(E^\sindex{2}_b) \quad\oplus\quad H_1(E^\sindex{1}_b)\otimes H_2(E^\sindex{2}_b)\,,
\end{equation}
and the pushforward of the inclusion $H_1(E_b^\sindex{1})\to H_1(S^\sindex{1})$ is the zero map as 1-cycles vanish at the singular fibres.
\end{proof}

\begin{example}
When $S^\sindex{1}$ and $S^\sindex{2}$ are rational surfaces, then $H_2(S^\sindex{i}, E^\sindex{i}_b)$ is free of rank $12$.
The image of the boundary map $\partial\colon H_3(T^*, F_b)\to H_2(F_b)$ defines a surjection onto the rank~$4$ module $H_1(E^\sindex{1})\otimes H_1(E^\sindex{2})$ and similarly the image of $\tau_\infty$ has rank $4$.
As a consequence we recover the known fact that $\rk H_3(T) = 4\rk H_2(S^\sindex{i}, E^\sindex{i}_b) -4 -4 = 40$ (see \parencite{Schoen_1988}).
\end{example}

Certain 3-cycles of $T$ are contained entirely in a fibre. 
By \propref{prop:thimbles_FP_generic}, these are cycles of the form $\Theta^\sindex{i} \otimes \gamma\sindex{i'}$ where $\Theta^\sindex{i}\in \operatorname{Sing}(S^\sindex{i})$ and $\gamma^\sindex{i'}\in H_1(E^\sindex{i'}_b)$.
\begin{definition}
The \emph{lattice of components of singular fibres} $\operatorname{Sing}(T)$ is the sublattice of $H_3(T)$ generated by such cycles.
\begin{equation}
    \operatorname{Sing}(T) = \bigoplus_{i\in\{1,2\}} \operatorname{Sing}(S^\sindex{i})\otimes H_1(E^\sindex{i'}_b)\,.
\end{equation}
\end{definition}

We now describe a sublattice of cycles for which period computations can be directly done, either by methods of \textcite{LairezEtAl_2024}, or because they are zero by construction. We call such cycles \emph{primary}.
\begin{definition}
The \emph{primary lattice} $\Prim(T/\bbP^1)$ is the sublattice of $H_3(T)$ generated by extensions and fibre components:
\begin{equation}
\Prim(T/\bbP^1) :=\Hpara_3(T/\bbP^1) \oplus \operatorname{Sing}(T/\bbP^1)\,.
\end{equation}
\end{definition}

\begin{proposition}\label{prop:fullrank_generic}
The primary lattice $\Prim(T/\bbP^1)$ has full rank in $H_3(T)$.
\end{proposition}
\begin{proof}
This is a consequence of the corresponding result for elliptic surfaces \propref{prop:fullrank_ES} and \lemref{prop:thimbles_FP_generic}.
\end{proof}

We conclude by remarking that the variation of $H_2(F_t)$ is entirely determined by the variation of $H_1(E^\sindex{1}_t)$ and of $H_1(E^\sindex{2}_t)$.
By the Künneth formula, we have 
\begin{equation}
	H^2(F_t) \simeq H^1(E^\sindex{1}_{t})\otimes H^1(E^\sindex{2}_{t})\oplus H^0(E^\sindex{1}_{t})\otimes H^2(E^\sindex{2}_{t}) \oplus H^2(E^\sindex{1}_{t})\otimes H^0(E^\sindex{2}_{t})\,.
\end{equation}
Note that $H^0(E^\sindex{i}_{t})$ and $H^2(E^\sindex{i}_{t})$ have rank $1$ and trivial monodromy for $i=1,2$. 
Thus the only relevant part of the homology of the fibre is $H^1(E^\sindex{1}_{t})\otimes H^1(E^\sindex{2}_{t})$, which has rank $2\times 2 = 4$.

In fact, we see that for $\gamma^\sindex{1}\otimes \gamma^\sindex{2}\in H^1(E^\sindex{1}_{t})\otimes H^1(E^\sindex{2}_{t})$, we have
\begin{equation}
	\ell_*(\gamma^\sindex{1}\otimes \gamma^\sindex{2}) = \ell^\sindex{1}_*\gamma^\sindex{1}\otimes \ell^\sindex{2}_*\gamma^\sindex{2}\,,
\end{equation}
so the monodromy representations of $S^\sindex{1}$ and $S^\sindex{2}$ determine the monodromy representation associated to $H_2(F_b)$ of $T$ entirely.

Furthermore the intersection product on $H^1(E^\sindex{1}_{t})\otimes H^1(E^\sindex{2}_{t})$ is itself also simply given by the product of the intersection products on each component:
\begin{equation}
	\langle \gamma^\sindex{1}\otimes \gamma^\sindex{2}, \eta^\sindex{1}\otimes \eta^\sindex{2}\rangle = 
	\langle \gamma^\sindex{1}\eta^\sindex{1}\rangle \langle \gamma^\sindex{2} \eta^\sindex{2}\rangle\,.
\end{equation}
In particular the methods of \secref{sec:fibrations} can be used to compute the lattice structure on $H_3(T)$.

\subsection{The singular case}
We now turn to the case where some singular fibres of $S^\sindex{1}$ and $S^\sindex{2}$ lie over the same critical value.
We proceed by realising a \emph{smoothing} of $\tilde T$ (formally) obtained by shifting the critical values.
This is similar to the approach of \textcite[\S5]{Bryan_2019}.
The smoothing corresponds to the generic case mentioned above, and all the methods there can be applied here.
Furthermore, a careful study allows to pinpoint the lattice $\Lambda_{\rm vc}$ of \emph{vanishing cycles} of the degeneration, i.e., the cycles that become homologically trivial in the singular limit.
In the singular limit, the vanishing cycles collapse to the singular locus.

\begin{remark}
In the simplest case of nodal singularities admitting a small resolution, as studied in \textcite{Donlagic_2025}, the third homology group $H_3(\widehat T)$ of the resolution $\widehat T$ of $T$ can be identified with the quotient $\Lambda_{\rm vc}^\perp / \Lambda_{\rm vc}$, from geometric transition theory.

This should be related to the 1-dimensional case. For example, consider a smooth curve of genus $g\ge 1$ with a simple node obtained from pinching a non-intersecting loop on $X$. 
The vanishing cycle is the 1-cycle corresponding to the homology class of the loop. 
Any cycle intersecting the vanishing cycle gets pinched by the singularity, and, in the resolution, is cut in half.
Therefore the surviving cycles of the resolution are precisely those orthogonal to the vanishing cycle, modulo the vanishing cycle.
We will not make any similar claim for threefolds, but keeping this in mind, we will consider $\Lambda_{\rm vc}^\perp / \Lambda_{\rm vc}$ throughout.
\end{remark}

Let $\left(\psi^\sindex{\varepsilon}\right)$ be the Möbius transformation $\left(\psi^\sindex{\varepsilon}\right)\colon t\mapsto t-\epsilon$ on $\bbP^1$.
Assume without loss of generality that $\infty \notin \Sigma$ and define the \emph{smoothing} of $T$ by
\begin{equation}
T^\sindex{\varepsilon} :=S^\sindex{1}\times_{\bbP^1} \left(\psi^\sindex{\varepsilon}\right)^*S^\sindex{2}\,,
\end{equation}
naturally equipped with a map $f^\sindex{\varepsilon}\colon T^\sindex{\varepsilon}\to \bbP^1$.
Its critical values are $\Sigma^\sindex{\varepsilon} := \Sigma^\sindex{1}\cup\left\{t+\varepsilon\st t\in\Sigma^\sindex{2}\right\}$,
and we write $F^\sindex{\varepsilon}_t = (f^\sindex{\varepsilon})^{-1}(t) = F_{t-\varepsilon}$ to denote its fibre above $t\in \bbP^1$.

\begin{remark}
To ease notations, we will not distinguish $S^\sindex{2}$ and $\phi^\sindex{\varepsilon}_*S^\sindex{2}$. 
It should be clear from context which one is meant.
All that changes is that $D_j$ should be replaced by $D_j-\varepsilon$ and $E^\sindex{2}_b$ is (canonically) identified with $E^\sindex{2}_{b-\varepsilon}.$
\end{remark}

Then for a small enough disk $D$ around zero, $t_i+\varepsilon\in D_i$ for $\varepsilon\in D$.
In particular we see that for $\varepsilon\in D\setminus\{0\}$, the fibre product $T_\varepsilon$ is generic, hence smooth (hence the name of smoothing).

Consider the loops $\ell_1, \dots, \ell_r \in \pi_1(\bbP^1\setminus\Sigma)$ defined in \secref{sec:local_decomposition}. 
When $t_i\in \Sigma^\sindex{1}\cap \Sigma^\sindex{2}$, the critical values get split when $\varepsilon\ne0$, and thus $\ell_i$ fails to generate $\pi_1(D_i \setminus \{t_i, t_i+\varepsilon\}, b_i)$.
To obtain a basis we instead replace it by two simple loops $\ell^\sindex{1}_i, \ell^\sindex{2}_i$ around $t_i$ and $t_{i}+\varepsilon$ respectively, such that $\ell_i = \ell^\sindex{2}_i\ell^\sindex{1}_i$.
This is represented in the bottom part of \figref{fig:vanishing_cycles}.
Doing so for every pair of colliding fibres, we obtain a basis of simple loops of $\pi_1(\bbP^1 \setminus\Sigma^\sindex{\varepsilon}, b)$.
The new loops we have introduced get pinched between $t_i$ and $t_i+\varepsilon$ when $\varepsilon\to0$.

\begin{lemma}
	The parabolic homology $\Hpara_3(T/\bbP^1)$ of $T$ injects naturally into $\Hpara_3(T^{\varepsilon}/\bbP^1)$.
\end{lemma}
\begin{proof}
This follows from the fact that for all $j$, any representative of $\ell_j$ remains a loop in $\pi_1(\bbP^1\setminus \Sigma^{\varepsilon})$ for $\epsilon$ sufficiently small.
\end{proof}

Let us focus on a specific $t_j\in \Sigma^\sindex{1}\cap \Sigma^\sindex{2}$.
Then recall from \propref{prop:thimbles_FP_generic} that in the smoothing $T^\sindex{\varepsilon}$ we have generators of $H_3((f^\sindex{\varepsilon})^{-1}(D_j), F^\sindex{\varepsilon}_b)$ given by products $\Delta^\sindex{i}\times \gamma^\sindex{i'}$ where $\Delta^\sindex{i}$ is a thimble of $H_2((f^\sindex{i})^{-1}(D_j), E^\sindex{i}_b)$ and $\gamma^\sindex{i'} \in H_1(E_b^\sindex{i'})$.

In particular, picking for $i=1,2$ such thimbles $\Delta^\sindex{i}\in H_2((f^\sindex{i})^{-1}(D_j), E^\sindex{i}_b)$, we construct an element of $H_3(T^\sindex{\varepsilon})$ as follows.
\begin{definition}
	The \emph{vanishing cycle} corresponding to $\Delta^\sindex{1}$ and $\Delta^\sindex{2}$ is the 3-cycle $[\Delta^\sindex{1},\Delta^\sindex{2}] \in H_3(T^\sindex{\varepsilon})$ obtained from
\begin{equation}
	\Delta^\sindex{1}\otimes \partial\Delta^\sindex{2} - \partial\Delta^\sindex{1}\otimes \Delta^\sindex{2} \in \ker\partial\,.
\end{equation}
We denote by $\Lambda_{\rm vc}$ the sublattice of $H_3(T^\sindex{\varepsilon})$ generated by vanishing cycles.
An illustrative representation of such a cycle can be found in \figref{fig:vanishing_cycles}.
\end{definition}

\begin{figure}[tbp]
  \centering
  \begin{subfigure}[b]{0.7\linewidth}
    \begin{center}
  \def\svgwidth{\columnwidth}
  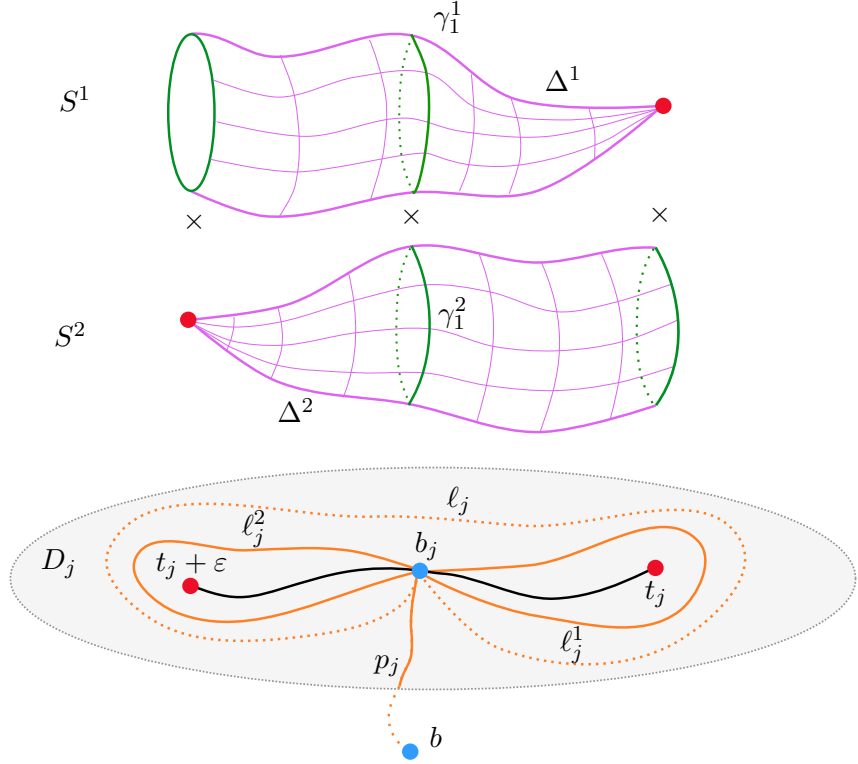

    \end{center}
  \end{subfigure}
    \caption{\textit{Bottom half:} when $\varepsilon\ne0$, $t_j$ splits into two critical values $t_j$ and $t_j+\varepsilon$ of $T^\varepsilon$. We define $\ell^\sindex{i}_j$ as in the drawing, so that $\ell^\sindex{2}_j\ell^\sindex{1}_j = \ell_j$. \textit{Full picture:} in the smoothing, we construct a vanishing cycle of by gluing two thimbles $\Delta^\sindex{1}$ and $\Delta^\sindex{2}$ of $S^\sindex{1}$ and $S^\sindex{2}$. The resulting cycle collapses when $\varepsilon\to 0$.}\label{fig:vanishing_cycles}
  \label{fig:thimble}
\end{figure}

We now define the lattice of \emph{primary cycles}. These are cycles of $H_3(T^\varepsilon)$ for which the periods of the limit in $T$ can be computed directly. 
\begin{definition}
The \emph{primary lattice} $Prim(T^\sindex{\varepsilon}/\bbP^1)$ of $T^\sindex{\varepsilon}/\bbP^1$ is the sublattice of $H_3(T^\varepsilon)$ generated by extensions of $T/\bbP^1$, vanishing cycles, and components of singular fibres of $T^{\varepsilon}/\bbP^1$ :
\begin{equation}
	\Prim(T^\sindex{\varepsilon}) := \Hpara_3(T) \oplus \Lambda_{\rm vc} \oplus \operatorname{Sing}(T^{\varepsilon})
\end{equation}
\end{definition}

Finally we give an equivalent of \propref{prop:fullrank_generic} in this setting.
\begin{proposition}\label{prop:fullrank}
The orthogonal complement of $\Lambda_{\rm vc}\otimes \Q$ is a sublattice of $\Prim(T^\sindex{\varepsilon}/\bbP^1) \otimes \Q$.
\end{proposition}
\begin{proof}
We first claim that $\Hpara_3(T^{\varepsilon})\oplus \operatorname{Sing}(T^{\varepsilon})$ has full rank --- this follows from \lemref{prop:thimbles_FP_generic} and Lemma 15 of \textcite{Pichon-Pharabod_2025}.
All we have to do is show that any extension of $\Hpara_3(T^\sindex{\varepsilon})\setminus \Hpara_3(T)\oplus\Lambda_{\rm vc}$ intersects $\Lambda_{\rm vc}$.
Consider such an extension $e = \tau_\ell(\gamma_1\otimes\gamma_2)$ and assume it does not.
Since $e\notin \Hpara_3(T)\oplus\Lambda_{\rm vc}$, it has a representative given as a decomposition of thimbles involving $\tau_{\ell^\sindex{i}_j}$ for some $t_j\in \Sigma^\sindex{1}\cap \Sigma^\sindex{2}$ and some $i=1,2$.
From \eqref{eq:extensions} and the fact that $\ell^\sindex{1}_j = (\ell^\sindex{2}_j)^{-1}\ell_j$, we can express any extension along $\ell^\sindex{2}_j$ as a sum of an extension along $\ell^\sindex{1}_j$ and along $\ell_j$.
We can thus assume that the only contribution to the intersection product comes from $\tau_{\ell^\sindex{1}_j}(\gamma)$ for some $\gamma = \gamma^\sindex{1}\otimes \gamma^\sindex{2} \in H_1(E_b^{\sindex{1}})\otimes H_1(E_b^{\sindex{2}})$.
The intersection of $e$ with a vanishing cycle $[\Delta^\sindex{1}, \Delta^\sindex{2}]$ at $t_j$ is given by
$\langle\gamma^\sindex{1},\partial\Delta^\sindex{1} \rangle \langle\gamma^\sindex{2}, \partial\Delta^\sindex{2}\rangle$.
This has to equal zero no matter which thimble $\Delta^\sindex{i}$ is, as long as it is a thimble of $S^\sindex{i}$ at $t_j$.
Assuming $\langle\gamma^\sindex{i},\partial\Delta^\sindex{i} \rangle\ne 0$ for some $\Delta^\sindex{i}$, we obtain that $\langle\gamma^\sindex{i'},\partial\Delta^\sindex{i'} \rangle=0$ for all $\Delta^\sindex{i'}$.
Therefore there is $i\in\{1,2\}$ such that $\langle\gamma^\sindex{i},\partial\Delta^\sindex{i} \rangle=0$ for all $\Delta^\sindex{i}$.
This implies that $\tau_{\ell^\sindex{i}_j}(\gamma^\sindex{i})=0$.
If $i=1$, then $\tau_{\ell^\sindex{1}_j}(\gamma)=0$, and if $i=2$, $\tau_{\ell^\sindex{1}_j}(\gamma) = \tau_{\ell_j}(\gamma)$.
Doing this for each pair of colliding fibres shows that $e\in \Hpara_3(T)\oplus\Lambda_{\rm vc}$, a contradiction.
\end{proof}

\section{Evaluation of periods}\label{sec:period_evaluation}

In this section we leverage the description of homology that we established in \secref{sec:fibre_products} to compute periods of fibered products of elliptic surfaces, using methods introduced in \textcite{LairezEtAl_2024}. 
The overall strategy is very similar to the one presented in \textcite{Pichon-Pharabod_2025}.
We will consider algebraic forms of the form
\begin{equation}
\omega = \omega_t \wedge \ud t \in H^3(T)\,,
\end{equation}
for some $\omega_t\in H^2(T)$.
In some sense, these are algebraic forms which do not contribute any periods to fibre components, as the wedge with $\ud t$ makes the restriction of $\omega$ to any fibre vanish.
In other words, periods of such forms on $\operatorname{Sing}(T^\varepsilon)$ and $\Lambda_{\rm vc}$ vanish.
When the threefold is Calabi--Yau, we explain how to express the holomorphic form in such a manner in \secref{sec:holomorphic_form}.

The main idea of our approach is the observation that periods of extensions can be evaluated by numerical integration methods.
Indeed for $\ell\in \pi_1(\bbP^1\setminus\Sigma, b)$ and $\eta\in H_2(F_b)$, we have that
\begin{equation}
\int_{\tau_\ell(\eta)}\omega = \int_{\ell}\left(\int_{\eta_t}\omega_t|_{F_t}\right) \ud t\,,
\end{equation}
where $\eta_t\in H^2(F_t)$ is the parallel transport of $\eta$ along $\ell$.
This expresses the period as a path integral of a period of the fibre.
Such a line can be efficiently evaluated using the Picard--Fuchs equation of the period of $F_t$, and assuming one is able to evaluate initial conditions at some point \parencite{VanDerHoeven_1999, Mezzarobba_2010}.
In practice we use the implementation of \textcite{Mezzarobba_2016} in SageMath \parencite{sagemath} in the \texttt{ore\_algebra} package \parencite{KauersEtAl_2015}, which allows to recover certified precision bounds on the values of the integrals.
For further details on the computation of periods of extensions, see \textcite[\S3.7]{LairezEtAl_2024}.
All in all, this gives a way to compute the periods of $\Hpara_3(T/\bbP^1)$.

It follows from \propref{prop:fullrank} that the knowledge of the periods on $\Hpara_3(T/\bbP^1)$, $\operatorname{Sing}(T_\varepsilon/\bbP^1)$ and $\Lambda_{\rm vc}$ are sufficient to recover the periods on the entirety of $\Lambda_{\rm vc}^\perp$. 
As all the cycles are expressed in a same generating set of thimbles, one can effectively relate a basis of $H_3(T)$ to the above, which renders this approach feasible.
For more details, we refer to \textcite[\S3.3]{Pichon-Pharabod_2025}

\subsection{Picard--Fuchs equation and periods of a product of elliptic curves}
We now turn to the crux of this computation in our setting, that is the evaluation of the periods of the fibre and the computation of the Picard--Fuchs equation.

First recall that from the Künneth formula
\begin{equation}
H_2(F_b) \simeq H_2(S^\sindex{1})\otimes H_0(S^\sindex{2})\quad \oplus\quad H_0(S^\sindex{1})\otimes H_2(S^\sindex{2})\quad \oplus\quad H_1(S^\sindex{1})\otimes H_1(S^\sindex{2})\,,
\end{equation}
and that the monodromy acts nontrivially only on the last term.
The dual decomposition holds in cohomology, and so we may restrict to rational forms of the form
\begin{equation}
\omega_t = \omega_t^\sindex{1}\otimes\omega_t^\sindex{2} \in H^1(E_t^\sindex{1})\otimes H^1(E_t^\sindex{2})\subset H^2(F_t)\,.
\end{equation}
Similarly let $\eta_t = \gamma_t^\sindex{1}\otimes \gamma_t^\sindex{2} \in H_1(E_t^\sindex{1})\otimes H_1(E_t^\sindex{2})\subset H_2(F_t)$.

Then the period $\int_{\eta_t}\omega_t$ is simply equal to the product of the respective periods of $E_t^\sindex{1}$ and $E_t^\sindex{1}$:
\begin{equation}\label{eq:product_forms}
	\int_{\eta_t}\omega_t = \int_{\gamma_t^\sindex{1}\otimes\gamma_t^\sindex{2}}\omega_t^\sindex{1}\otimes\omega_t^\sindex{2} = \int_{\gamma_t^\sindex{1}}\omega_t^\sindex{1}\int_{\gamma_t^\sindex{2}}\omega_t^\sindex{2}\,.
\end{equation}
In particular it becomes apparent that the Picard--Fuchs operator $\Lop$ of $\omega_t$ is the \emph{symmetric product}~$\Lop^\sindex{1}\,\circledS\,\Lop^\sindex{2}$ of the Picard--Fuchs operators $\Lop^\sindex{i}$ of $\omega^\sindex{i}_t$ for $i=1,2$.

Let us further describe how to control the cohomology of $E^\sindex{i}_t$, which we assume to be given as a cubic hypersurface in $\bbP^2$ over the field $\C$ when $t$ is a fixed smooth point, and over $\C(t)$ when $t$ is kept as a parameter.
We have an isomorphism called the \emph{residue map} \parencite[\S2]{Griffiths_1969}
\begin{equation}
	\res\colon H^2(\bbP^2\setminus E^\sindex{i}_t) \xrightarrow{\sim} H^1(E^\sindex{i}_t)\,,
\end{equation}
which send an element of the form $\omega\wedge \ud f^\sindex{i}/f^\sindex{i}$ to $\omega$, where $f^\sindex{i}$ is the defining homogeneous cubic equation of $E^\sindex{i}_t$.
From a result of \textcite{Grothendieck_1966}, elements of the domain of this map can conveniently be described in terms of rational functions of the form
\begin{equation}
\frac{a}{(f^\sindex{i})^k}\Omega_2\,,
\end{equation}
where $k\ge 2$, $\Omega_2$ is the volume form of $\bbP^2$ and $a$ is a homogeneous polynomial with degree $3(k-1)$, modulo derivatives.
The Hodge filtration of $H^1(E^\sindex{i}_t)$ coincides with the filtration by the pole order on the rational functions.
Furthermore, Griffiths--Dwork reduction allows to reduce any such rational function to a unique canonical representative with minimal pole order.
For elliptic curves, we find a basis given by
\begin{equation}\label{eq:cohomology_basis_EC}
\omega^\sindex{i}_1 = \res\left(\frac{1}{f^\sindex{i}}\Omega_2\right) \qquad\quad \omega^\sindex{i}_2 = \res\left(\frac{x^3}{(f^\sindex{i})^2}\Omega_2\right)\,,
\end{equation}
with $\omega^\sindex{i}_1$ being the holomorphic form.
Working over $\C(t)$, the action of the \emph{Gauss-Manin connection} $\nabla_t$ on the cohomology is then simply given by formal derivation with respect to $t$ on the rational functions:
\begin{equation}
\nabla_t \res\left( \frac{a}{(f^\sindex{i})^k}\Omega_2 \right)= \res \left(\partial_t \frac{a}{(f^\sindex{i})^k}\Omega_2\right)\,.
\end{equation}
Using Griffiths--Dwork, one may then compute the Picard--Fuchs equations of such algebraic forms \parencite{Lairez_2016}.

Furthermore, the action of Griffiths-Dwork reduction on $\omega_t = \omega^\sindex{1}_t\otimes\omega^\sindex{2}_t$ is simply given by the product formula, as can be seen from \eqref{eq:product_forms}:
\begin{equation}
\nabla_t\,\omega_t = \nabla_t\,\omega^\sindex{1}_t\otimes\omega^\sindex{2}_t + \omega^\sindex{1}_t\otimes\nabla_t\,\omega^\sindex{2}_t\,,
\end{equation}
and thus the values of derivatives of $\int_{\eta_t}\omega_t$ can be recovered using the general Leibniz rule from the periods of $E^\sindex{i}_t$.
In practice we use the methods of \textcite{LairezEtAl_2024}, and more specifically the implementation in \texttt{lefschetz-family}\footnote{\url{https://github.com/ericpipha/lefschetz-family}}, to compute these periods numerically with high certified precision.

For more details on the residue map and Griffiths--Dwork reduction, we refer to \textcites[\S2]{Griffiths_1969}[\S5.3]{CoxKatz_1999}{Lairez_2016}.

\subsection{Holomorphic forms}\label{sec:holomorphic_form}

We now focus on the case where $S^\sindex{1}$ and $S^\sindex{2}$ are rational.
Let $f^\sindex{i}$ be the defining equation of $E_t^\sindex{i}$ as a homogeneous cubic polynomials in $\C[t][x,y,z]$.

A section of the holomorphic bundle $H^{1,0}(E^\sindex{i})$ over $\bbP^1$ is in particular given by $\omega^\sindex{i}_1 = \res\left(1/f^\sindex{i}\Omega_2\right)$, as mentioned in \eqref{eq:cohomology_basis_EC}.
One would be tempted to say that the holomorphic form of $T$ is then given by $\omega^\sindex{1}\otimes \omega^\sindex{2}\wedge \ud t$, but this is in general not true. 
A similar fact is true in the case of elliptic surfaces, see \textcite[\S4.1]{Pichon-Pharabod_2025}.
It is instead given by (a scalar multiple) of 
\begin{equation}
\omega = f(t) \,\omega_t^\sindex{1}\otimes \omega_t^\sindex{2} \wedge \ud t\,,
\end{equation}
where $f(t) \in \C(t)$ ensures that the resulting rational form has no residues.

In practice, to identify which $f(t)$ is allowed, we rely on ideas of \textcite{Stiller_1987} adapted to our setting.
In short, we start by computing the Picard--Fuchs equation $\Lop$ of $\omega_t = \omega_t^\sindex{1}\otimes \omega^\sindex{2}_t$.
The only points where $\omega$ could have a residue are above the singularities of $\Lop$.
Whether it has residues can be checked formally by computing formal log-power series expansions of the solutions of $\Lop$ at $t_0$, for $t_0$ such a singularity.
More specifically we want to check whether the coefficient in $(t-t_0)^{-1}$ of all solutions is zero which can be done formally using the Frobenius method.
When it is not, we can correct it by multiplying by a suitable power of $(t-t_0)$ to ensure that the residue does vanish.
Doing this consistently throughout all singular values of $\Lop$ yields the $f(t)\in\C(t)$ we yearned for.

\section{An application to the Gamma conjecture}\label{sec:gamma_class}

We now turn to Calabi--Yau fibre products carrying a motive of type $(1,1,1,1)$.
Such Calabi--Yau threefolds come in one-parameter families \parencite{VanStraten_2018}, and the associated Picard--Fuchs equation is a Calabi--Yau operator.
Our goal in this section is to identify a general shape for the Gamma-class formula (see below for the definition).
We first focus on \emph{Hadamard products}, a certain type of fibre products of elliptic surfaces, many of which incarnate motives of type $(1,1,1,1)$.

\subsection{Hadamard products}

Given two rational elliptic surfaces $S^\sindex{1}$ and $S^\sindex{2}$, we contruct a one-parameter family as follows:
Let $\varphi_u \colon t\mapsto u/t$.
Consider the fibre product $T_u = S^\sindex{1}\times_{\bbP^1}\varphi_u^* \,S^\sindex{2}$, where $\varphi_u^* \,S^\sindex{2}$ means only that the parameter of $S^\sindex{2}$ has been transformed by $\varphi$, so that the fibre above $t$ is $E^\sindex{2}_{\varphi(t)}$.
\begin{definition}
The \emph{Hadamard product} of $S^\sindex{1}$ and $S^\sindex{2}$ is the family of Calabi--Yau threefolds $S^\sindex{1}\times_{\bbP^1}\varphi_u^* \,S^\sindex{2}$. We denote it by $S^\sindex{1}\times_{u}S^\sindex{2}$.
\end{definition}
We note here that the fibres may be smoothed uniformly in $u$ by considering the family
\begin{equation}
T_u^\sindex{\varepsilon} :=S^\sindex{1}\times_{\bbP^1} {\left(\psi^\sindex{\varepsilon}\right)}^*\,\varphi_u^*S^\sindex{2}\,.
\end{equation}

\begin{remark}
The variable $u$ will always denote the parameter of the family of threefolds, to distinguish it from $t$, which corresponds to the parameter of the elliptic surfaces and their fibre product.
\end{remark}

We consider the following rational elliptic surfaces with section, given by the equations
\begin{equation}\label{eq:elliptic_surfaces}
\begin{aligned}
A&: y^{2} - y x  -t y = x^{3}+tx^2\,,\\
B&: y^{2} = x^{3} - \left(12 t - 1\right) x^{2} + 48 t^{2} x - 64 t^{3}\,,\\
C&: y^{2} = x^{3} + \left(144 t - 3\right) x -  144 t + 2\,,\\
D&: y^{2} = x^{3}-  3 x + 1728 t - 2\,,\\
a&: y^{2} - \left(2 t - 1\right) y x + 3 t^{2} x^{2} + 2 t^{3} y + \left(-3 t^{4}\right) x + t^{6}= x^{3}\,,\\
b&: y^{2} - \left(t +1 1\right) y x  -t y= x^{3} + t x^{2}\,,\\
c&: y^{2} - \left(3 t - 1\right) y x + 3 t^{2} x^{2} + 2 t^{3} y = x^{3}+ 3 t^{4} x - t^{6}\,,\\
d&: y^{2} - \left(4 t - 1\right) y x + 2 t^{2} x^{2} = x^{3} - 4 t^{4} x - (8t^2 - 8 t + 1)t^4\,,\\
e&: \left(16 t - 1\right) y^{2} - \left(16 t - 1\right) y x -t y = \left(16 t - 1\right) x^{3}  +t x^{2}\,,\\
f&:  y^{2} - \left(3 t - 1\right) y x + 9 t^{3} y = x^{3} - t^3 (6 t-1) \left(9 t^2-3 t+1\right)\,,\\
g&:  y^{2} - \left(6 t - 1\right) y x  -2 t^{3} y= x^{3} - 3 t^{2} x^{2} +3 t^{4} x - t^6\,,\\
h&:   y^{2} = 9 x^{3} - 3 (-1 + 3 t) (-1 + 27 t)^3 - 6 (27 t-1)^4 \left(27 t^2+18 t-1\right)\,,\\
i&: \left(64 t - 1\right) y^{2} = \left(64 t - 1\right) x^{3} - \left(48 t - 3\right) x -16 t - 2\,,\\
j&: \left(432 t - 1\right) y^{2}  = \left(432 t - 1\right) x^{3} - \left(1296 t - 3\right) x +864 t + 2\,.
 \end{aligned}
\end{equation}
These are \emph{extremal} rational elliptic surfaces with three or four singular fibres, in the sense of \textcite{Herfurtner_1991} (see \S4 therein), see also \textcite{MirandaPersson_1986}.
Their fibre configuration is given in \tabref{tab:fibre_configuations} --- we impose that there is a semi-stable fibre at $0$ (i.e., type $\mathit{I}_n$), so that the Hadamard product has a maximal unipotent monodromy point at 0. 
These fibre products are studied in \textcite{AlmkvistVanStraten_2023}, from where we copied the notation.

\begin{remark}
There is no particular reason to consider these specific rational surfaces, other than the expectation that with a low number of singular fibres, the rank of $\Hpara_3(T)$ will be small, which increases the chance of the transcendental part consisting only of a $(1,1,1,1)$ motive.
Of course there are many more elliptic surfaces which one could consider.
\end{remark}

\begin{table}[]
\centering
\begin{tabular}{ccccc}
\toprule
\multirow{2}{*}{Surface} & \multicolumn{3}{c}{Fibre types} &\multirow{2}{*}{\thead{Holomorphic\\ solution at 0}}\\ 
 &   at 0&at $\infty$ & others & \\ 
 \midrule
A& $\mathit{I}_4$ & $\mathit{I}_1^*$ & $\mathit{I}_1$ &$1, 4, 36, 400,\dots$\\
B& $\mathit{I}_3$ & $\mathit{IV}^*$ & $\mathit{I}_1$ &$1, 6, 90, 1680, \dots$\\
C& $\mathit{I}_2$ & $\mathit{III}^*$ & $\mathit{I}_1$ &$1, 12, 420, 18480, \dots$\\
D& $\mathit{I}_1$ & $\mathit{II}^*$ & $\mathit{I}_1$ &$1, 60, 13860, 4084080, \dots$\\
a& $\mathit{I}_6$ & $\mathit{I}_3$ & $\mathit{I}_2 + \mathit{I}_1$ &$1, 2, 10, 56, \dots$\\
b& $\mathit{I}_5$ & $\mathit{I}_5$ & $2\mathit{I}_1$ &$1, 3, 19, 147, \dots$\\
c& $\mathit{I}_6$ & $\mathit{I}_2$ & $\mathit{I}_3 + \mathit{I}_1$ &$1, 3, 15, 93, \dots$\\
d& $\mathit{I}_4$ & $\mathit{I}_2$ & $\mathit{I}_4 + \mathit{I}_2$ &$1, 4, 20, 112, \dots$\\
e& $\mathit{I}_4$ & $\mathit{I}_1$ & $\mathit{I}_1^*$&$1, 12, 164, 2352, \dots$\\
f& $\mathit{I}_3$ & $\mathit{I}_3$ & $2\mathit{I}_3$ &$1, 3, 9, 21, \dots$\\
g& $\mathit{I}_6$ & $\mathit{I}_1$ & $\mathit{I}_3 + \mathit{I}_2$ &$1, 6, 42, 312, \dots$\\
h& $\mathit{I}_3$ & $\mathit{I}_1$ & $\mathit{IV}^*$ &$1, 21, 495, 12171, \dots$\\
i& $\mathit{I}_2$ & $\mathit{I}_1$ & $\mathit{III}^*$ &$1, 52, 2980, 176848, \dots$\\
j& $\mathit{I}_1$ & $\mathit{I}_1$ & $\mathit{II}^*$ &$1, 372, 148644, 60907728, \dots$\\
  \bottomrule
\end{tabular}

\caption{The fibre configurations of the rational surfaces, and the Taylor expansion of the holomorphic period around $t=0$.}
\label{tab:fibre_configuations}
\end{table}

When $S^\sindex{1}=A$ and $S^\sindex{2}=b$, then it was shown in \textcite{GolyshevVanStraten_2023} that the fibres of the family $T_u$ admit a simultaneous crepant resolution, yielding a flat family of smooth Calabi--Yau threefolds with middle Hodge number $(1,1,1,1)$.
While it is not generally known whether such simultaneous resolutions can be performed, one may always simultaneously smooth the fibres, by applying the approach mentioned in the previous section to the whole family.
We therefore consider any pair $S^\sindex{1}, S^\sindex{2}\in \left\{A,B,C,D,a,b,c,d,e,f,g,h,i,j\right\}$.
In particular with this approach the motivic part sit in the homology lattice of the smoothing, which has the much bigger rank $40$.
We will see that it often does so in a non-unimodular manner.

The Picard--Fuchs equation $\Lop^{\rm Had}$ (in $u$) of this family is the \emph{Hadamard product} of $\Lop^\sindex{1}$ and $\Lop^\sindex{2}$, that is, the minimal differential operator of the Hadamard product 
\begin{equation}
	1+a^\sindex{1}_1a^\sindex{2}_1 u + a^\sindex{1}_2a^\sindex{2}_2 u^2 + \cdots
\end{equation} 
of the holomorphic solutions $1 + a_1^\sindex{i}t+ a_2^\sindex{i}t^2 + \cdots$ of $\Lop^\sindex{i}$ for $i=1,2$.
This is remarkable as it gives a motivic meaning to this Hadamard product, which \emph{a priori} is not a geometric operation.
For all these products, the Hadamard product $\Lop^{\rm Had}$ has order $4$, and has a maximal unipotent monodromy point at $0$.
In fact, it turns out that $\Lop^{\rm Had}$ is a Calabi--Yau operator --- we refer to \textcite{VanStraten_2018} and the references therein for the definition of Calabi--Yau operators, and to \textcite{AlmkvistVanStraten_2023} for a discussion about these specific operators, among others.

\begin{example}
We will rely on the family $T_u = A\times_u c$ as a running example.
In this case the fibres of $T_u$ are of type
\begin{equation}
\mathit{I}_4\times\mathit{I}_2\,,\qquad 
\mathit{I}_1^*\times\mathit{I}_6\,,\qquad
\mathit{I}_0\times\mathit{I}_3\,,\qquad
\mathit{I}_0\times\mathit{I}_1\quad
\text{ and }\qquad
\mathit{I}_1\times\mathit{I}_0
\end{equation}
for generic values of $u$.
$T_u$ acquires singularities precisely when $u=0$, $u=\infty$ or the running $I_1$ fibre of $A$ collides with the singular fibres of $c$.
We may compute that $\Lambda_{\rm vc}^\perp/\Lambda_{\rm vc}$ is unimodular of rank $8$, while $\Hpara_3(T_u)$ has rank 4 (as expected since the corresponding motive has rank 4).
We may compute the periods of the holomorphic differential $\omega_1$ of $T_{u_0}$ for a generic value $u_0\in \bbP^1$, as well as those of its first three derivatives $\nabla_u^k\, \omega$, $k=1,2,3$ using the methods presented in this paper.
We recover the $4\times 8$ period matrix.
\end{example}

\begin{definition}
We define the \emph{transcendental lattice} to be the saturation of $\Hpara_3(T_u)$ in $\Lambda_{\rm vs}^\perp$. We denote it $\Tr(T_u)$.
\end{definition}
It is a rank $4$ sublattice, and for generic values of $u$, it is the orthogonal complement of components of singular fibres, that is, cycles on which the holomorphic periods vanish\footnote{We do not exclude the possibility that for specific values of $u$, some periods vanish. A good analogy is the case of K3 surface, where the Néron-Severi lattice can be enhanced at non-generic values of the parameter.}.
By restricting the period matrix to $\Tr(T_1)$, we obtain a $4\times 4$ invertible matrix $\Pi(1)$ carrying the $(1,1,1,1)$-motive.
Since we have the Picard--Fuchs equation $\Lop^{\rm Had}$ of $\omega_u$ we may extend this matrix by analytic continuation to obtain a (multivalued) function $\Pi(u)$ of $u$.
The main point of this construction is that the monodromy will act on this matrix by multiplication by matrices with integer coefficients.
This follows from the fact that $\Hpara_3(T_u)$ is stable under monodromy, because monodromy acts by a braiding action on $\pi_1(\bbP^1\setminus\Sigma, b)$, and thus extensions are mapped to extensions.

Note that $\Tr(T_u)$ does not have to be a unimodular sublattice of $H_3(T_u)$ --- in most cases we consider, it is not.
\begin{example}
The transcendental lattice $\Tr(T_u)$ of $A\times_u c$ has intersection product given by the matrix
\begin{equation}
\left(\begin{array}{rrrr}0 & 0 & 1 & 0 \\0 & 0 & 0 & 3 \\-1 & 0 & 0 & 0 \\0 & -3 & 0 & 0\end{array}\right)\,.
\end{equation}
In the same basis, $\Hpara_3(T_u)$ is generated by the cycles
\begin{equation}
\left(1,\,0,\,0,\,-1\right)\,,\quad \left(0,\,-1,\,0,\,0\right)\,,\quad \left(0,\,0,\,-1,\,0\right)\,,\quad\text{and}\quad \left(0,\,0,\,0,\,4\right)
\end{equation}
and therefore it is a proper sublattice of index $4$ in $\Tr(T_u)$.
\end{example}

\subsection{The Gamma-class formula}

A basis of solutions of $\Lop^{\rm Had}$ around the MUM point $t=0$ is given by the \emph{scaled Frobenius basis}
\begin{equation} \label{eqn:oursolns}
    \varpi_j(u)=\frac{1}{(2\pi i)^j}\sum_{k=0}^j\binom{j}{k}f_k(u)\text{log}^{j-k}(u).
\end{equation}
for $j=0,1,2,3$, where $f_0$, $f_1$, $f_2$ and $f_3$ are in $\mathbb Q[[u]]$ \parencite{Frobenius_1873, Mezzarobba_2010}.

\begin{example}
For $T_u = A\times_u c$, we have
\begin{equation}
\begin{split}
    f_0(u) &= 1 + 12u + 540u^2 + 37200u^3 + \cdots\,,\\
    f_1(u) &= \hspace{1.65em}40 u + 2196 u^2 + \frac{485680}{3} u^3 + \cdots\,,\\
    f_2(u) &= \hspace{1.55em}16 u + 1920 u^2 + \frac{1551328}{9} u^3 +\cdots\,,\\
    f_3(u) &= \hspace{0.9em}-32 u - 1312 u^2 - \frac{1730624}{27} u^3  \cdots\,.
\end{split}
\end{equation}
\end{example}

Let $\varpi(u)=(\varpi_j(u))_j$ be the (column) vector of solutions and denote by $\Pi(u)$ the periods of $T_u$.
Then it was conjectured (and proven in certain cases) that we have \parencite{CandelasEtAl_1991a, Libgober_1999, KatzarkovEtAl_2008, Iritani_2009, HalversonEtAl_2015, CandelasEtAl_2020}
\begin{equation}
\Pi(u)^t = \rho\varpi(u)
\end{equation}
where 
\begin{equation}\label{eq:gamma_class}
    \rho= (2\pi i)^3\begin{pmatrix} 
    \lambda\chi & \frac{c_2\cdot H}{24} & 0 & \frac{H^3}{3!}\\
    \frac{c_2\cdot H}{24} & \frac{\sigma}{2} & -\frac{H^3}{2!} & 0\\
    1 & 0 & 0 & 0\\
    0 & 1 & 0 & 0
    \end{pmatrix}
\end{equation}
up to an integral change of basis of homology, and up to a rational factor, where $\lambda = \zeta(3)/(2\pi i)^3$ and
 $H^3$, $c_2\cdot H$, $\chi$ and $\sigma$ are integers.
An equation of this type is called a \emph{Gamma-class formula}.
In the case where the Calabi--Yau operator is hypergeometric, for example in the case of the famous quintic threefold \parencite{CandelasEtAl_1991}, geometric realisations $X_u$ are known and these integers are identified with topological invariants of the mirror Calabi--Yau family of $X_u$.
More precisely, $H^3$ corresponds to the triple intersection number of the hyperplane class, $c_2\cdot H$ is the second Chern class, $\chi$ is the Euler characteristic (i.e. the top Chern class) and $\sigma$ is either $0$ or $1$ depending on whether $H^3$ is even or odd respectively.	

Trying to apply this conjecture to the AESZ list by trying to find matching integers, one quickly finds out that this conjecture cannot hold in this state.
A more general formula was proposed in \textcite{KatzEtAl_2024, KatzSchimannek_2023, Schimannek_2025}, of the form
\begin{equation}\label{eq:gamma_class2}
    \rho= (2\pi i)^3\begin{pmatrix} 
    N^2 \lambda\chi + S & N\frac{c_2\cdot H}{24} & 0 & N^2\frac{H^3}{3!}\\
    \frac{c_2\cdot H}{24} & \frac{N\sigma}{2} & -N\frac{H^3}{2!} & 0\\
    1 & 0 & 0 & 0\\
    0 & N & 0 & 0
    \end{pmatrix}\,,
\end{equation}
where $N$ is an integer, and $S$ is another constant conjectured to be rational from the monodromy conjecture of \textcite[Conjecture 1]{VanStraten_2018}.

One of the punchlines of this present text is that we may compute the matrix $\rho$ numerically with hundreds of certified digits of accuracy, in reasonable time, for many different families.
Then using the LLL algorithm \parencite{LenstraEtAl_1982}, we can recover closed formulae for the entries of $\rho$ with a large level of confidence.

\begin{example}
The Gamma-class formula for $T_u = A\times_u c$  agrees with
\begin{equation}
\left(\begin{array}{rrrr}-112 \lambda & 0 & 0 & 4 \\0 & 0 & -12 & 0 \\1 & 0 & 0 & 0 \\0 & 3 & 0 & 0\end{array}\right)\,,
\end{equation}
with 150 certified decimal digits of precision.
It satisfies the conjectural form \eqref{eq:gamma_class2} with
\begin{equation}
\begin{gathered}
\chi=-112\,, \qquad c_2\cdot H=0 \,,\qquad H^3=24\\
S=0\,,\qquad \sigma=0\,, \qquad\text{and} \qquad N=3\,.
\end{gathered}
\end{equation}
\end{example}

\subsection{An extended Gamma-class formula}
Using the methods presented above, we have computed the Gamma-class formulae for all the 105 fibre products of elliptic surfaces considered.
These were all obtained with at least 150 decimal digits of certified precision.
The average computation time was 1 minutes 13 for each fibre product, with a maximum of 3 minutes 27 and a minimum of 32 seconds.

From this we state the following conjectural form for the Gamma-class, which is satisfied by all examples in our list.
\begin{conjecture}\label{conj:gamma_class}
The Gamma-class formula is of the form 
\begin{equation}\label{eq:gamma_class_conj}
    \rho= (2\pi i)^3\begin{pmatrix} 
    \chi\lambda - \frac{\alpha}{2} \frac{c_2\cdot H}{24} -\frac{\delta}{2} \, &\, M\frac{c_2\cdot H}{24} \,& \,\frac{\alpha}{2} \frac{H^3}{2!} \,& \,M\frac{H^3}{3!}\\[6pt]
    \frac{c_2\cdot H}{24} & N\frac{\sigma}{2} & -\frac{H^3}{2!} & 0\\[6pt]
    1 & 0 & 0 & 0\\[6pt]
    \frac{\alpha}{2} \frac{N}{M} & N & 0 & 0
    \end{pmatrix}\,.
\end{equation}
where $\alpha, \delta, \sigma\in\{0,1\}$ and $c_2H, H^3, \chi\in \Z$ and $M,N\in\N$, and with intersection product
\begin{equation}
\left(\begin{array}{rrrr}0 & 0 & M & 0 \\0 & 0 & 0 & N \\-M & 0 & 0 & 0 \\0 & -N & 0 & 0\end{array}\right)\,.
\end{equation}
The corresponding invariants for all the families of elliptic surfaces considered here are given in \tabref{tab:gamma_class_invariants}.
Matching invariants for all the irreducible operators of the Calabi--Yau database with integral monodromy of degree less than 20 are given in \tabref{tab:cydb_gamma_class}.
To the best of the author's knowledge, the invariants $\alpha$ and $\delta$ do not appear in the literature.
\end{conjecture}
In particular the case $M=N=1$ and $\alpha=\delta=0$ coincides with \eqref{eq:gamma_class}.
These invariant are not unique.
In particular when $\alpha=0$, $c_2\cdot H$ is only defined up to $24N/\operatorname{gcd}(N,M)$.
When $\alpha\ne 0$ this can also be extended up to allowing $\alpha$ to be any odd integer.

\begin{remark}
Although we keep the notations of \eqref{eq:gamma_class}, we do not make any statement about the meaning of the invariants.
\end{remark}

\begin{remark}
We remark that every single possible combination of $(\alpha, \delta, \sigma) \in \{0,1\}^3$ appears in \tabref{tab:gamma_class_invariants}.
\end{remark}

We have checked \conjref{conj:gamma_class} on the examples of the CYDB with degree less than 20, and obtained matches for all the irreducible operators of the database that admit integral monodromy. 
Furthermore the intersection product mentioned in \conjref{conj:gamma_class} is invariant under monodromy in these families, which implies it is correct up to a scalar.
The results are compiled in \tabref{tab:cydb_gamma_class}.
Some operators have monodromy with coefficients in a number field, and will require further care to figure out ---  this, along with the details of the computation of \tabref{tab:cydb_gamma_class}, is work in preparation.

We remark that for 354 of these 613 operators, we find $N=M$.
This means the operators have integral $\operatorname{Sp}_4(\Z)$-monodromy, that is, monodromy preserving the intersection product
\begin{equation}
\left(\begin{array}{rrrr}0 & 0 & 1 & 0 \\0 & 0 & 0 & 1 \\-1 & 0 & 0 & 0 \\0 & -1 & 0 & 0\end{array}\right)
\end{equation}
in some basis for which the monodromy has integer coefficients.
In \tabref{tab:cydb_gamma_class}, we have tried to set $N$ as low as possible while maintaining integral monodromy.
In all cases, $M$ divides $N$.

\vfill
\begin{center}
    \textit{(continued on next page)}
\end{center}
\vfill

\begingroup
\footnotesize
\begin{table}[]
\centering
\caption{Gamma-class invariants for the 105 Hadamard products of the elliptic surfaces given in \eqref{eq:elliptic_surfaces}.}
\label{tab:gamma_class_invariants}
\begin{minipage}{0.48\textwidth}
\begin{tabular}{ccccccccc}
\toprule
$T_u$ & $\chi$ & $c_2\cdot H$ & $H^3$ & $\sigma$ &$\alpha$ & $\delta$ &  $N$ & $M$ \\
\midrule
$A\times_u A$ & $-128$ & $184$ & $16$ & $0$ & $0$ & $0$ & $1$ & $1$\\
$A\times_u B$ & $-144$ & $12$ & $12$ & $0$ & $0$ & $0$ & $1$ & $1$\\
$A\times_u C$ & $-176$ & $8$ & $8$ & $0$ & $0$ & $0$ & $1$ & $1$\\
$A\times_u D$ & $-256$ & $4$ & $4$ & $0$ & $0$ & $0$ & $1$ & $1$\\
$A\times_u a$ & $-120$ & $24$ & $24$ & $0$ & $0$ & $0$ & $2$ & $1$\\
$A\times_u b$ & $-120$ & $20$ & $20$ & $0$ & $0$ & $0$ & $1$ & $1$\\
$A\times_u c$ & $-112$ & $0$ & $24$ & $0$ & $0$ & $0$ & $3$ & $1$\\
$A\times_u d$ & $-88$ & $40$ & $16$ & $0$ & $0$ & $0$ & $4$ & $2$\\
$A\times_u e$ & $96$ & $112$ & $16$ & $0$ & $1$ & $0$ & $4$ & $1$\\
$A\times_u f$ & $-120$ & $32$ & $12$ & $0$ & $1$ & $1$ & $3$ & $3$\\
$A\times_u g$ & $-8$ & $-96$ & $24$ & $0$ & $1$ & $0$ & $6$ & $1$\\
$A\times_u h$ & $168$ & $216$ & $12$ & $0$ & $1$ & $1$ & $3$ & $1$\\
$A\times_u i$ & $272$ & $32$ & $8$ & $0$ & $1$ & $0$ & $2$ & $1$\\
$A\times_u j$ & $472$ & $64$ & $4$ & $0$ & $1$ & $1$ & $1$ & $1$\\
$B\times_u B$ & $-144$ & $54$ & $9$ & $1$ & $0$ & $0$ & $3$ & $1$\\
$B\times_u C$ & $-156$ & $48$ & $6$ & $0$ & $0$ & $0$ & $3$ & $1$\\
$B\times_u D$ & $-204$ & $42$ & $3$ & $1$ & $0$ & $0$ & $3$ & $1$\\
$B\times_u a$ & $-162$ & $72$ & $18$ & $0$ & $0$ & $0$ & $6$ & $1$\\
$B\times_u b$ & $-150$ & $66$ & $15$ & $1$ & $0$ & $0$ & $3$ & $1$\\
$B\times_u c$ & $-156$ & $0$ & $18$ & $0$ & $0$ & $0$ & $3$ & $1$\\
$B\times_u d$ & $-162$ & $42$ & $12$ & $0$ & $0$ & $0$ & $12$ & $2$\\
$B\times_u e$ & $24$ & $24$ & $12$ & $0$ & $1$ & $0$ & $12$ & $1$\\
$B\times_u f$ & $-198$ & $81$ & $9$ & $1$ & $1$ & $1$ & $3$ & $3$\\
$B\times_u g$ & $-78$ & $54$ & $18$ & $0$ & $1$ & $0$ & $6$ & $1$\\
$B\times_u h$ & $90$ & $-63$ & $9$ & $1$ & $1$ & $1$ & $3$ & $1$\\
$B\times_u i$ & $180$ & $-6$ & $6$ & $0$ & $1$ & $0$ & $6$ & $1$\\
$B\times_u j$ & $342$ & $51$ & $3$ & $1$ & $1$ & $1$ & $3$ & $1$\\
$C\times_u C$ & $-144$ & $40$ & $4$ & $0$ & $0$ & $0$ & $2$ & $1$\\
$C\times_u D$ & $-156$ & $32$ & $2$ & $0$ & $0$ & $0$ & $2$ & $1$\\
$C\times_u a$ & $-228$ & $24$ & $12$ & $0$ & $0$ & $0$ & $2$ & $1$\\
$C\times_u b$ & $-200$ & $16$ & $10$ & $0$ & $0$ & $0$ & $2$ & $1$\\
$C\times_u c$ & $-224$ & $72$ & $12$ & $0$ & $0$ & $0$ & $6$ & $1$\\
$C\times_u d$ & $-268$ & $44$ & $8$ & $0$ & $0$ & $0$ & $4$ & $2$\\
$C\times_u e$ & $-64$ & $32$ & $8$ & $0$ & $1$ & $0$ & $4$ & $1$\\
$C\times_u f$ & $-312$ & $82$ & $6$ & $0$ & $1$ & $1$ & $6$ & $3$\\
$C\times_u g$ & $-172$ & $204$ & $12$ & $0$ & $1$ & $0$ & $6$ & $1$\\
$C\times_u h$ & $0$ & $-126$ & $6$ & $0$ & $1$ & $1$ & $6$ & $1$\\
$C\times_u i$ & $80$ & $52$ & $4$ & $0$ & $1$ & $0$ & $2$ & $1$\\
$C\times_u j$ & $208$ & $-10$ & $2$ & $0$ & $1$ & $1$ & $2$ & $1$\\
$D\times_u D$ & $-120$ & $22$ & $1$ & $1$ & $0$ & $0$ & $1$ & $1$\\
$D\times_u a$ & $-366$ & $24$ & $6$ & $0$ & $0$ & $0$ & $2$ & $1$\\
$D\times_u b$ & $-310$ & $14$ & $5$ & $1$ & $0$ & $0$ & $1$ & $1$\\
$D\times_u c$ & $-364$ & $0$ & $6$ & $0$ & $0$ & $0$ & $3$ & $1$\\
$D\times_u d$ & $-470$ & $46$ & $4$ & $0$ & $0$ & $0$ & $4$ & $2$\\
$D\times_u e$ & $-200$ & $40$ & $4$ & $0$ & $1$ & $0$ & $4$ & $1$\\
$D\times_u f$ & $-534$ & $59$ & $3$ & $1$ & $1$ & $1$ & $3$ & $3$\\
$D\times_u g$ & $-338$ & $66$ & $6$ & $0$ & $1$ & $0$ & $6$ & $1$\\
$D\times_u h$ & $-126$ & $27$ & $3$ & $1$ & $1$ & $1$ & $3$ & $1$\\
$D\times_u i$ & $-44$ & $14$ & $2$ & $0$ & $1$ & $0$ & $2$ & $1$\\
$D\times_u j$ & $62$ & $25$ & $1$ & $1$ & $1$ & $1$ & $1$ & $1$\\
$a\times_u a$ & $-72$ & $24$ & $36$ & $0$ & $0$ & $0$ & $2$ & $1$\\
$a\times_u b$ & $-90$ & $24$ & $30$ & $0$ & $0$ & $0$ & $2$ & $1$\\
$a\times_u c$ & $-60$ & $72$ & $36$ & $0$ & $0$ & $0$ & $6$ & $1$\\
  \bottomrule
\end{tabular}
\end{minipage}%
\hfill
\begin{minipage}{0.48\textwidth}
\begin{tabular}{ccccccccc}
\toprule
$T_u$ & $\chi$ & $c_2\cdot H$ & $H^3$ & $\sigma$ &$\alpha$ & $\delta$ &  $N$ & $M$ \\
\midrule
$a\times_u d$ & $12$ & $36$ & $24$ & $0$ & $0$ & $0$ & $4$ & $2$\\
$a\times_u e$ & $216$ & $0$ & $24$ & $0$ & $1$ & $0$ & $4$ & $1$\\
$a\times_u f$ & $-18$ & $30$ & $18$ & $0$ & $1$ & $1$ & $6$ & $3$\\
$a\times_u g$ & $96$ & $-396$ & $36$ & $0$ & $1$ & $0$ & $6$ & $1$\\
$a\times_u h$ & $306$ & $-18$ & $18$ & $0$ & $1$ & $1$ & $6$ & $1$\\
$a\times_u i$ & $444$ & $12$ & $12$ & $0$ & $1$ & $0$ & $2$ & $1$\\
$a\times_u j$ & $726$ & $42$ & $6$ & $0$ & $1$ & $1$ & $2$ & $1$\\
$b\times_u b$ & $-100$ & $22$ & $25$ & $1$ & $0$ & $0$ & $1$ & $1$\\
$b\times_u c$ & $-80$ & $0$ & $30$ & $0$ & $0$ & $0$ & $3$ & $1$\\
$b\times_u d$ & $-30$ & $38$ & $20$ & $0$ & $0$ & $0$ & $4$ & $2$\\
$b\times_u e$ & $160$ & $8$ & $20$ & $0$ & $1$ & $0$ & $4$ & $1$\\
$b\times_u f$ & $-60$ & $127$ & $15$ & $1$ & $1$ & $1$ & $3$ & $3$\\
$b\times_u g$ & $50$ & $-390$ & $30$ & $0$ & $1$ & $0$ & $6$ & $1$\\
$b\times_u h$ & $240$ & $63$ & $15$ & $1$ & $1$ & $1$ & $3$ & $1$\\
$b\times_u i$ & $360$ & $118$ & $10$ & $0$ & $1$ & $0$ & $2$ & $1$\\
$b\times_u j$ & $600$ & $-115$ & $5$ & $1$ & $1$ & $1$ & $1$ & $1$\\
$c\times_u c$ & $-48$ & $0$ & $36$ & $0$ & $0$ & $0$ & $3$ & $1$\\
$c\times_u d$ & $28$ & $36$ & $24$ & $0$ & $0$ & $0$ & $12$ & $2$\\
$c\times_u e$ & $224$ & $0$ & $24$ & $0$ & $1$ & $0$ & $12$ & $1$\\
$c\times_u f$ & $0$ & $54$ & $18$ & $0$ & $1$ & $1$ & $3$ & $3$\\
$c\times_u g$ & $108$ & $36$ & $36$ & $0$ & $1$ & $0$ & $6$ & $1$\\
$c\times_u h$ & $312$ & $198$ & $18$ & $0$ & $1$ & $1$ & $3$ & $1$\\
$c\times_u i$ & $448$ & $108$ & $12$ & $0$ & $1$ & $0$ & $6$ & $1$\\
$c\times_u j$ & $728$ & $-54$ & $6$ & $0$ & $1$ & $1$ & $3$ & $1$\\
$d\times_u d$ & $80$ & $16$ & $16$ & $0$ & $0$ & $0$ & $4$ & $2$\\
$d\times_u e$ & $360$ & $40$ & $16$ & $0$ & $0$ & $0$ & $4$ & $2$\\
$d\times_u f$ & $138$ & $38$ & $12$ & $0$ & $0$ & $1$ & $12$ & $6$\\
$d\times_u g$ & $236$ & $12$ & $24$ & $0$ & $0$ & $0$ & $12$ & $2$\\
$d\times_u h$ & $462$ & $54$ & $12$ & $0$ & $0$ & $1$ & $12$ & $2$\\
$d\times_u i$ & $628$ & $20$ & $8$ & $0$ & $0$ & $0$ & $4$ & $2$\\
$d\times_u j$ & $986$ & $34$ & $4$ & $0$ & $0$ & $1$ & $4$ & $2$\\
$e\times_u e$ & $320$ & $64$ & $16$ & $0$ & $0$ & $0$ & $4$ & $1$\\
$e\times_u i$ & $384$ & $8$ & $8$ & $0$ & $0$ & $0$ & $4$ & $1$\\
$e\times_u j$ & $528$ & $28$ & $4$ & $0$ & $0$ & $1$ & $4$ & $1$\\
$f\times_u e$ & $384$ & $20$ & $12$ & $0$ & $0$ & $1$ & $12$ & $3$\\
$f\times_u f$ & $36$ & $6$ & $9$ & $1$ & $0$ & $0$ & $3$ & $3$\\
$f\times_u g$ & $234$ & $0$ & $18$ & $0$ & $0$ & $1$ & $6$ & $3$\\
$f\times_u h$ & $504$ & $6$ & $9$ & $1$ & $0$ & $0$ & $3$ & $3$\\
$f\times_u i$ & $696$ & $40$ & $6$ & $0$ & $0$ & $1$ & $6$ & $3$\\
$f\times_u j$ & $1104$ & $2$ & $3$ & $1$ & $0$ & $0$ & $3$ & $3$\\
$g\times_u e$ & $328$ & $120$ & $24$ & $0$ & $0$ & $0$ & $12$ & $1$\\
$g\times_u g$ & $264$ & $72$ & $36$ & $0$ & $0$ & $0$ & $6$ & $1$\\
$g\times_u h$ & $390$ & $0$ & $18$ & $0$ & $0$ & $1$ & $6$ & $1$\\
$g\times_u i$ & $500$ & $24$ & $12$ & $0$ & $0$ & $0$ & $6$ & $1$\\
$g\times_u j$ & $754$ & $48$ & $6$ & $0$ & $0$ & $1$ & $6$ & $1$\\
$h\times_u e$ & $336$ & $180$ & $12$ & $0$ & $0$ & $1$ & $12$ & $1$\\
$h\times_u h$ & $324$ & $54$ & $9$ & $1$ & $0$ & $0$ & $3$ & $1$\\
$h\times_u i$ & $336$ & $72$ & $6$ & $0$ & $0$ & $1$ & $6$ & $1$\\
$h\times_u j$ & $420$ & $18$ & $3$ & $1$ & $0$ & $0$ & $3$ & $1$\\
$i\times_u i$ & $304$ & $40$ & $4$ & $0$ & $0$ & $0$ & $2$ & $1$\\
$i\times_u j$ & $320$ & $8$ & $2$ & $0$ & $0$ & $1$ & $2$ & $1$\\
$j\times_u j$ & $244$ & $22$ & $1$ & $1$ & $0$ & $0$ & $1$ & $1$\\
\\
  \bottomrule
\end{tabular}
\end{minipage}
\end{table}


\begin{adjustwidth}{-10pt}{0cm}
\setlength{\columnsep}{30pt}
\twocolumn
\begin{strip}
\captionof{table}{ The Gamma-class invariants of the irreducible Calabi--Yau operators of order 4 of the CYDB with integral monodromy, up to degree 20, as well as $4.24.3$} \label{tab:cydb_gamma_class}
\end{strip}
\begingroup
\footnotesize
\tablehead{\toprule New Number & $\chi$ & $c_2\cdot H$ & $H^3$ & $\sigma$ &$\alpha$ & $\delta$ &  $N$ & $M$\\\midrule}
\tabletail{\bottomrule}
\begin{supertabular}[c]{@{}ccccccccc@{}}
$4\,.\,1\,.\,1$ & $-200$ & $2$ & $5$ & $1$ & $0$ & $0$ & $1$ & $1$\\
$4\,.\,1\,.\,2$ & $-288$ & $10$ & $1$ & $1$ & $0$ & $0$ & $1$ & $1$\\
$4\,.\,1\,.\,3$ & $-128$ & $16$ & $16$ & $0$ & $0$ & $0$ & $1$ & $1$\\
$4\,.\,1\,.\,4$ & $-144$ & $6$ & $9$ & $1$ & $0$ & $0$ & $1$ & $1$\\
$4\,.\,1\,.\,5$ & $-144$ & $12$ & $12$ & $0$ & $0$ & $0$ & $1$ & $1$\\
$4\,.\,1\,.\,6$ & $-176$ & $8$ & $8$ & $0$ & $0$ & $0$ & $1$ & $1$\\
$4\,.\,1\,.\,7$ & $-296$ & $20$ & $2$ & $0$ & $0$ & $0$ & $1$ & $1$\\
$4\,.\,1\,.\,8$ & $-204$ & $18$ & $3$ & $1$ & $0$ & $0$ & $1$ & $1$\\
$4\,.\,1\,.\,9$ & $-484$ & $22$ & $1$ & $1$ & $0$ & $0$ & $1$ & $1$\\
$4\,.\,1\,.\,10$ & $-144$ & $16$ & $4$ & $0$ & $0$ & $0$ & $1$ & $1$\\
$4\,.\,1\,.\,11$ & $-156$ & $0$ & $6$ & $0$ & $0$ & $0$ & $1$ & $1$\\
$4\,.\,1\,.\,12$ & $-156$ & $8$ & $2$ & $0$ & $0$ & $0$ & $1$ & $1$\\
$4\,.\,1\,.\,13$ & $-120$ & $22$ & $1$ & $1$ & $0$ & $0$ & $1$ & $1$\\
$4\,.\,1\,.\,14$ & $-256$ & $4$ & $4$ & $0$ & $0$ & $0$ & $1$ & $1$\\
$4\,.\,2\,.\,1$ & $-120$ & $24$ & $24$ & $0$ & $0$ & $0$ & $2$ & $1$\\
$4\,.\,2\,.\,2$ & $-162$ & $24$ & $18$ & $1$ & $0$ & $0$ & $2$ & $1$\\
$4\,.\,2\,.\,3$ & $-228$ & $24$ & $12$ & $0$ & $0$ & $0$ & $2$ & $1$\\
$4\,.\,2\,.\,4$ & $-366$ & $24$ & $6$ & $1$ & $0$ & $0$ & $2$ & $1$\\
$4\,.\,2\,.\,5$ & $-120$ & $20$ & $20$ & $0$ & $0$ & $0$ & $1$ & $1$\\
$4\,.\,2\,.\,6$ & $-150$ & $18$ & $15$ & $1$ & $0$ & $0$ & $1$ & $1$\\
$4\,.\,2\,.\,7$ & $-200$ & $16$ & $10$ & $0$ & $0$ & $0$ & $1$ & $1$\\
$4\,.\,2\,.\,8$ & $-310$ & $14$ & $5$ & $1$ & $0$ & $0$ & $1$ & $1$\\
$4\,.\,2\,.\,9$ & $-112$ & $0$ & $24$ & $0$ & $0$ & $0$ & $3$ & $1$\\
$4\,.\,2\,.\,10$ & $-156$ & $0$ & $18$ & $0$ & $0$ & $0$ & $3$ & $1$\\
$4\,.\,2\,.\,11$ & $-224$ & $0$ & $12$ & $0$ & $0$ & $0$ & $3$ & $1$\\
$4\,.\,2\,.\,12$ & $-364$ & $0$ & $6$ & $0$ & $0$ & $0$ & $3$ & $1$\\
$4\,.\,2\,.\,13$ & $-88$ & $40$ & $16$ & $0$ & $0$ & $0$ & $4$ & $2$\\
$4\,.\,2\,.\,14$ & $-162$ & $42$ & $12$ & $0$ & $0$ & $0$ & $4$ & $2$\\
$4\,.\,2\,.\,15$ & $-268$ & $44$ & $8$ & $0$ & $0$ & $0$ & $4$ & $2$\\
$4\,.\,2\,.\,16$ & $-470$ & $46$ & $4$ & $0$ & $0$ & $0$ & $4$ & $2$\\
$4\,.\,2\,.\,17$ & $96$ & $400$ & $16$ & $0$ & $1$ & $0$ & $4$ & $1$\\
$4\,.\,2\,.\,18$ & $24$ & $312$ & $12$ & $0$ & $1$ & $0$ & $4$ & $1$\\
$4\,.\,2\,.\,19$ & $-200$ & $136$ & $4$ & $0$ & $1$ & $0$ & $4$ & $1$\\
$4\,.\,2\,.\,20$ & $-120$ & $992$ & $12$ & $0$ & $1$ & $1$ & $3$ & $3$\\
$4\,.\,2\,.\,21$ & $-198$ & $33$ & $9$ & $0$ & $1$ & $1$ & $3$ & $3$\\
$4\,.\,2\,.\,22$ & $-312$ & $514$ & $6$ & $0$ & $1$ & $1$ & $3$ & $3$\\
$4\,.\,2\,.\,23$ & $-534$ & $35$ & $3$ & $0$ & $1$ & $1$ & $3$ & $3$\\
$4\,.\,2\,.\,24$ & $-8$ & $624$ & $24$ & $0$ & $1$ & $0$ & $6$ & $1$\\
$4\,.\,2\,.\,25$ & $-78$ & $486$ & $18$ & $0$ & $1$ & $0$ & $6$ & $1$\\
$4\,.\,2\,.\,26$ & $-172$ & $348$ & $12$ & $0$ & $1$ & $0$ & $6$ & $1$\\
$4\,.\,2\,.\,27$ & $-338$ & $210$ & $6$ & $0$ & $1$ & $0$ & $6$ & $1$\\
$4\,.\,2\,.\,28$ & $168$ & $288$ & $12$ & $0$ & $1$ & $1$ & $3$ & $1$\\
$4\,.\,2\,.\,29$ & $90$ & $225$ & $9$ & $0$ & $1$ & $1$ & $3$ & $1$\\
$4\,.\,2\,.\,30$ & $-126$ & $99$ & $3$ & $0$ & $1$ & $1$ & $3$ & $1$\\
$4\,.\,2\,.\,31$ & $272$ & $176$ & $8$ & $0$ & $1$ & $0$ & $2$ & $1$\\
$4\,.\,2\,.\,32$ & $180$ & $138$ & $6$ & $0$ & $1$ & $0$ & $2$ & $1$\\
$4\,.\,2\,.\,33$ & $80$ & $100$ & $4$ & $0$ & $1$ & $0$ & $2$ & $1$\\
$4\,.\,2\,.\,34$ & $-44$ & $62$ & $2$ & $0$ & $1$ & $0$ & $2$ & $1$\\
$4\,.\,2\,.\,35$ & $472$ & $64$ & $4$ & $0$ & $1$ & $1$ & $1$ & $1$\\
$4\,.\,2\,.\,36$ & $342$ & $51$ & $3$ & $0$ & $1$ & $1$ & $1$ & $1$\\
$4\,.\,2\,.\,37$ & $208$ & $38$ & $2$ & $0$ & $1$ & $1$ & $1$ & $1$\\
$4\,.\,2\,.\,38$ & $62$ & $25$ & $1$ & $0$ & $1$ & $1$ & $1$ & $1$\\
$4\,.\,2\,.\,40$ & $580$ & $-11$ & $1$ & $0$ & $1$ & $0$ & $1$ & $1$\\
$4\,.\,2\,.\,41$ & $324$ & $51$ & $3$ & $0$ & $1$ & $0$ & $3$ & $3$\\
$4\,.\,2\,.\,46$ & $972$ & $-23$ & $1$ & $0$ & $1$ & $0$ & $1$ & $1$\\
$4\,.\,2\,.\,47$ & $304$ & $2$ & $2$ & $0$ & $0$ & $0$ & $2$ & $2$\\
$4\,.\,2\,.\,50$ & $244$ & $1$ & $1$ & $0$ & $1$ & $0$ & $1$ & $1$\\
$4\,.\,2\,.\,51$ & $528$ & $20$ & $2$ & $0$ & $0$ & $0$ & $2$ & $2$\\
$4\,.\,2\,.\,52$ & $-128$ & $0$ & $48$ & $0$ & $0$ & $0$ & $4$ & $1$\\
$4\,.\,2\,.\,53$ & $-116$ & $0$ & $24$ & $0$ & $0$ & $0$ & $1$ & $1$\\
$4\,.\,2\,.\,54$ & $-180$ & $48$ & $24$ & $0$ & $1$ & $0$ & $3$ & $3$\\
$4\,.\,2\,.\,55$ & $-116$ & $8$ & $32$ & $0$ & $0$ & $0$ & $1$ & $1$\\
$4\,.\,2\,.\,56$ & $-120$ & $12$ & $36$ & $0$ & $0$ & $0$ & $1$ & $1$\\
$4\,.\,2\,.\,57$ & $-200$ & $128$ & $20$ & $0$ & $1$ & $1$ & $5$ & $5$\\
$4\,.\,2\,.\,58$ & $36$ & $96$ & $12$ & $0$ & $1$ & $1$ & $3$ & $3$\\
$4\,.\,2\,.\,59$ & $108$ & $64$ & $4$ & $0$ & $1$ & $1$ & $1$ & $1$\\
$4\,.\,2\,.\,60$ & $-128$ & $40$ & $40$ & $0$ & $0$ & $0$ & $2$ & $1$\\
$4\,.\,2\,.\,61$ & $-116$ & $4$ & $28$ & $0$ & $0$ & $0$ & $1$ & $1$\\
$4\,.\,2\,.\,62$ & $-96$ & $12$ & $42$ & $0$ & $0$ & $0$ & $1$ & $1$\\
$4\,.\,2\,.\,63$ & $96$ & $144$ & $48$ & $0$ & $0$ & $0$ & $8$ & $1$\\
$4\,.\,2\,.\,64$ & $-96$ & $66$ & $66$ & $0$ & $0$ & $0$ & $6$ & $2$\\
$4\,.\,2\,.\,65$ & $-128$ & $16$ & $24$ & $0$ & $1$ & $0$ & $6$ & $3$\\
$4\,.\,2\,.\,69$ & $-128$ & $40$ & $40$ & $0$ & $0$ & $0$ & $8$ & $4$\\
$4\,.\,2\,.\,71$ & $80$ & $100$ & $16$ & $0$ & $1$ & $0$ & $4$ & $2$\\
$4\,.\,3\,.\,1$ & $-80$ & $0$ & $120$ & $0$ & $0$ & $0$ & $5$ & $1$\\
$4\,.\,3\,.\,2$ & $160$ & $-21$ & $3$ & $0$ & $1$ & $0$ & $1$ & $1$\\
$4\,.\,3\,.\,3$ & $136$ & $-12$ & $12$ & $0$ & $1$ & $0$ & $2$ & $1$\\
$4\,.\,3\,.\,4$ & $108$ & $9$ & $9$ & $0$ & $1$ & $0$ & $9$ & $9$\\
$4\,.\,3\,.\,5$ & $-88$ & $12$ & $18$ & $0$ & $0$ & $0$ & $1$ & $1$\\
$4\,.\,3\,.\,7$ & $12$ & $6$ & $9$ & $1$ & $0$ & $0$ & $1$ & $1$\\
$4\,.\,3\,.\,8$ & $-72$ & $12$ & $18$ & $1$ & $0$ & $0$ & $2$ & $2$\\
$4\,.\,3\,.\,9$ & $-100$ & $22$ & $25$ & $1$ & $0$ & $0$ & $1$ & $1$\\
$4\,.\,3\,.\,10$ & $-48$ & $0$ & $12$ & $0$ & $0$ & $0$ & $3$ & $3$\\
$4\,.\,3\,.\,11$ & $264$ & $12$ & $6$ & $1$ & $0$ & $0$ & $6$ & $6$\\
$4\,.\,3\,.\,15$ & $-58$ & $100$ & $16$ & $0$ & $1$ & $1$ & $1$ & $1$\\
$4\,.\,3\,.\,16$ & $144$ & $12$ & $6$ & $0$ & $0$ & $0$ & $3$ & $3$\\
$4\,.\,3\,.\,17$ & $80$ & $12$ & $6$ & $1$ & $0$ & $0$ & $2$ & $2$\\
$4\,.\,3\,.\,18$ & $-64$ & $12$ & $6$ & $0$ & $0$ & $0$ & $1$ & $1$\\
$4\,.\,3\,.\,19$ & $300$ & $14$ & $5$ & $1$ & $0$ & $0$ & $5$ & $5$\\
$4\,.\,3\,.\,20$ & $224$ & $10$ & $7$ & $1$ & $0$ & $0$ & $7$ & $7$\\
$4\,.\,3\,.\,24$ & $-148$ & $272$ & $56$ & $0$ & $1$ & $0$ & $2$ & $1$\\
$4\,.\,3\,.\,25$ & $-152$ & $288$ & $60$ & $0$ & $1$ & $1$ & $3$ & $1$\\
$4\,.\,3\,.\,26$ & $40$ & $4$ & $10$ & $0$ & $0$ & $0$ & $5$ & $5$\\
$4\,.\,3\,.\,31$ & $-128$ & $20$ & $44$ & $0$ & $0$ & $0$ & $1$ & $1$\\
$4\,.\,3\,.\,32$ & $-8$ & $8$ & $8$ & $0$ & $0$ & $0$ & $1$ & $1$\\
$4\,.\,3\,.\,33$ & $192$ & $20$ & $12$ & $1$ & $0$ & $0$ & $12$ & $12$\\
$4\,.\,3\,.\,34$ & $160$ & $8$ & $8$ & $1$ & $0$ & $0$ & $8$ & $8$\\
$4\,.\,4\,.\,5$ & $384$ & $4$ & $6$ & $0$ & $0$ & $0$ & $6$ & $6$\\
$4\,.\,4\,.\,6$ & $160$ & $22$ & $10$ & $1$ & $0$ & $1$ & $2$ & $2$\\
$4\,.\,4\,.\,7$ & $312$ & $0$ & $6$ & $0$ & $0$ & $0$ & $3$ & $3$\\
$4\,.\,4\,.\,15$ & $444$ & $12$ & $6$ & $1$ & $0$ & $0$ & $2$ & $2$\\
$4\,.\,4\,.\,16$ & $628$ & $22$ & $4$ & $0$ & $0$ & $0$ & $4$ & $4$\\
$4\,.\,4\,.\,23$ & $1104$ & $-37$ & $3$ & $0$ & $1$ & $0$ & $3$ & $3$\\
$4\,.\,4\,.\,24$ & $600$ & $14$ & $5$ & $1$ & $0$ & $1$ & $1$ & $1$\\
$4\,.\,4\,.\,33$ & $-88$ & $0$ & $12$ & $0$ & $0$ & $0$ & $2$ & $1$\\
$4\,.\,4\,.\,34$ & $-120$ & $10$ & $13$ & $1$ & $0$ & $0$ & $1$ & $1$\\
$4\,.\,4\,.\,35$ & $160$ & $8$ & $8$ & $0$ & $0$ & $0$ & $8$ & $8$\\
$4\,.\,4\,.\,36$ & $-120$ & $22$ & $7$ & $1$ & $0$ & $0$ & $1$ & $1$\\
$4\,.\,4\,.\,37$ & $-128$ & $20$ & $44$ & $0$ & $0$ & $0$ & $1$ & $1$\\
$4\,.\,4\,.\,38$ & $48$ & $13$ & $1$ & $0$ & $1$ & $1$ & $1$ & $1$\\
$4\,.\,4\,.\,39$ & $-116$ & $34$ & $10$ & $0$ & $1$ & $0$ & $1$ & $1$\\
$4\,.\,4\,.\,40$ & $-44$ & $26$ & $2$ & $0$ & $1$ & $0$ & $1$ & $1$\\
$4\,.\,4\,.\,41$ & $-8$ & $8$ & $8$ & $0$ & $0$ & $0$ & $1$ & $1$\\
$4\,.\,4\,.\,42$ & $-128$ & $10$ & $10$ & $0$ & $0$ & $0$ & $2$ & $2$\\
$4\,.\,4\,.\,43$ & $640$ & $20$ & $2$ & $1$ & $0$ & $0$ & $2$ & $2$\\
$4\,.\,4\,.\,44$ & $-120$ & $4$ & $10$ & $1$ & $0$ & $0$ & $2$ & $1$\\
$4\,.\,4\,.\,45$ & $-72$ & $16$ & $4$ & $0$ & $0$ & $0$ & $4$ & $4$\\
$4\,.\,4\,.\,46$ & $-78$ & $44$ & $8$ & $0$ & $0$ & $0$ & $2$ & $1$\\
$4\,.\,4\,.\,47$ & $-18$ & $28$ & $4$ & $1$ & $0$ & $0$ & $4$ & $1$\\
$4\,.\,4\,.\,48$ & $-60$ & $4$ & $4$ & $0$ & $0$ & $0$ & $1$ & $1$\\
$4\,.\,4\,.\,49$ & $192$ & $100$ & $4$ & $0$ & $1$ & $0$ & $4$ & $1$\\
$4\,.\,4\,.\,50$ & $-92$ & $-124$ & $8$ & $0$ & $1$ & $1$ & $1$ & $1$\\
$4\,.\,4\,.\,51$ & $-32$ & $4$ & $4$ & $0$ & $0$ & $0$ & $1$ & $1$\\
$4\,.\,4\,.\,52$ & $48$ & $32$ & $8$ & $0$ & $0$ & $0$ & $2$ & $1$\\
$4\,.\,4\,.\,53$ & $180$ & $8$ & $2$ & $0$ & $0$ & $0$ & $1$ & $1$\\
$4\,.\,4\,.\,54$ & $24$ & $28$ & $4$ & $0$ & $1$ & $1$ & $1$ & $1$\\
$4\,.\,4\,.\,55$ & $1200$ & $2$ & $2$ & $0$ & $0$ & $0$ & $2$ & $2$\\
$4\,.\,4\,.\,56$ & $136$ & $16$ & $4$ & $0$ & $0$ & $0$ & $2$ & $1$\\
$4\,.\,4\,.\,57$ & $0$ & $24$ & $6$ & $0$ & $0$ & $0$ & $3$ & $1$\\
$4\,.\,4\,.\,58$ & $-44$ & $2$ & $5$ & $1$ & $0$ & $0$ & $1$ & $1$\\
$4\,.\,4\,.\,59$ & $-44$ & $23$ & $3$ & $0$ & $1$ & $1$ & $3$ & $3$\\
$4\,.\,4\,.\,60$ & $24$ & $14$ & $2$ & $0$ & $0$ & $0$ & $2$ & $2$\\
$4\,.\,4\,.\,61$ & $-16$ & $20$ & $2$ & $0$ & $0$ & $0$ & $1$ & $1$\\
$4\,.\,4\,.\,62$ & $-92$ & $6$ & $3$ & $1$ & $0$ & $0$ & $1$ & $1$\\
$4\,.\,4\,.\,63$ & $104$ & $1$ & $1$ & $0$ & $1$ & $0$ & $1$ & $1$\\
$4\,.\,4\,.\,64$ & $192$ & $4$ & $4$ & $0$ & $0$ & $0$ & $1$ & $1$\\
$4\,.\,4\,.\,65$ & $-72$ & $-16$ & $8$ & $0$ & $1$ & $1$ & $2$ & $2$\\
$4\,.\,4\,.\,66$ & $-130$ & $14$ & $2$ & $1$ & $0$ & $1$ & $2$ & $2$\\
$4\,.\,4\,.\,67$ & $384$ & $10$ & $1$ & $1$ & $0$ & $0$ & $1$ & $1$\\
$4\,.\,4\,.\,68$ & $-102$ & $14$ & $5$ & $1$ & $0$ & $0$ & $1$ & $1$\\
$4\,.\,4\,.\,69$ & $-64$ & $20$ & $6$ & $0$ & $0$ & $0$ & $3$ & $3$\\
$4\,.\,4\,.\,70$ & $432$ & $20$ & $2$ & $0$ & $0$ & $0$ & $1$ & $1$\\
$4\,.\,4\,.\,71$ & $192$ & $8$ & $12$ & $0$ & $0$ & $0$ & $12$ & $12$\\
$4\,.\,4\,.\,72$ & $-116$ & $22$ & $6$ & $0$ & $1$ & $0$ & $3$ & $3$\\
$4\,.\,4\,.\,73$ & $68$ & $2$ & $2$ & $0$ & $1$ & $0$ & $1$ & $1$\\
$4\,.\,4\,.\,74$ & $-144$ & $10$ & $3$ & $1$ & $0$ & $0$ & $3$ & $3$\\
$4\,.\,4\,.\,76$ & $-24$ & $20$ & $2$ & $1$ & $0$ & $0$ & $2$ & $1$\\
$4\,.\,4\,.\,77$ & $192$ & $8$ & $24$ & $1$ & $0$ & $0$ & $24$ & $24$\\
$4\,.\,5\,.\,1$ & $-90$ & $36$ & $90$ & $0$ & $0$ & $0$ & $3$ & $1$\\
$4\,.\,5\,.\,2$ & $-106$ & $40$ & $46$ & $1$ & $0$ & $0$ & $2$ & $1$\\
$4\,.\,5\,.\,3$ & $-18$ & $72$ & $54$ & $1$ & $0$ & $0$ & $6$ & $1$\\
$4\,.\,5\,.\,4$ & $-88$ & $8$ & $80$ & $0$ & $0$ & $0$ & $2$ & $1$\\
$4\,.\,5\,.\,5$ & $-100$ & $4$ & $70$ & $1$ & $0$ & $0$ & $2$ & $1$\\
$4\,.\,5\,.\,6$ & $-32$ & $96$ & $96$ & $0$ & $0$ & $0$ & $8$ & $1$\\
$4\,.\,5\,.\,7$ & $-98$ & $12$ & $42$ & $0$ & $0$ & $0$ & $1$ & $1$\\
$4\,.\,5\,.\,8$ & $-2$ & $24$ & $6$ & $1$ & $0$ & $0$ & $2$ & $1$\\
$4\,.\,5\,.\,9$ & $544$ & $16$ & $16$ & $0$ & $0$ & $0$ & $8$ & $8$\\
$4\,.\,5\,.\,10$ & $-78$ & $92$ & $56$ & $0$ & $0$ & $0$ & $4$ & $1$\\
$4\,.\,5\,.\,11$ & $304$ & $32$ & $32$ & $0$ & $0$ & $0$ & $16$ & $1$\\
$4\,.\,5\,.\,12$ & $450$ & $42$ & $12$ & $0$ & $0$ & $0$ & $12$ & $6$\\
$4\,.\,5\,.\,13$ & $360$ & $20$ & $8$ & $0$ & $0$ & $0$ & $4$ & $4$\\
$4\,.\,5\,.\,14$ & $40$ & $40$ & $16$ & $0$ & $0$ & $0$ & $2$ & $1$\\
$4\,.\,5\,.\,15$ & $-32$ & $72$ & $12$ & $0$ & $1$ & $1$ & $1$ & $1$\\
$4\,.\,5\,.\,16$ & $-50$ & $16$ & $10$ & $0$ & $0$ & $0$ & $1$ & $1$\\
$4\,.\,5\,.\,17$ & $-8$ & $48$ & $24$ & $0$ & $0$ & $0$ & $4$ & $1$\\
$4\,.\,5\,.\,18$ & $-86$ & $22$ & $61$ & $1$ & $0$ & $0$ & $1$ & $1$\\
$4\,.\,5\,.\,19$ & $258$ & $46$ & $4$ & $0$ & $0$ & $0$ & $4$ & $2$\\
$4\,.\,5\,.\,20$ & $-84$ & $18$ & $57$ & $1$ & $0$ & $0$ & $1$ & $1$\\
$4\,.\,5\,.\,21$ & $60$ & $-387$ & $9$ & $0$ & $1$ & $1$ & $9$ & $1$\\
$4\,.\,5\,.\,22$ & $-102$ & $18$ & $21$ & $1$ & $0$ & $0$ & $1$ & $1$\\
$4\,.\,5\,.\,23$ & $-88$ & $4$ & $34$ & $0$ & $0$ & $0$ & $1$ & $1$\\
$4\,.\,5\,.\,24$ & $-100$ & $2$ & $29$ & $1$ & $0$ & $0$ & $1$ & $1$\\
$4\,.\,5\,.\,25$ & $-102$ & $177$ & $33$ & $0$ & $1$ & $0$ & $1$ & $1$\\
$4\,.\,5\,.\,26$ & $40$ & $4$ & $6$ & $0$ & $0$ & $0$ & $3$ & $3$\\
$4\,.\,5\,.\,27$ & $-92$ & $98$ & $38$ & $0$ & $1$ & $1$ & $1$ & $1$\\
$4\,.\,5\,.\,28$ & $-44$ & $16$ & $10$ & $1$ & $0$ & $0$ & $2$ & $2$\\
$4\,.\,5\,.\,29$ & $-100$ & $8$ & $14$ & $0$ & $0$ & $0$ & $1$ & $1$\\
$4\,.\,5\,.\,30$ & $-108$ & $14$ & $17$ & $1$ & $0$ & $0$ & $1$ & $1$\\
$4\,.\,5\,.\,31$ & $-92$ & $260$ & $56$ & $0$ & $1$ & $1$ & $1$ & $1$\\
$4\,.\,5\,.\,32$ & $-60$ & $0$ & $12$ & $0$ & $0$ & $0$ & $1$ & $1$\\
$4\,.\,5\,.\,33$ & $-52$ & $48$ & $18$ & $0$ & $0$ & $0$ & $3$ & $1$\\
$4\,.\,5\,.\,34$ & $-112$ & $84$ & $42$ & $1$ & $0$ & $0$ & $6$ & $1$\\
$4\,.\,5\,.\,35$ & $-100$ & $18$ & $21$ & $1$ & $0$ & $0$ & $1$ & $1$\\
$4\,.\,5\,.\,36$ & $-80$ & $54$ & $15$ & $1$ & $0$ & $0$ & $3$ & $1$\\
$4\,.\,5\,.\,37$ & $-32$ & $16$ & $10$ & $0$ & $0$ & $0$ & $1$ & $1$\\
$4\,.\,5\,.\,38$ & $-52$ & $0$ & $6$ & $0$ & $0$ & $0$ & $3$ & $3$\\
$4\,.\,5\,.\,39$ & $-72$ & $56$ & $20$ & $1$ & $0$ & $0$ & $4$ & $1$\\
$4\,.\,5\,.\,40$ & $-76$ & $0$ & $30$ & $0$ & $0$ & $0$ & $3$ & $1$\\
$4\,.\,5\,.\,41$ & $168$ & $48$ & $12$ & $0$ & $1$ & $1$ & $6$ & $1$\\
$4\,.\,5\,.\,42$ & $80$ & $48$ & $12$ & $0$ & $1$ & $1$ & $2$ & $1$\\
$4\,.\,5\,.\,43$ & $-88$ & $16$ & $28$ & $0$ & $0$ & $0$ & $2$ & $1$\\
$4\,.\,5\,.\,44$ & $-72$ & $20$ & $26$ & $1$ & $0$ & $0$ & $2$ & $1$\\
$4\,.\,5\,.\,45$ & $52$ & $6$ & $18$ & $0$ & $1$ & $1$ & $3$ & $1$\\
$4\,.\,5\,.\,46$ & $-98$ & $98$ & $14$ & $0$ & $1$ & $1$ & $1$ & $1$\\
$4\,.\,5\,.\,47$ & $-32$ & $320$ & $80$ & $0$ & $1$ & $0$ & $4$ & $1$\\
$4\,.\,5\,.\,48$ & $544$ & $64$ & $16$ & $0$ & $0$ & $0$ & $4$ & $1$\\
$4\,.\,5\,.\,49$ & $-92$ & $36$ & $48$ & $0$ & $0$ & $0$ & $2$ & $1$\\
$4\,.\,5\,.\,50$ & $192$ & $20$ & $20$ & $1$ & $0$ & $0$ & $4$ & $4$\\
$4\,.\,5\,.\,51$ & $-74$ & $14$ & $23$ & $1$ & $0$ & $0$ & $1$ & $1$\\
$4\,.\,5\,.\,52$ & $180$ & $2$ & $20$ & $1$ & $0$ & $0$ & $2$ & $2$\\
$4\,.\,5\,.\,53$ & $1040$ & $50$ & $5$ & $1$ & $0$ & $0$ & $5$ & $1$\\
$4\,.\,5\,.\,54$ & $-60$ & $70$ & $10$ & $0$ & $1$ & $1$ & $1$ & $1$\\
$4\,.\,5\,.\,55$ & $-64$ & $-76$ & $20$ & $0$ & $1$ & $0$ & $1$ & $1$\\
$4\,.\,5\,.\,56$ & $-16$ & $-104$ & $40$ & $0$ & $1$ & $0$ & $2$ & $1$\\
$4\,.\,5\,.\,57$ & $496$ & $-40$ & $8$ & $0$ & $1$ & $0$ & $2$ & $1$\\
$4\,.\,5\,.\,58$ & $18$ & $189$ & $45$ & $0$ & $1$ & $0$ & $3$ & $1$\\
$4\,.\,5\,.\,59$ & $558$ & $54$ & $9$ & $1$ & $0$ & $1$ & $3$ & $1$\\
$4\,.\,5\,.\,60$ & $54$ & $29$ & $5$ & $0$ & $1$ & $0$ & $1$ & $1$\\
$4\,.\,5\,.\,61$ & $426$ & $22$ & $1$ & $1$ & $0$ & $1$ & $1$ & $1$\\
$4\,.\,5\,.\,62$ & $48$ & $-52$ & $20$ & $0$ & $1$ & $0$ & $2$ & $1$\\
$4\,.\,5\,.\,63$ & $528$ & $40$ & $4$ & $0$ & $0$ & $0$ & $2$ & $1$\\
$4\,.\,5\,.\,64$ & $236$ & $-43$ & $5$ & $0$ & $1$ & $1$ & $1$ & $1$\\
$4\,.\,5\,.\,65$ & $252$ & $-9$ & $15$ & $0$ & $1$ & $1$ & $3$ & $3$\\
$4\,.\,5\,.\,66$ & $-84$ & $6$ & $15$ & $1$ & $0$ & $0$ & $3$ & $3$\\
$4\,.\,5\,.\,67$ & $272$ & $22$ & $10$ & $0$ & $0$ & $0$ & $2$ & $2$\\
$4\,.\,5\,.\,68$ & $-126$ & $6$ & $51$ & $1$ & $0$ & $0$ & $3$ & $3$\\
$4\,.\,5\,.\,69$ & $810$ & $9$ & $9$ & $0$ & $1$ & $1$ & $9$ & $9$\\
$4\,.\,5\,.\,70$ & $-104$ & $18$ & $21$ & $1$ & $0$ & $0$ & $1$ & $1$\\
$4\,.\,5\,.\,71$ & $234$ & $54$ & $18$ & $0$ & $1$ & $0$ & $3$ & $1$\\
$4\,.\,5\,.\,72$ & $34$ & $14$ & $3$ & $1$ & $0$ & $1$ & $3$ & $3$\\
$4\,.\,5\,.\,73$ & $-72$ & $4$ & $16$ & $0$ & $0$ & $0$ & $1$ & $1$\\
$4\,.\,5\,.\,74$ & $-72$ & $44$ & $14$ & $1$ & $0$ & $0$ & $2$ & $1$\\
$4\,.\,5\,.\,75$ & $-92$ & $12$ & $18$ & $0$ & $0$ & $0$ & $1$ & $1$\\
$4\,.\,5\,.\,76$ & $-90$ & $12$ & $18$ & $0$ & $0$ & $0$ & $1$ & $1$\\
$4\,.\,5\,.\,77$ & $-108$ & $6$ & $33$ & $1$ & $0$ & $0$ & $3$ & $1$\\
$4\,.\,5\,.\,78$ & $-104$ & $2$ & $29$ & $1$ & $0$ & $0$ & $1$ & $1$\\
$4\,.\,5\,.\,79$ & $-86$ & $16$ & $22$ & $0$ & $0$ & $0$ & $1$ & $1$\\
$4\,.\,5\,.\,80$ & $-144$ & $2$ & $26$ & $0$ & $0$ & $0$ & $2$ & $2$\\
$4\,.\,5\,.\,81$ & $-108$ & $6$ & $42$ & $0$ & $0$ & $0$ & $2$ & $2$\\
$4\,.\,5\,.\,82$ & $-96$ & $8$ & $14$ & $0$ & $0$ & $0$ & $1$ & $1$\\
$4\,.\,5\,.\,83$ & $-88$ & $2$ & $11$ & $1$ & $0$ & $0$ & $1$ & $1$\\
$4\,.\,5\,.\,84$ & $-86$ & $0$ & $30$ & $0$ & $0$ & $0$ & $1$ & $1$\\
$4\,.\,5\,.\,85$ & $66$ & $12$ & $6$ & $0$ & $0$ & $1$ & $3$ & $3$\\
$4\,.\,5\,.\,86$ & $-92$ & $2$ & $11$ & $1$ & $0$ & $0$ & $1$ & $1$\\
$4\,.\,5\,.\,87$ & $-128$ & $18$ & $18$ & $0$ & $0$ & $0$ & $2$ & $2$\\
$4\,.\,5\,.\,88$ & $-92$ & $16$ & $22$ & $0$ & $0$ & $0$ & $1$ & $1$\\
$4\,.\,5\,.\,89$ & $-44$ & $4$ & $16$ & $0$ & $0$ & $0$ & $1$ & $1$\\
$4\,.\,5\,.\,90$ & $320$ & $8$ & $8$ & $0$ & $0$ & $0$ & $2$ & $2$\\
$4\,.\,5\,.\,91$ & $68$ & $40$ & $16$ & $0$ & $0$ & $0$ & $4$ & $1$\\
$4\,.\,5\,.\,92$ & $-100$ & $60$ & $24$ & $0$ & $1$ & $0$ & $4$ & $2$\\
$4\,.\,5\,.\,93$ & $-100$ & $16$ & $12$ & $0$ & $0$ & $0$ & $3$ & $3$\\
$4\,.\,5\,.\,94$ & $-132$ & $47$ & $27$ & $0$ & $1$ & $0$ & $3$ & $3$\\
$4\,.\,5\,.\,95$ & $-74$ & $372$ & $12$ & $0$ & $1$ & $0$ & $1$ & $1$\\
$4\,.\,5\,.\,96$ & $-74$ & $622$ & $16$ & $0$ & $1$ & $0$ & $2$ & $2$\\
$4\,.\,5\,.\,97$ & $52$ & $12$ & $12$ & $0$ & $0$ & $0$ & $1$ & $1$\\
$4\,.\,5\,.\,98$ & $-100$ & $20$ & $26$ & $1$ & $0$ & $0$ & $2$ & $1$\\
$4\,.\,5\,.\,99$ & $-136$ & $14$ & $15$ & $1$ & $0$ & $0$ & $3$ & $3$\\
$4\,.\,5\,.\,100$ & $-16$ & $14$ & $6$ & $0$ & $0$ & $0$ & $6$ & $6$\\
$4\,.\,5\,.\,101$ & $-114$ & $20$ & $8$ & $1$ & $0$ & $1$ & $2$ & $2$\\
$4\,.\,5\,.\,102$ & $-108$ & $60$ & $48$ & $0$ & $1$ & $1$ & $6$ & $3$\\
$4\,.\,5\,.\,103$ & $-132$ & $281$ & $5$ & $0$ & $1$ & $0$ & $5$ & $5$\\
$4\,.\,5\,.\,104$ & $-128$ & $16$ & $52$ & $1$ & $0$ & $0$ & $4$ & $4$\\
$4\,.\,5\,.\,105$ & $376$ & $-16$ & $8$ & $0$ & $1$ & $1$ & $2$ & $2$\\
$4\,.\,5\,.\,106$ & $220$ & $53$ & $17$ & $0$ & $1$ & $0$ & $1$ & $1$\\
$4\,.\,5\,.\,107$ & $-176$ & $8$ & $20$ & $1$ & $0$ & $0$ & $4$ & $4$\\
$4\,.\,5\,.\,108$ & $-36$ & $14$ & $8$ & $0$ & $0$ & $0$ & $8$ & $8$\\
$4\,.\,5\,.\,109$ & $-72$ & $0$ & $12$ & $0$ & $0$ & $0$ & $1$ & $1$\\
$4\,.\,5\,.\,110$ & $-64$ & $12$ & $24$ & $0$ & $0$ & $0$ & $1$ & $1$\\
$4\,.\,5\,.\,111$ & $-60$ & $12$ & $24$ & $0$ & $0$ & $0$ & $1$ & $1$\\
$4\,.\,5\,.\,112$ & $-36$ & $0$ & $24$ & $0$ & $0$ & $0$ & $2$ & $1$\\
$4\,.\,5\,.\,113$ & $-116$ & $32$ & $20$ & $1$ & $0$ & $0$ & $4$ & $2$\\
$4\,.\,5\,.\,114$ & $768$ & $8$ & $8$ & $0$ & $0$ & $0$ & $2$ & $2$\\
$4\,.\,5\,.\,115$ & $792$ & $3$ & $3$ & $0$ & $1$ & $0$ & $3$ & $3$\\
$4\,.\,5\,.\,116$ & $752$ & $52$ & $4$ & $0$ & $1$ & $0$ & $1$ & $1$\\
$4\,.\,5\,.\,117$ & $608$ & $-47$ & $1$ & $0$ & $1$ & $0$ & $1$ & $1$\\
$4\,.\,5\,.\,118$ & $768$ & $8$ & $8$ & $0$ & $0$ & $0$ & $2$ & $2$\\
$4\,.\,5\,.\,119$ & $-20$ & $40$ & $10$ & $0$ & $0$ & $0$ & $5$ & $1$\\
$4\,.\,5\,.\,120$ & $-36$ & $8$ & $8$ & $0$ & $0$ & $0$ & $1$ & $1$\\
$4\,.\,5\,.\,121$ & $20$ & $8$ & $8$ & $0$ & $0$ & $0$ & $1$ & $1$\\
$4\,.\,5\,.\,122$ & $-12$ & $12$ & $6$ & $0$ & $0$ & $0$ & $3$ & $3$\\
$4\,.\,5\,.\,123$ & $96$ & $64$ & $16$ & $0$ & $1$ & $0$ & $1$ & $1$\\
$4\,.\,5\,.\,124$ & $-60$ & $12$ & $6$ & $0$ & $0$ & $0$ & $1$ & $1$\\
$4\,.\,5\,.\,125$ & $-68$ & $12$ & $6$ & $0$ & $0$ & $0$ & $1$ & $1$\\
$4\,.\,5\,.\,126$ & $-96$ & $20$ & $8$ & $0$ & $0$ & $0$ & $1$ & $1$\\
$4\,.\,5\,.\,127$ & $-114$ & $4$ & $10$ & $0$ & $0$ & $0$ & $1$ & $1$\\
$4\,.\,5\,.\,128$ & $-8$ & $41$ & $5$ & $0$ & $1$ & $1$ & $5$ & $5$\\
$4\,.\,5\,.\,129$ & $68$ & $46$ & $10$ & $0$ & $1$ & $1$ & $5$ & $5$\\
$4\,.\,5\,.\,130$ & $52$ & $16$ & $6$ & $0$ & $0$ & $0$ & $3$ & $3$\\
$4\,.\,5\,.\,131$ & $300$ & $12$ & $6$ & $0$ & $0$ & $0$ & $3$ & $3$\\
$4\,.\,5\,.\,132$ & $192$ & $0$ & $12$ & $0$ & $0$ & $0$ & $1$ & $1$\\
$4\,.\,5\,.\,133$ & $416$ & $16$ & $16$ & $0$ & $0$ & $0$ & $4$ & $2$\\
$4\,.\,5\,.\,134$ & $56$ & $36$ & $12$ & $0$ & $0$ & $0$ & $3$ & $1$\\
$4\,.\,6\,.\,3$ & $-44$ & $20$ & $14$ & $0$ & $0$ & $0$ & $1$ & $1$\\
$4\,.\,6\,.\,4$ & $-76$ & $10$ & $19$ & $1$ & $0$ & $0$ & $1$ & $1$\\
$4\,.\,6\,.\,5$ & $-212$ & $0$ & $18$ & $1$ & $0$ & $0$ & $2$ & $2$\\
$4\,.\,6\,.\,6$ & $-158$ & $0$ & $18$ & $0$ & $0$ & $0$ & $1$ & $1$\\
$4\,.\,6\,.\,7$ & $-312$ & $0$ & $18$ & $0$ & $0$ & $0$ & $3$ & $3$\\
$4\,.\,6\,.\,8$ & $-50$ & $22$ & $13$ & $1$ & $0$ & $0$ & $1$ & $1$\\
$4\,.\,6\,.\,9$ & $-88$ & $36$ & $162$ & $1$ & $0$ & $0$ & $2$ & $1$\\
$4\,.\,6\,.\,10$ & $-160$ & $0$ & $18$ & $0$ & $0$ & $0$ & $1$ & $1$\\
$4\,.\,6\,.\,11$ & $528$ & $16$ & $4$ & $0$ & $0$ & $0$ & $1$ & $1$\\
$4\,.\,6\,.\,12$ & $0$ & $60$ & $30$ & $0$ & $0$ & $0$ & $5$ & $1$\\
$4\,.\,6\,.\,14$ & $-190$ & $12$ & $24$ & $0$ & $0$ & $0$ & $1$ & $1$\\
$4\,.\,6\,.\,15$ & $832$ & $16$ & $16$ & $0$ & $0$ & $0$ & $4$ & $4$\\
$4\,.\,6\,.\,16$ & $992$ & $64$ & $16$ & $0$ & $0$ & $0$ & $4$ & $1$\\
$4\,.\,6\,.\,17$ & $48$ & $16$ & $60$ & $0$ & $0$ & $0$ & $3$ & $3$\\
$4\,.\,6\,.\,18$ & $0$ & $96$ & $96$ & $0$ & $0$ & $0$ & $24$ & $1$\\
$4\,.\,6\,.\,19$ & $-80$ & $140$ & $230$ & $1$ & $0$ & $0$ & $10$ & $1$\\
$4\,.\,6\,.\,20$ & $-12$ & $64$ & $84$ & $1$ & $0$ & $0$ & $12$ & $3$\\
$4\,.\,6\,.\,21$ & $-36$ & $-158340$ & $60$ & $0$ & $1$ & $0$ & $6$ & $1$\\
$4\,.\,6\,.\,22$ & $-76$ & $12$ & $96$ & $0$ & $0$ & $0$ & $1$ & $1$\\
$4\,.\,6\,.\,23$ & $-80$ & $20$ & $116$ & $0$ & $0$ & $0$ & $1$ & $1$\\
$4\,.\,6\,.\,24$ & $-72$ & $144$ & $216$ & $0$ & $0$ & $0$ & $12$ & $1$\\
$4\,.\,6\,.\,25$ & $-72$ & $465$ & $117$ & $0$ & $1$ & $1$ & $1$ & $1$\\
$4\,.\,6\,.\,26$ & $-80$ & $8$ & $86$ & $0$ & $0$ & $0$ & $1$ & $1$\\
$4\,.\,6\,.\,27$ & $-84$ & $48$ & $204$ & $1$ & $0$ & $0$ & $4$ & $1$\\
$4\,.\,6\,.\,28$ & $-72$ & $18$ & $27$ & $1$ & $0$ & $0$ & $1$ & $1$\\
$4\,.\,6\,.\,30$ & $-32$ & $8$ & $32$ & $0$ & $0$ & $0$ & $1$ & $1$\\
$4\,.\,6\,.\,31$ & $-50$ & $120$ & $150$ & $1$ & $0$ & $0$ & $10$ & $1$\\
$4\,.\,6\,.\,33$ & $-352$ & $64$ & $16$ & $0$ & $0$ & $0$ & $4$ & $1$\\
$4\,.\,6\,.\,34$ & $-32$ & $32$ & $48$ & $0$ & $1$ & $0$ & $12$ & $3$\\
$4\,.\,6\,.\,35$ & $270$ & $0$ & $12$ & $0$ & $1$ & $0$ & $3$ & $3$\\
$4\,.\,6\,.\,36$ & $-64$ & $32$ & $56$ & $0$ & $0$ & $0$ & $2$ & $1$\\
$4\,.\,6\,.\,37$ & $1168$ & $20$ & $8$ & $0$ & $0$ & $0$ & $8$ & $8$\\
$4\,.\,6\,.\,38$ & $528$ & $4$ & $10$ & $0$ & $0$ & $0$ & $10$ & $10$\\
$4\,.\,6\,.\,39$ & $-324$ & $228$ & $24$ & $0$ & $1$ & $1$ & $2$ & $2$\\
$4\,.\,6\,.\,40$ & $2664$ & $12$ & $6$ & $0$ & $0$ & $0$ & $3$ & $3$\\
$4\,.\,6\,.\,41$ & $-84$ & $20$ & $26$ & $0$ & $0$ & $0$ & $1$ & $1$\\
$4\,.\,7\,.\,1$ & $-64$ & $32$ & $56$ & $0$ & $0$ & $0$ & $2$ & $1$\\
$4\,.\,7\,.\,2$ & $1168$ & $20$ & $8$ & $1$ & $0$ & $0$ & $8$ & $8$\\
$4\,.\,7\,.\,3$ & $0$ & $96$ & $96$ & $0$ & $0$ & $0$ & $24$ & $1$\\
$4\,.\,7\,.\,4$ & $-58$ & $6$ & $6$ & $0$ & $1$ & $0$ & $1$ & $1$\\
$4\,.\,7\,.\,5$ & $68$ & $62$ & $2$ & $0$ & $1$ & $0$ & $1$ & $1$\\
$4\,.\,7\,.\,6$ & $-72$ & $144$ & $216$ & $0$ & $0$ & $0$ & $12$ & $1$\\
$4\,.\,7\,.\,7$ & $-50$ & $120$ & $150$ & $1$ & $0$ & $0$ & $10$ & $1$\\
$4\,.\,7\,.\,8$ & $-72$ & $-4104$ & $216$ & $0$ & $1$ & $0$ & $12$ & $1$\\
$4\,.\,7\,.\,9$ & $-18$ & $90$ & $54$ & $0$ & $1$ & $1$ & $6$ & $1$\\
$4\,.\,7\,.\,10$ & $-36$ & $-18660$ & $60$ & $0$ & $1$ & $0$ & $6$ & $1$\\
$4\,.\,7\,.\,11$ & $-324$ & $948$ & $24$ & $0$ & $1$ & $0$ & $2$ & $2$\\
$4\,.\,7\,.\,12$ & $270$ & $12$ & $12$ & $0$ & $0$ & $1$ & $3$ & $3$\\
$4\,.\,7\,.\,13$ & $320$ & $16$ & $16$ & $0$ & $0$ & $0$ & $2$ & $1$\\
$4\,.\,7\,.\,14$ & $290$ & $4$ & $4$ & $0$ & $0$ & $1$ & $1$ & $1$\\
$4\,.\,7\,.\,15$ & $528$ & $22$ & $10$ & $1$ & $0$ & $1$ & $10$ & $10$\\
$4\,.\,7\,.\,16$ & $16$ & $0$ & $60$ & $0$ & $0$ & $0$ & $3$ & $1$\\
$4\,.\,7\,.\,17$ & $234$ & $54$ & $18$ & $0$ & $1$ & $0$ & $9$ & $9$\\
$4\,.\,7\,.\,18$ & $-352$ & $64$ & $16$ & $0$ & $0$ & $0$ & $4$ & $1$\\
$4\,.\,7\,.\,19$ & $-32$ & $32$ & $48$ & $0$ & $1$ & $0$ & $12$ & $3$\\
$4\,.\,7\,.\,20$ & $-18$ & $6$ & $27$ & $1$ & $0$ & $0$ & $1$ & $1$\\
$4\,.\,7\,.\,21$ & $738$ & $6$ & $9$ & $1$ & $0$ & $1$ & $3$ & $3$\\
$4\,.\,8\,.\,1$ & $-90$ & $24$ & $30$ & $1$ & $0$ & $0$ & $2$ & $1$\\
$4\,.\,8\,.\,2$ & $-60$ & $72$ & $36$ & $0$ & $0$ & $0$ & $6$ & $1$\\
$4\,.\,8\,.\,3$ & $12$ & $36$ & $24$ & $0$ & $0$ & $0$ & $4$ & $2$\\
$4\,.\,8\,.\,4$ & $-18$ & $30$ & $18$ & $0$ & $1$ & $1$ & $6$ & $3$\\
$4\,.\,8\,.\,5$ & $96$ & $36$ & $36$ & $0$ & $1$ & $0$ & $6$ & $1$\\
$4\,.\,8\,.\,6$ & $-80$ & $0$ & $30$ & $0$ & $0$ & $0$ & $3$ & $1$\\
$4\,.\,8\,.\,7$ & $-30$ & $38$ & $20$ & $0$ & $0$ & $0$ & $4$ & $2$\\
$4\,.\,8\,.\,8$ & $-60$ & $31$ & $15$ & $0$ & $1$ & $1$ & $3$ & $3$\\
$4\,.\,8\,.\,9$ & $50$ & $-102$ & $30$ & $0$ & $1$ & $0$ & $6$ & $1$\\
$4\,.\,8\,.\,10$ & $28$ & $36$ & $24$ & $0$ & $0$ & $0$ & $12$ & $2$\\
$4\,.\,8\,.\,11$ & $0$ & $30$ & $18$ & $0$ & $1$ & $1$ & $3$ & $3$\\
$4\,.\,8\,.\,12$ & $108$ & $468$ & $36$ & $0$ & $1$ & $0$ & $6$ & $1$\\
$4\,.\,8\,.\,13$ & $138$ & $38$ & $12$ & $1$ & $0$ & $1$ & $12$ & $6$\\
$4\,.\,8\,.\,14$ & $236$ & $12$ & $24$ & $0$ & $0$ & $0$ & $12$ & $2$\\
$4\,.\,8\,.\,15$ & $234$ & $0$ & $18$ & $1$ & $0$ & $1$ & $6$ & $3$\\
$4\,.\,8\,.\,16$ & $-90$ & $14$ & $47$ & $1$ & $0$ & $0$ & $1$ & $1$\\
$4\,.\,8\,.\,17$ & $14$ & $14$ & $6$ & $0$ & $1$ & $1$ & $3$ & $3$\\
$4\,.\,8\,.\,18$ & $-88$ & $40$ & $52$ & $0$ & $0$ & $0$ & $2$ & $1$\\
$4\,.\,8\,.\,19$ & $80$ & $14$ & $8$ & $0$ & $0$ & $0$ & $4$ & $4$\\
$4\,.\,8\,.\,20$ & $-90$ & $4$ & $34$ & $0$ & $0$ & $0$ & $1$ & $1$\\
$4\,.\,8\,.\,21$ & $-6$ & $168$ & $36$ & $0$ & $1$ & $1$ & $12$ & $1$\\
$4\,.\,8\,.\,22$ & $-90$ & $8$ & $38$ & $0$ & $0$ & $0$ & $1$ & $1$\\
$4\,.\,8\,.\,23$ & $-90$ & $83$ & $11$ & $0$ & $1$ & $0$ & $1$ & $1$\\
$4\,.\,8\,.\,24$ & $-86$ & $18$ & $24$ & $1$ & $0$ & $0$ & $2$ & $2$\\
$4\,.\,8\,.\,25$ & $384$ & $60$ & $18$ & $0$ & $0$ & $0$ & $6$ & $2$\\
$4\,.\,8\,.\,26$ & $-84$ & $16$ & $22$ & $0$ & $0$ & $0$ & $1$ & $1$\\
$4\,.\,8\,.\,27$ & $12$ & $22$ & $10$ & $0$ & $0$ & $0$ & $2$ & $2$\\
$4\,.\,8\,.\,28$ & $-86$ & $20$ & $26$ & $0$ & $0$ & $0$ & $1$ & $1$\\
$4\,.\,8\,.\,29$ & $-74$ & $116$ & $32$ & $0$ & $1$ & $1$ & $2$ & $1$\\
$4\,.\,8\,.\,30$ & $-68$ & $0$ & $12$ & $0$ & $0$ & $0$ & $1$ & $1$\\
$4\,.\,8\,.\,31$ & $-76$ & $6$ & $15$ & $1$ & $0$ & $0$ & $1$ & $1$\\
$4\,.\,8\,.\,32$ & $36$ & $60$ & $36$ & $0$ & $0$ & $0$ & $3$ & $1$\\
$4\,.\,8\,.\,33$ & $-110$ & $0$ & $30$ & $1$ & $0$ & $0$ & $2$ & $2$\\
$4\,.\,8\,.\,34$ & $-80$ & $6$ & $15$ & $1$ & $0$ & $0$ & $1$ & $1$\\
$4\,.\,8\,.\,35$ & $-72$ & $18$ & $117$ & $1$ & $0$ & $0$ & $1$ & $1$\\
$4\,.\,8\,.\,36$ & $-80$ & $20$ & $116$ & $0$ & $0$ & $0$ & $1$ & $1$\\
$4\,.\,8\,.\,37$ & $-116$ & $22$ & $33$ & $1$ & $0$ & $0$ & $3$ & $3$\\
$4\,.\,8\,.\,38$ & $-20$ & $8$ & $8$ & $0$ & $0$ & $0$ & $1$ & $1$\\
$4\,.\,8\,.\,39$ & $-90$ & $4$ & $30$ & $0$ & $0$ & $0$ & $3$ & $3$\\
$4\,.\,8\,.\,40$ & $-76$ & $12$ & $96$ & $0$ & $0$ & $0$ & $1$ & $1$\\
$4\,.\,8\,.\,41$ & $-50$ & $12$ & $24$ & $0$ & $0$ & $0$ & $1$ & $1$\\
$4\,.\,8\,.\,42$ & $-8$ & $-40$ & $32$ & $0$ & $1$ & $1$ & $4$ & $2$\\
$4\,.\,8\,.\,43$ & $-60$ & $16$ & $28$ & $0$ & $0$ & $0$ & $1$ & $1$\\
$4\,.\,8\,.\,44$ & $360$ & $-20$ & $16$ & $0$ & $1$ & $0$ & $8$ & $2$\\
$4\,.\,8\,.\,45$ & $-60$ & $1516$ & $40$ & $0$ & $1$ & $1$ & $2$ & $2$\\
$4\,.\,8\,.\,46$ & $6$ & $8$ & $8$ & $0$ & $0$ & $0$ & $1$ & $1$\\
$4\,.\,8\,.\,47$ & $-50$ & $0$ & $18$ & $1$ & $0$ & $0$ & $2$ & $1$\\
$4\,.\,8\,.\,48$ & $-64$ & $22$ & $13$ & $1$ & $0$ & $0$ & $1$ & $1$\\
$4\,.\,8\,.\,49$ & $-50$ & $6$ & $12$ & $0$ & $0$ & $0$ & $2$ & $2$\\
$4\,.\,8\,.\,50$ & $-80$ & $140$ & $230$ & $1$ & $0$ & $0$ & $10$ & $1$\\
$4\,.\,8\,.\,51$ & $-80$ & $8$ & $86$ & $0$ & $0$ & $0$ & $1$ & $1$\\
$4\,.\,8\,.\,52$ & $544$ & $16$ & $16$ & $0$ & $1$ & $0$ & $4$ & $1$\\
$4\,.\,8\,.\,53$ & $304$ & $16$ & $16$ & $0$ & $0$ & $0$ & $2$ & $2$\\
$4\,.\,8\,.\,54$ & $-12$ & $46396$ & $84$ & $0$ & $1$ & $0$ & $12$ & $3$\\
$4\,.\,8\,.\,55$ & $192$ & $32$ & $96$ & $0$ & $0$ & $0$ & $12$ & $6$\\
$4\,.\,8\,.\,56$ & $-22$ & $8$ & $8$ & $0$ & $0$ & $0$ & $1$ & $1$\\
$4\,.\,8\,.\,57$ & $-58$ & $196$ & $40$ & $0$ & $1$ & $1$ & $1$ & $1$\\
$4\,.\,8\,.\,58$ & $-62$ & $12$ & $6$ & $0$ & $0$ & $0$ & $1$ & $1$\\
$4\,.\,8\,.\,59$ & $-6$ & $0$ & $30$ & $0$ & $0$ & $1$ & $3$ & $3$\\
$4\,.\,8\,.\,60$ & $-66$ & $12$ & $6$ & $0$ & $0$ & $0$ & $1$ & $1$\\
$4\,.\,8\,.\,61$ & $-12$ & $0$ & $30$ & $1$ & $0$ & $0$ & $2$ & $2$\\
$4\,.\,8\,.\,62$ & $222$ & $12$ & $6$ & $0$ & $0$ & $1$ & $3$ & $3$\\
$4\,.\,8\,.\,63$ & $-68$ & $18$ & $30$ & $0$ & $0$ & $0$ & $2$ & $2$\\
$4\,.\,8\,.\,64$ & $96$ & $16$ & $16$ & $0$ & $0$ & $1$ & $2$ & $1$\\
$4\,.\,8\,.\,65$ & $-72$ & $196$ & $40$ & $0$ & $1$ & $1$ & $1$ & $1$\\
$4\,.\,8\,.\,66$ & $24$ & $24$ & $12$ & $0$ & $0$ & $0$ & $2$ & $1$\\
$4\,.\,8\,.\,67$ & $-84$ & $162$ & $30$ & $0$ & $1$ & $1$ & $1$ & $1$\\
$4\,.\,8\,.\,68$ & $-84$ & $48$ & $204$ & $1$ & $0$ & $0$ & $4$ & $1$\\
$4\,.\,8\,.\,69$ & $-36$ & $88$ & $64$ & $0$ & $0$ & $0$ & $4$ & $1$\\
$4\,.\,8\,.\,70$ & $12$ & $40$ & $16$ & $0$ & $0$ & $0$ & $2$ & $1$\\
$4\,.\,8\,.\,71$ & $450$ & $24$ & $18$ & $1$ & $0$ & $0$ & $6$ & $3$\\
$4\,.\,8\,.\,72$ & $-72$ & $18$ & $27$ & $1$ & $0$ & $0$ & $1$ & $1$\\
$4\,.\,8\,.\,73$ & $-84$ & $20$ & $26$ & $0$ & $0$ & $0$ & $1$ & $1$\\
$4\,.\,8\,.\,74$ & $-32$ & $8$ & $32$ & $0$ & $0$ & $0$ & $1$ & $1$\\
$4\,.\,8\,.\,75$ & $440$ & $0$ & $12$ & $0$ & $0$ & $0$ & $2$ & $2$\\
$4\,.\,8\,.\,76$ & $362$ & $24$ & $6$ & $1$ & $0$ & $0$ & $2$ & $1$\\
$4\,.\,8\,.\,77$ & $-58$ & $8$ & $20$ & $0$ & $0$ & $0$ & $2$ & $1$\\
$4\,.\,8\,.\,78$ & $216$ & $10$ & $6$ & $0$ & $0$ & $0$ & $6$ & $6$\\
$4\,.\,8\,.\,79$ & $16$ & $0$ & $60$ & $0$ & $0$ & $0$ & $3$ & $1$\\
$4\,.\,8\,.\,80$ & $416$ & $0$ & $24$ & $0$ & $0$ & $0$ & $4$ & $4$\\
$4\,.\,8\,.\,81$ & $208$ & $64$ & $16$ & $0$ & $1$ & $0$ & $2$ & $1$\\
$4\,.\,8\,.\,82$ & $-12$ & $64$ & $84$ & $1$ & $0$ & $0$ & $12$ & $3$\\
$4\,.\,8\,.\,83$ & $180$ & $10$ & $4$ & $0$ & $0$ & $0$ & $4$ & $4$\\
$4\,.\,8\,.\,84$ & $-84$ & $0$ & $30$ & $0$ & $0$ & $0$ & $3$ & $3$\\
$4\,.\,8\,.\,85$ & $136$ & $24$ & $12$ & $0$ & $0$ & $0$ & $2$ & $1$\\
$4\,.\,8\,.\,86$ & $-78$ & $2$ & $35$ & $1$ & $0$ & $0$ & $1$ & $1$\\
$4\,.\,8\,.\,88$ & $-76$ & $6$ & $15$ & $1$ & $0$ & $0$ & $1$ & $1$\\
$4\,.\,9\,.\,1$ & $-76$ & $20$ & $42$ & $1$ & $0$ & $0$ & $6$ & $3$\\
$4\,.\,9\,.\,2$ & $-56$ & $96$ & $84$ & $0$ & $0$ & $0$ & $6$ & $1$\\
$4\,.\,9\,.\,3$ & $-18$ & $6$ & $27$ & $1$ & $0$ & $0$ & $1$ & $1$\\
$4\,.\,9\,.\,4$ & $738$ & $-15$ & $9$ & $0$ & $1$ & $1$ & $3$ & $3$\\
$4\,.\,9\,.\,5$ & $-68$ & $24$ & $36$ & $0$ & $0$ & $0$ & $2$ & $1$\\
$4\,.\,9\,.\,6$ & $-74$ & $12$ & $102$ & $0$ & $0$ & $0$ & $1$ & $1$\\
$4\,.\,9\,.\,7$ & $468$ & $0$ & $18$ & $0$ & $0$ & $0$ & $9$ & $9$\\
$4\,.\,9\,.\,8$ & $-20$ & $16$ & $34$ & $0$ & $0$ & $0$ & $1$ & $1$\\
$4\,.\,9\,.\,9$ & $-74$ & $22$ & $31$ & $1$ & $0$ & $0$ & $1$ & $1$\\
$4\,.\,9\,.\,10$ & $-68$ & $22$ & $31$ & $1$ & $0$ & $0$ & $1$ & $1$\\
$4\,.\,10\,.\,1$ & $-78$ & $10$ & $91$ & $1$ & $0$ & $0$ & $1$ & $1$\\
$4\,.\,10\,.\,2$ & $276$ & $196$ & $24$ & $0$ & $1$ & $0$ & $6$ & $6$\\
$4\,.\,10\,.\,3$ & $132$ & $0$ & $18$ & $1$ & $0$ & $0$ & $6$ & $6$\\
$4\,.\,10\,.\,4$ & $300$ & $0$ & $18$ & $1$ & $0$ & $0$ & $6$ & $6$\\
$4\,.\,10\,.\,5$ & $-6$ & $462$ & $18$ & $0$ & $1$ & $1$ & $3$ & $3$\\
$4\,.\,10\,.\,6$ & $96$ & $28$ & $24$ & $0$ & $1$ & $1$ & $3$ & $3$\\
$4\,.\,10\,.\,7$ & $54$ & $1372$ & $24$ & $0$ & $1$ & $1$ & $3$ & $3$\\
$4\,.\,10\,.\,8$ & $-12$ & $30$ & $18$ & $0$ & $1$ & $1$ & $3$ & $3$\\
$4\,.\,10\,.\,9$ & $-120$ & $16$ & $64$ & $0$ & $0$ & $0$ & $1$ & $1$\\
$4\,.\,10\,.\,10$ & $416$ & $80$ & $48$ & $0$ & $0$ & $0$ & $12$ & $3$\\
$4\,.\,11\,.\,1$ & $-82$ & $4$ & $18$ & $0$ & $0$ & $1$ & $3$ & $3$\\
$4\,.\,11\,.\,2$ & $-64$ & $10$ & $97$ & $1$ & $0$ & $0$ & $1$ & $1$\\
$4\,.\,11\,.\,3$ & $-78$ & $10$ & $91$ & $1$ & $0$ & $0$ & $1$ & $1$\\
$4\,.\,11\,.\,4$ & $-50$ & $8$ & $20$ & $0$ & $0$ & $0$ & $1$ & $1$\\
$4\,.\,11\,.\,5$ & $40$ & $308$ & $8$ & $0$ & $1$ & $1$ & $2$ & $2$\\
$4\,.\,11\,.\,6$ & $-74$ & $12$ & $102$ & $0$ & $0$ & $0$ & $1$ & $1$\\
$4\,.\,11\,.\,7$ & $468$ & $0$ & $18$ & $0$ & $0$ & $0$ & $9$ & $9$\\
$4\,.\,11\,.\,8$ & $-20$ & $16$ & $34$ & $0$ & $0$ & $0$ & $1$ & $1$\\
$4\,.\,11\,.\,9$ & $-68$ & $24$ & $36$ & $0$ & $0$ & $0$ & $2$ & $1$\\
$4\,.\,11\,.\,10$ & $-68$ & $22$ & $31$ & $1$ & $0$ & $0$ & $1$ & $1$\\
$4\,.\,11\,.\,11$ & $-74$ & $22$ & $31$ & $1$ & $0$ & $0$ & $1$ & $1$\\
$4\,.\,11\,.\,12$ & $-78$ & $2$ & $35$ & $1$ & $0$ & $0$ & $1$ & $1$\\
$4\,.\,11\,.\,13$ & $-82$ & $6$ & $39$ & $1$ & $0$ & $0$ & $3$ & $1$\\
$4\,.\,11\,.\,14$ & $-74$ & $18$ & $27$ & $1$ & $0$ & $0$ & $1$ & $1$\\
$4\,.\,11\,.\,15$ & $190$ & $14$ & $3$ & $1$ & $0$ & $1$ & $3$ & $3$\\
$4\,.\,11\,.\,16$ & $-80$ & $2$ & $35$ & $1$ & $0$ & $0$ & $1$ & $1$\\
$4\,.\,11\,.\,17$ & $144$ & $30$ & $6$ & $0$ & $1$ & $0$ & $3$ & $3$\\
$4\,.\,11\,.\,18$ & $-82$ & $26$ & $26$ & $0$ & $1$ & $0$ & $1$ & $1$\\
$4\,.\,11\,.\,19$ & $-42$ & $24$ & $12$ & $0$ & $1$ & $0$ & $3$ & $3$\\
$4\,.\,11\,.\,20$ & $-72$ & $20$ & $32$ & $0$ & $0$ & $0$ & $2$ & $1$\\
$4\,.\,11\,.\,21$ & $600$ & $56$ & $8$ & $1$ & $0$ & $0$ & $8$ & $2$\\
$4\,.\,12\,.\,2$ & $528$ & $16$ & $16$ & $0$ & $0$ & $0$ & $2$ & $2$\\
$4\,.\,12\,.\,3$ & $192$ & $16$ & $40$ & $0$ & $0$ & $0$ & $2$ & $2$\\
$4\,.\,12\,.\,4$ & $192$ & $80$ & $80$ & $0$ & $0$ & $0$ & $4$ & $1$\\
$4\,.\,12\,.\,5$ & $276$ & $148$ & $24$ & $0$ & $1$ & $0$ & $6$ & $6$\\
$4\,.\,12\,.\,6$ & $-20$ & $60$ & $36$ & $0$ & $0$ & $0$ & $3$ & $1$\\
$4\,.\,12\,.\,7$ & $-40$ & $18$ & $21$ & $1$ & $0$ & $0$ & $3$ & $3$\\
$4\,.\,12\,.\,8$ & $768$ & $64$ & $16$ & $0$ & $0$ & $0$ & $4$ & $1$\\
$4\,.\,12\,.\,9$ & $768$ & $8$ & $8$ & $0$ & $0$ & $0$ & $2$ & $2$\\
$4\,.\,12\,.\,10$ & $768$ & $8$ & $8$ & $0$ & $0$ & $0$ & $2$ & $2$\\
$4\,.\,12\,.\,12$ & $152$ & $16$ & $16$ & $0$ & $0$ & $0$ & $1$ & $1$\\
$4\,.\,12\,.\,13$ & $416$ & $16$ & $40$ & $0$ & $0$ & $0$ & $2$ & $2$\\
$4\,.\,12\,.\,14$ & $90$ & $36$ & $36$ & $0$ & $1$ & $1$ & $3$ & $3$\\
$4\,.\,12\,.\,15$ & $-68$ & $4$ & $40$ & $0$ & $0$ & $0$ & $1$ & $1$\\
$4\,.\,12\,.\,16$ & $768$ & $16$ & $16$ & $0$ & $0$ & $0$ & $2$ & $1$\\
$4\,.\,12\,.\,17$ & $48$ & $16$ & $64$ & $0$ & $0$ & $0$ & $2$ & $1$\\
$4\,.\,12\,.\,18$ & $1056$ & $16$ & $16$ & $0$ & $0$ & $0$ & $4$ & $4$\\
$4\,.\,13\,.\,2$ & $768$ & $16$ & $16$ & $0$ & $0$ & $0$ & $2$ & $1$\\
$4\,.\,13\,.\,3$ & $48$ & $16$ & $64$ & $0$ & $0$ & $0$ & $2$ & $1$\\
$4\,.\,13\,.\,4$ & $1056$ & $16$ & $16$ & $0$ & $0$ & $0$ & $4$ & $4$\\
$4\,.\,13\,.\,5$ & $768$ & $8$ & $8$ & $0$ & $0$ & $0$ & $2$ & $2$\\
$4\,.\,13\,.\,6$ & $132$ & $0$ & $18$ & $1$ & $0$ & $0$ & $6$ & $6$\\
$4\,.\,13\,.\,7$ & $300$ & $0$ & $18$ & $1$ & $0$ & $0$ & $6$ & $6$\\
$4\,.\,13\,.\,8$ & $-6$ & $30$ & $18$ & $0$ & $1$ & $1$ & $3$ & $3$\\
$4\,.\,13\,.\,9$ & $96$ & $412$ & $24$ & $0$ & $1$ & $1$ & $3$ & $3$\\
$4\,.\,13\,.\,10$ & $54$ & $28$ & $24$ & $0$ & $1$ & $1$ & $3$ & $3$\\
$4\,.\,13\,.\,11$ & $-12$ & $30$ & $18$ & $0$ & $1$ & $1$ & $3$ & $3$\\
$4\,.\,13\,.\,12$ & $1088$ & $20$ & $12$ & $1$ & $0$ & $0$ & $12$ & $12$\\
$4\,.\,13\,.\,13$ & $-50$ & $8$ & $8$ & $0$ & $0$ & $0$ & $1$ & $1$\\
$4\,.\,13\,.\,14$ & $276$ & $100$ & $24$ & $0$ & $1$ & $1$ & $6$ & $6$\\
$4\,.\,13\,.\,15$ & $-14$ & $0$ & $48$ & $0$ & $0$ & $0$ & $3$ & $1$\\
$4\,.\,13\,.\,16$ & $48$ & $36$ & $36$ & $0$ & $0$ & $0$ & $6$ & $2$\\
$4\,.\,13\,.\,17$ & $-74$ & $2$ & $35$ & $1$ & $0$ & $0$ & $1$ & $1$\\
$4\,.\,14\,.\,1$ & $48$ & $16$ & $240$ & $0$ & $0$ & $0$ & $6$ & $6$\\
$4\,.\,14\,.\,2$ & $-68$ & $4$ & $40$ & $0$ & $0$ & $0$ & $1$ & $1$\\
$4\,.\,14\,.\,3$ & $-72$ & $20$ & $32$ & $0$ & $0$ & $0$ & $2$ & $1$\\
$4\,.\,14\,.\,4$ & $-74$ & $18$ & $27$ & $1$ & $0$ & $0$ & $1$ & $1$\\
$4\,.\,14\,.\,5$ & $-80$ & $2$ & $35$ & $1$ & $0$ & $0$ & $1$ & $1$\\
$4\,.\,14\,.\,6$ & $-82$ & $20$ & $26$ & $0$ & $0$ & $0$ & $1$ & $1$\\
$4\,.\,14\,.\,7$ & $152$ & $16$ & $16$ & $0$ & $0$ & $0$ & $1$ & $1$\\
$4\,.\,14\,.\,8$ & $416$ & $16$ & $40$ & $0$ & $0$ & $0$ & $2$ & $2$\\
$4\,.\,14\,.\,9$ & $640$ & $80$ & $48$ & $0$ & $0$ & $0$ & $12$ & $3$\\
$4\,.\,14\,.\,10$ & $992$ & $20$ & $32$ & $1$ & $0$ & $0$ & $32$ & $32$\\
$4\,.\,14\,.\,11$ & $-256$ & $80$ & $80$ & $0$ & $0$ & $0$ & $4$ & $1$\\
$4\,.\,15\,.\,1$ & $90$ & $36$ & $36$ & $0$ & $1$ & $1$ & $3$ & $3$\\
$4\,.\,15\,.\,2$ & $992$ & $20$ & $32$ & $1$ & $0$ & $0$ & $32$ & $32$\\
$4\,.\,15\,.\,3$ & $640$ & $80$ & $48$ & $0$ & $0$ & $0$ & $12$ & $3$\\
$4\,.\,15\,.\,4$ & $-256$ & $80$ & $80$ & $0$ & $0$ & $0$ & $4$ & $1$\\
$4\,.\,16\,.\,1$ & $2664$ & $12$ & $24$ & $0$ & $0$ & $0$ & $6$ & $6$\\
$4\,.\,16\,.\,2$ & $156$ & $162$ & $54$ & $0$ & $1$ & $1$ & $9$ & $3$\\
$4\,.\,16\,.\,3$ & $-88$ & $24$ & $12$ & $0$ & $0$ & $0$ & $2$ & $1$\\
$4\,.\,16\,.\,4$ & $-90$ & $12$ & $6$ & $0$ & $0$ & $1$ & $3$ & $3$\\
$4\,.\,16\,.\,5$ & $76$ & $6$ & $18$ & $0$ & $0$ & $0$ & $6$ & $6$\\
$4\,.\,16\,.\,6$ & $378$ & $12$ & $6$ & $0$ & $0$ & $1$ & $3$ & $3$\\
$4\,.\,16\,.\,7$ & $-14$ & $-402$ & $18$ & $0$ & $1$ & $1$ & $3$ & $3$\\
$4\,.\,16\,.\,8$ & $546$ & $0$ & $54$ & $0$ & $0$ & $1$ & $9$ & $3$\\
$4\,.\,16\,.\,9$ & $-70$ & $12$ & $6$ & $0$ & $0$ & $0$ & $1$ & $1$\\
$4\,.\,16\,.\,10$ & $-58$ & $12$ & $6$ & $0$ & $0$ & $0$ & $1$ & $1$\\
$4\,.\,16\,.\,11$ & $356$ & $0$ & $18$ & $1$ & $0$ & $0$ & $6$ & $6$\\
$4\,.\,16\,.\,12$ & $248$ & $24$ & $12$ & $0$ & $0$ & $0$ & $2$ & $1$\\
$4\,.\,16\,.\,13$ & $-4$ & $30$ & $18$ & $0$ & $1$ & $1$ & $3$ & $3$\\
$4\,.\,16\,.\,14$ & $96$ & $64$ & $16$ & $0$ & $1$ & $1$ & $2$ & $1$\\
$4\,.\,16\,.\,15$ & $40$ & $28$ & $24$ & $0$ & $1$ & $1$ & $3$ & $3$\\
$4\,.\,16\,.\,16$ & $34$ & $8$ & $8$ & $0$ & $0$ & $0$ & $1$ & $1$\\
$4\,.\,16\,.\,17$ & $110$ & $28$ & $24$ & $0$ & $1$ & $1$ & $3$ & $3$\\
$4\,.\,18\,.\,1$ & $-50$ & $8$ & $8$ & $0$ & $0$ & $0$ & $1$ & $1$\\
$4\,.\,18\,.\,2$ & $-14$ & $0$ & $48$ & $0$ & $0$ & $0$ & $3$ & $1$\\
$4\,.\,18\,.\,3$ & $48$ & $36$ & $36$ & $0$ & $0$ & $0$ & $6$ & $2$\\
$4\,.\,18\,.\,4$ & $-62$ & $0$ & $36$ & $0$ & $0$ & $0$ & $1$ & $1$\\
$4\,.\,24\,.\,3$ & $-82$ & $0$ & $30$ & $0$ & $0$ & $0$ & $1$ & $1$\\
\end{supertabular}
\endgroup
\end{adjustwidth}
\clearpage
\onecolumn


%
%
  
  \printbibliography

@ARTICLE{AlmkvistEtAl_2005,
  AUTHOR = {Almkvist, G. and van Enckevort, Christian and Straten, D. and Zudilin, W.},
  DATE = {2005-07},
  JOURNALTITLE = {arXiv: Algebraic Geometry},
  TITLE = {Tables of {{Calabi-Yau}} Equations},
  URLDATE = {2025-01-07},
}

@ARTICLE{AlmkvistVanStraten_2023,
  AUTHOR = {Almkvist, Gert and {van Straten}, Duco},
  DATE = {2023},
  DOI = {10.1007/s10801-023-01272-0},
  ISSN = {0925-9899},
  JOURNALTITLE = {Journal of Algebraic Combinatorics},
  LANGID = {english},
  NUMBER = {4},
  PAGES = {1203--1259},
  TITLE = {Calabi-{{Yau}} Operators of Degree Two},
  VOLUME = {58},
}

@ARTICLE{BruinEtAl_2019,
  AUTHOR = {Bruin, Nils and Sijsling, Jeroen and Zotine, Alexandre},
  PUBLISHER = {Mathematical Sciences Publishers},
  DATE = {2019-01},
  DOI = {10.2140/obs.2019.2.155},
  ISSN = {2329-907X},
  JOURNALTITLE = {The Open Book Series},
  NUMBER = {1},
  PAGES = {155--171},
  TITLE = {Numerical Computation of Endomorphism Rings of {{Jacobians}}},
  URLDATE = {2024-03-20},
  VOLUME = {2},
}

@MISC{Bryan_2019,
  AUTHOR = {Bryan, Jim},
  PUBLISHER = {arXiv},
  DATE = {2019-02},
  DOI = {10.48550/arXiv.1902.08695},
  EPRINT = {1902.08695},
  EPRINTCLASS = {math},
  EPRINTTYPE = {arXiv},
  NUMBER = {arXiv:1902.08695},
  TITLE = {The {{Donaldson-Thomas}} Partition Function of the Banana Manifold},
  URLDATE = {2025-05-06},
}

@ARTICLE{CandelasEtAl_2020,
  AUTHOR = {Candelas, Philip and {de la Ossa}, Xenia and Elmi, Mohamed and {van Straten}, Duco},
  DATE = {2020},
  DOI = {10.1007/JHEP10(2020)202},
  ISSN = {1126-6708},
  JOURNALTITLE = {Journal of High Energy Physics},
  LANGID = {english},
  NUMBER = {10},
  PAGES = {74},
  TITLE = {A One Parameter Family of {{Calabi-Yau}} Manifolds with Attractor Points of Rank Two},
  VOLUME = {2020},
}

@ARTICLE{CandelasEtAl_1991,
  AUTHOR = {Candelas, Philip and De La Ossa, Xenia C. and Green, Paul S. and Parkes, Linda},
  DATE = {1991-04},
  DOI = {10.1016/0370-2693(91)91218-K},
  ISSN = {03702693},
  JOURNALTITLE = {Physics Letters B},
  LANGID = {english},
  NUMBER = {1-2},
  PAGES = {118--126},
  TITLE = {An Exactly Soluble Superconformal Theory from a Mirror Pair of {{Calabi-Yau}} Manifolds},
  URLDATE = {2023-11-24},
  VOLUME = {258},
}

@ARTICLE{CandelasEtAl_1991a,
  AUTHOR = {Candelas, Philip and De La Ossa, Xenia C. and Green, Paul S. and Parkes, Linda},
  DATE = {1991-07},
  DOI = {10.1016/0550-3213(91)90292-6},
  ISSN = {0550-3213},
  JOURNALTITLE = {Nuclear Physics B},
  NUMBER = {1},
  PAGES = {21--74},
  TITLE = {A Pair of {{Calabi-Yau}} Manifolds as an Exactly Soluble Superconformal Theory},
  URLDATE = {2025-05-08},
  VOLUME = {359},
}

@INBOOK{ClingherEtAl_2016,
  AUTHOR = {Clingher, Adrian and Doran, Charles F. and Lewis, Jacob and Novoseltsev, Andrey Y. and Thompson, Alan},
  EDITOR = {Kerr, Matt and Pearlstein, GregoryEditors},
  LOCATION = {Cambridge},
  PUBLISHER = {Cambridge University Press},
  BOOKTITLE = {Recent Advances in Hodge Theory: {{Period}} Domains, Algebraic Cycles, and Arithmetic},
  DATE = {2016},
  PAGES = {165--228},
  SERIES = {London Mathematical Society Lecture Note Series},
  TITLE = {The 14th Case {{VHS}} via {{K3}} Fibrations},
}

@BOOK{CoxKatz_1999,
  AUTHOR = {Cox, David and Katz, Sheldon},
  LOCATION = {Providence, Rhode Island},
  PUBLISHER = {American Mathematical Society},
  DATE = {1999-03},
  DOI = {10.1090/surv/068},
  ISBN = {978-0-8218-2127-5 978-1-4704-1295-1},
  LANGID = {english},
  SERIES = {Mathematical {{Surveys}} and {{Monographs}}},
  TITLE = {Mirror {{Symmetry}} and {{Algebraic Geometry}}},
  URLDATE = {2024-03-20},
  VOLUME = {68},
}

@ARTICLE{CynkVanStraten_2019,
  AUTHOR = {Cynk, S{ł}awomir and {van Straten}, Duco},
  PUBLISHER = {Instytut Matematyczny Polskiej Akademii Nauk},
  DATE = {2019},
  DOI = {10.4064/ap180608-23-10},
  ISSN = {0066-2216, 1730-6272},
  JOURNALTITLE = {Annales Polonici Mathematici},
  LANGID = {polish},
  PAGES = {243--258},
  TITLE = {{Periods of rigid double octic Calabi--Yau threefolds}},
  URLDATE = {2024-03-28},
  VOLUME = {123},
}

@ARTICLE{DeconinckVanHoeij_2001,
  AUTHOR = {Deconinck, Bernard and {van Hoeij}, Mark},
  DATE = {2001-05},
  DOI = {10.1016/S0167-2789(01)00156-7},
  ISSN = {0167-2789},
  JOURNALTITLE = {Physica D: Nonlinear Phenomena},
  PAGES = {28--46},
  SERIES = {Advances in {{Nonlinear Mathematics}} and {{Science}}: {{A Special Issue}} to {{Honor Vladimir Zakharov}}},
  TITLE = {Computing {{Riemann}} Matrices of Algebraic Curves},
  URLDATE = {2024-03-28},
  VOLUME = {152--153},
}

@MISC{Donlagic_2025,
  AUTHOR = {{Đ}onlagić, Azur},
  PUBLISHER = {arXiv},
  DATE = {2025-04},
  DOI = {10.48550/arXiv.2504.09383},
  EPRINT = {2504.09383},
  EPRINTCLASS = {math},
  EPRINTTYPE = {arXiv},
  NUMBER = {arXiv:2504.09383},
  TITLE = {Numerical {{Calculation}} of {{Periods}} on {{Schoen}}'s {{Class}} of {{Calabi-Yau Threefolds}}},
  URLDATE = {2025-05-06},
}

@ARTICLE{Elmi_2024,
  AUTHOR = {Elmi, Mohamed},
  DATE = {2024-07},
  DOI = {10.1007/JHEP07(2024)076},
  ISSN = {1029-8479},
  JOURNALTITLE = {Journal of High Energy Physics},
  LANGID = {english},
  NUMBER = {7},
  PAGES = {76},
  TITLE = {Hadamard Products and {{BPS}} Networks},
  URLDATE = {2025-07-08},
  VOLUME = {2024},
}

@MISC{ElsenhansJahnel_2022,
  AUTHOR = {Elsenhans, Andreas-Stephan and Jahnel, Jörg},
  DATE = {2022},
  TITLE = {Real and Complex Multiplication on {{K3}} Surfaces via Period Integration},
}

@INBOOK{Esole_2017,
  AUTHOR = {Esole, Mboyo},
  EDITOR = {Cardona, Alexander and Morales, Pedro and Ocampo, Hernán and Paycha, Sylvie and Reyes Lega, Andrés F.},
  LOCATION = {Cham},
  PUBLISHER = {Springer International Publishing},
  BOOKTITLE = {Quantization, {{Geometry}} and {{Noncommutative Structures}} in {{Mathematics}} and {{Physics}}},
  DATE = {2017},
  DOI = {10.1007/978-3-319-65427-0_7},
  ISBN = {978-3-319-65427-0},
  LANGID = {english},
  PAGES = {247--276},
  TITLE = {Introduction to {{Elliptic Fibrations}}},
  URLDATE = {2024-03-20},
}

@ARTICLE{Frobenius_1873,
  AUTHOR = {Frobenius, G.},
  PUBLISHER = {De Gruyter},
  CHAPTER = {Journal für die reine und angewandte Mathematik},
  DATE = {1873-01},
  DOI = {10.1515/crll.1873.76.214},
  ISSN = {1435-5345},
  LANGID = {ngerman},
  NUMBER = {76},
  PAGES = {214--235},
  TITLE = {{Über die Integration der linearen Differentialgleichungen durch Reihen.}},
  URLDATE = {2024-03-20},
  VOLUME = {1873},
}

@ARTICLE{GolyshevVanStraten_2023,
  AUTHOR = {Golyshev, Vasily and {van Straten}, Duco},
  DATE = {2023},
  DOI = {10.4310/PAMQ.2023.v19.n1.a9},
  ISSN = {15588599, 15588602},
  JOURNALTITLE = {Pure and Applied Mathematics Quarterly},
  LANGID = {english},
  NUMBER = {1},
  PAGES = {233--265},
  TITLE = {Congruences via Fibered Motives},
  URLDATE = {2024-03-21},
  VOLUME = {19},
}

@ARTICLE{Griffiths_1969,
  AUTHOR = {Griffiths, Philip A.},
  PUBLISHER = {Annals of Mathematics},
  DATE = {1969},
  DOI = {10.2307/1970746},
  EPRINT = {1970746},
  EPRINTTYPE = {jstor},
  ISSN = {0003-486X},
  JOURNALTITLE = {Annals of Mathematics},
  NUMBER = {3},
  PAGES = {460--495},
  SHORTTITLE = {On the {{Periods}} of {{Certain Rational Integrals}}},
  TITLE = {On the {{Periods}} of {{Certain Rational Integrals}}: {{I}}},
  URLDATE = {2024-03-20},
  VOLUME = {90},
}

@ARTICLE{Grothendieck_1966,
  AUTHOR = {Grothendieck, A.},
  DATE = {1966-01},
  DOI = {10.1007/BF02684807},
  ISSN = {1618-1913},
  JOURNALTITLE = {Publications Mathématiques de l'Institut des Hautes Études Scientifiques},
  LANGID = {english},
  NUMBER = {1},
  PAGES = {95--103},
  TITLE = {On the de {{Rham}} Cohomology of Algebraic Varieties},
  URLDATE = {2024-03-20},
  VOLUME = {29},
}

@ARTICLE{HalversonEtAl_2015,
  AUTHOR = {Halverson, James and Jockers, Hans and Lapan, Joshua M. and Morrison, David R.},
  DATE = {2015-02},
  DOI = {10.1007/s00220-014-2157-z},
  ISSN = {1432-0916},
  JOURNALTITLE = {Communications in Mathematical Physics},
  LANGID = {english},
  NUMBER = {3},
  PAGES = {1563--1584},
  TITLE = {Perturbative {{Corrections}} to {{Kähler Moduli Spaces}}},
  URLDATE = {2025-05-08},
  VOLUME = {333},
}

@ARTICLE{Herfurtner_1991,
  AUTHOR = {Herfurtner, Stephan},
  DATE = {1991},
  DOI = {10.1007/BF01445211},
  ISSN = {0025-5831},
  JOURNALTITLE = {Mathematische Annalen},
  LANGID = {english},
  NUMBER = {2},
  PAGES = {319--342},
  TITLE = {Elliptic Surfaces with Four Singular Fibres},
  VOLUME = {291},
}

@ARTICLE{Iritani_2009,
  AUTHOR = {Iritani, Hiroshi},
  DATE = {2009-10},
  DOI = {10.1016/j.aim.2009.05.016},
  ISSN = {0001-8708},
  JOURNALTITLE = {Advances in Mathematics},
  NUMBER = {3},
  PAGES = {1016--1079},
  TITLE = {An Integral Structure in Quantum Cohomology and Mirror Symmetry for Toric Orbifolds},
  URLDATE = {2025-05-08},
  VOLUME = {222},
}

@ARTICLE{KapustkaKapustka_2009,
  AUTHOR = {Kapustka, Grzegorz and Kapustka, Micha{ł}},
  DATE = {2009},
  DOI = {10.1142/S0129167X09005339},
  ISSN = {0129-167X},
  JOURNALTITLE = {International Journal of Mathematics},
  LANGID = {english},
  NUMBER = {4},
  PAGES = {401--426},
  TITLE = {Fiber Products of Elliptic Surfaces with Section and Associated {{Kummer}} Fibrations},
  VOLUME = {20},
}

@ARTICLE{KatzEtAl_2024,
  AUTHOR = {Katz, Sheldon and Klemm, Albrecht and Schimannek, Thorsten and Sharpe, Eric},
  DATE = {2024-02},
  DOI = {10.1007/s00220-023-04896-2},
  ISSN = {1432-0916},
  JOURNALTITLE = {Communications in Mathematical Physics},
  LANGID = {english},
  NUMBER = {3},
  PAGES = {62},
  TITLE = {Topological {{Strings}} on {{Non-commutative Resolutions}}},
  URLDATE = {2025-05-08},
  VOLUME = {405},
}

@MISC{KatzSchimannek_2023,
  AUTHOR = {Katz, Sheldon and Schimannek, Thorsten},
  PUBLISHER = {arXiv},
  DATE = {2023-06},
  DOI = {10.48550/arXiv.2307.00047},
  EPRINT = {2307.00047},
  EPRINTCLASS = {hep-th},
  EPRINTTYPE = {arXiv},
  NUMBER = {arXiv:2307.00047},
  TITLE = {New Non-Commutative Resolutions of Determinantal {{Calabi-Yau}} Threefolds from Hybrid {{GLSM}}},
  URLDATE = {2025-05-08},
}

@INBOOK{KatzarkovEtAl_2008,
  AUTHOR = {Katzarkov, L. and Kontsevich, M. and Pantev, T.},
  EDITOR = {Donagi, Ron and Wendland, Katrin},
  LOCATION = {Providence, Rhode Island},
  PUBLISHER = {American Mathematical Society},
  BOOKTITLE = {Proceedings of {{Symposia}} in {{Pure Mathematics}}},
  DATE = {2008},
  DOI = {10.1090/pspum/078/2483750},
  ISBN = {978-0-8218-4430-4 978-0-8218-9385-2},
  LANGID = {english},
  PAGES = {87--174},
  TITLE = {Hodge Theoretic Aspects of Mirror Symmetry},
  URLDATE = {2025-05-08},
  VOLUME = {78},
}

@INBOOK{KauersEtAl_2015,
  AUTHOR = {Kauers, Manuel and Jaroschek, Maximilian and Johansson, Fredrik},
  EDITOR = {Gutierrez, Jaime and Schicho, Josef and Weimann, Martin},
  LOCATION = {Cham},
  PUBLISHER = {Springer International Publishing},
  BOOKTITLE = {Computer {{Algebra}} and {{Polynomials}}: {{Applications}} of {{Algebra}} and {{Number Theory}}},
  DATE = {2015},
  DOI = {10.1007/978-3-319-15081-9_6},
  ISBN = {978-3-319-15081-9},
  LANGID = {english},
  PAGES = {105--125},
  TITLE = {Ore {{Polynomials}} in {{Sage}}},
  URLDATE = {2024-03-20},
}

@MISC{KnappMcGovern_2025,
  AUTHOR = {Knapp, Johanna and McGovern, Joseph},
  PUBLISHER = {arXiv},
  DATE = {2025-04},
  DOI = {10.48550/arXiv.2504.06147},
  EPRINT = {2504.06147},
  EPRINTCLASS = {hep-th},
  EPRINTTYPE = {arXiv},
  NUMBER = {arXiv:2504.06147},
  TITLE = {Noncommutative Resolutions and {{CICY}} Quotients from a Non-Abelian {{GLSM}}},
  URLDATE = {2025-07-08},
}

@ARTICLE{Kodaira_1963,
  AUTHOR = {Kodaira, Kunihiko},
  DATE = {1963},
  DOI = {10.2307/1970500},
  ISSN = {0003-486X},
  JOURNALTITLE = {Annals of Mathematics. Second Series},
  LANGID = {english},
  PAGES = {1--40},
  TITLE = {On Compact Analytic Surfaces. {{III}}},
  VOLUME = {78},
}

@ARTICLE{Lairez_2016,
  AUTHOR = {Lairez, Pierre},
  DATE = {2016-07},
  DOI = {10.1090/mcom/3054},
  ISSN = {0025-5718, 1088-6842},
  JOURNALTITLE = {Mathematics of Computation},
  LANGID = {english},
  NUMBER = {300},
  PAGES = {1719--1752},
  TITLE = {Computing Periods of Rational Integrals},
  URLDATE = {2025-04-20},
  VOLUME = {85},
}

@ARTICLE{LairezEtAl_2024,
  AUTHOR = {Lairez, Pierre and {Pichon-Pharabod}, Eric and Vanhove, Pierre},
  DATE = {2024-11},
  DOI = {10.1090/mcom/3947},
  ISSN = {0025-5718, 1088-6842},
  JOURNALTITLE = {Mathematics of Computation},
  LANGID = {english},
  NUMBER = {350},
  PAGES = {2985--3025},
  TITLE = {Effective Homology and Periods of Complex Projective Hypersurfaces},
  URLDATE = {2024-12-05},
  VOLUME = {93},
}

@ARTICLE{Lamotke_1981,
  AUTHOR = {Lamotke, Klaus},
  DATE = {1981-01},
  DOI = {10.1016/0040-9383(81)90013-6},
  ISSN = {0040-9383},
  JOURNALTITLE = {Topology},
  NUMBER = {1},
  PAGES = {15--51},
  TITLE = {The Topology of Complex Projective Varieties after {{S}}. {{Lefschetz}}},
  URLDATE = {2023-11-24},
  VOLUME = {20},
}

@ARTICLE{LenstraEtAl_1982,
  AUTHOR = {Lenstra, A. K. and Lenstra, H. W. and Lovász, L.},
  DATE = {1982-12},
  DOI = {10.1007/BF01457454},
  ISSN = {1432-1807},
  JOURNALTITLE = {Mathematische Annalen},
  LANGID = {english},
  NUMBER = {4},
  PAGES = {515--534},
  TITLE = {Factoring Polynomials with Rational Coefficients},
  URLDATE = {2024-05-23},
  VOLUME = {261},
}

@ARTICLE{Libgober_1999,
  AUTHOR = {Libgober, Anatoly},
  DATE = {1999},
  DOI = {10.4310/MRL.1999.v6.n2.a2},
  ISSN = {10732780, 1945001X},
  JOURNALTITLE = {Mathematical Research Letters},
  LANGID = {english},
  NUMBER = {2},
  PAGES = {141--149},
  TITLE = {Chern Classes and the Periods of Mirrors},
  URLDATE = {2025-05-08},
  VOLUME = {6},
}

@ARTICLE{Mezzarobba_2010,
  AUTHOR = {Mezzarobba, Marc},
  LOCATION = {New York, NY, USA},
  PUBLISHER = {Association for Computing Machinery},
  DATE = {2010-07},
  DOI = {10.1145/1837934.1837965},
  ISBN = {978-1-4503-0150-3},
  JOURNALTITLE = {Proceedings of the 2010 {{International Symposium}} on {{Symbolic}} and {{Algebraic Computation}}},
  PAGES = {139--145},
  SERIES = {{{ISSAC}} '10},
  SHORTTITLE = {{{NumGfun}}},
  TITLE = {{{NumGfun}}: A Package for Numerical and Analytic Computation with {{D-finite}} Functions},
  URLDATE = {2024-03-20},
}

@MISC{Mezzarobba_2016,
  AUTHOR = {Mezzarobba, Marc},
  PUBLISHER = {arXiv},
  DATE = {2016-07},
  DOI = {10.48550/arXiv.1607.01967},
  EPRINT = {1607.01967},
  EPRINTCLASS = {cs},
  EPRINTTYPE = {arXiv},
  NUMBER = {arXiv:1607.01967},
  TITLE = {Rigorous {{Multiple-Precision Evaluation}} of {{D-Finite Functions}} in {{SageMath}}},
  URLDATE = {2024-03-20},
}

@BOOK{Miranda_1989,
  AUTHOR = {Miranda, Rick},
  LOCATION = {Pisa},
  PUBLISHER = {ETS Editrice},
  DATE = {1989},
  LANGID = {english},
  TITLE = {The Basic Theory of Elliptic Surfaces. {{Notes}} of Lectures},
}

@ARTICLE{MirandaPersson_1986,
  AUTHOR = {Miranda, Rick and Persson, Ulf},
  DATE = {1986-12},
  DOI = {10.1007/BF01160474},
  ISSN = {0025-5874, 1432-1823},
  JOURNALTITLE = {Mathematische Zeitschrift},
  LANGID = {english},
  NUMBER = {4},
  PAGES = {537--558},
  TITLE = {On Extremal Rational Elliptic Surfaces},
  URLDATE = {2025-05-11},
  VOLUME = {193},
}

@BOOK{Moishezon_1977,
  AUTHOR = {Moishezon, Boris},
  LOCATION = {Berlin, Heidelberg},
  PUBLISHER = {Springer},
  DATE = {1977},
  DOI = {10.1007/BFb0063355},
  ISBN = {978-3-540-08355-9 978-3-540-37276-9},
  SERIES = {Lecture {{Notes}} in {{Mathematics}}},
  TITLE = {Complex {{Surfaces}} and {{Connected Sums}} of {{Complex Projective Planes}}},
  URLDATE = {2024-03-20},
  VOLUME = {603},
}

@ARTICLE{MolinNeurohr_2019,
  AUTHOR = {Molin, Pascal and Neurohr, Christian},
  DATE = {2019-03},
  DOI = {10.1090/mcom/3351},
  ISSN = {0025-5718, 1088-6842},
  JOURNALTITLE = {Mathematics of Computation},
  LANGID = {english},
  NUMBER = {316},
  PAGES = {847--888},
  TITLE = {Computing Period Matrices and the {{Abel-Jacobi}} Map of Superelliptic Curves},
  URLDATE = {2024-03-28},
  VOLUME = {88},
}

@THESIS{Neurohr_2018,
  AUTHOR = {Neurohr, Christian},
  DATE = {2018-06},
  TITLE = {Efficient Integration on {{Riemann}} Surfaces \& Applications},
  TYPE = {phdthesis},
}

@ARTICLE{Pichon-Pharabod_2025,
  AUTHOR = {{Pichon-Pharabod}, Eric},
  DATE = {2025-03},
  DOI = {10.1016/j.jsc.2024.102357},
  ISSN = {0747-7171},
  JOURNALTITLE = {Journal of Symbolic Computation},
  PAGES = {102357},
  TITLE = {A Semi-Numerical Algorithm for the Homology Lattice and Periods of Complex Elliptic Surfaces over the Projective Line},
  URLDATE = {2024-12-05},
  VOLUME = {127},
}

@MISC{Schimannek_2025,
  AUTHOR = {Schimannek, Thorsten},
  PUBLISHER = {arXiv},
  DATE = {2025-04},
  DOI = {10.48550/arXiv.2504.06115},
  EPRINT = {2504.06115},
  EPRINTCLASS = {hep-th},
  EPRINTTYPE = {arXiv},
  LANGID = {english},
  NUMBER = {arXiv:2504.06115},
  TITLE = {In Search of Almost Generic {{Calabi-Yau}} 3-Folds},
  URLDATE = {2025-05-08},
}

@ARTICLE{Schoen_1988,
  AUTHOR = {Schoen, Chad},
  DATE = {1988-06},
  DOI = {10.1007/BF01215188},
  ISSN = {1432-1823},
  JOURNALTITLE = {Mathematische Zeitschrift},
  LANGID = {english},
  NUMBER = {2},
  PAGES = {177--199},
  TITLE = {On Fiber Products of Rational Elliptic Surfaces with Section},
  URLDATE = {2023-11-24},
  VOLUME = {197},
}

@INBOOK{SchuttShioda_2010,
  AUTHOR = {Schütt, Matthias and Shioda, Tetsuji},
  LOCATION = {Tokyo},
  PUBLISHER = {Mathematical Society of Japan},
  BOOKTITLE = {Algebraic Geometry in {{East Asia}} -- {{Seoul}} 2008. {{Proceedings}} of the 3rd International Conference ``{{Algebraic}} Geometry in {{East Asia}}, {{III}}'', {{Seoul}}, {{Korea}}, {{November}} 11--15, 2008},
  DATE = {2010},
  ISBN = {978-4-931469-63-1},
  LANGID = {english},
  PAGES = {51--160},
  TITLE = {Elliptic Surfaces},
}

@ARTICLE{Sertoz_2019,
  AUTHOR = {Sertöz, Emre Can},
  DATE = {2019-04},
  DOI = {10.1090/mcom/3430},
  EPRINT = {1803.08068},
  EPRINTCLASS = {cs, math},
  EPRINTTYPE = {arXiv},
  ISSN = {0025-5718, 1088-6842},
  JOURNALTITLE = {Mathematics of Computation},
  LANGID = {english},
  NUMBER = {320},
  PAGES = {2987--3022},
  TITLE = {Computing {{Periods}} of {{Hypersurfaces}}},
  URLDATE = {2023-11-24},
  VOLUME = {88},
}

@ARTICLE{Stiller_1987,
  AUTHOR = {Stiller, Peter},
  DATE = {1987-05},
  DOI = {10.2140/pjm.1987.128.157},
  ISSN = {0030-8730, 0030-8730},
  JOURNALTITLE = {Pacific Journal of Mathematics},
  LANGID = {english},
  NUMBER = {1},
  PAGES = {157--189},
  TITLE = {The {{Picard}} Numbers of Elliptic Surfaces with Many Symmetries},
  URLDATE = {2023-11-24},
  VOLUME = {128},
}

@MISC{Swierczewski_2017,
  AUTHOR = {Swierczewski, Chris},
  DATE = {2017},
  SHORTTITLE = {Abelfuctions},
  TITLE = {Abelfunctions: {{A}} Library for Computing with {{Abelian}} Functions, {{Riemann}} Surfaces, and Algebraic Curves},
}

@MISC{sagemath,
  AUTHOR = {{The Sage Developers}},
  DATE = {2023},
  DOI = {10/gr8dhc},
  TITLE = {{{SageMath}}, the {{Sage Mathematics Software}}},
}

@ARTICLE{VanDerHoeven_1999,
  AUTHOR = {{van der Hoeven}, Joris},
  DATE = {1999-01},
  DOI = {10.1016/S0304-3975(98)00102-9},
  ISSN = {0304-3975},
  JOURNALTITLE = {Theoretical Computer Science},
  NUMBER = {1},
  PAGES = {199--215},
  SERIES = {Real {{Numbers}} and {{Computers}}},
  TITLE = {Fast Evaluation of Holonomic Functions},
  URLDATE = {2024-03-20},
  VOLUME = {210},
}

@INBOOK{VanStraten_2018,
  AUTHOR = {{van Straten}, Duco},
  LOCATION = {Somerville, MA},
  PUBLISHER = {International Press; Beijing: Higher Education Press},
  BOOKTITLE = {Uniformization, {{Riemann-Hilbert}} Correspondence, {{Calabi-Yau}} Manifolds and {{Picard-Fuchs}} Equations. {{Based}} on the Conference, {{Institute Mittag-Leffler}}, {{Stockholm}}, {{Sweden}}, {{July}} 13--18, 2015},
  DATE = {2018},
  ISBN = {978-1-57146-363-0},
  LANGID = {english},
  PAGES = {401--451},
  TITLE = {Calabi-{{Yau}} Operators},
}

\end{document}